\documentclass[11pt,reqno]{amsart}

\usepackage[utf8]{inputenc}
\usepackage[margin=1.25in]{geometry}
\usepackage{hyperref}
\usepackage{appendix}
\usepackage{amsfonts}
\usepackage{amsthm}
\usepackage{amssymb}
\usepackage{stmaryrd} 
\usepackage{amsmath}
\usepackage{amsthm}
\usepackage{mathrsfs}
\usepackage{lipsum} 

\theoremstyle{plain}
\newtheorem{theorem}{Theorem}[section]
\newtheorem{corollary}[theorem]{Corollary}
\newtheorem{lemma}[theorem]{Lemma}
\newtheorem{Proposition}[theorem]{Proposition}

\newtheorem{Example}[theorem]{Example}
\newtheorem{Definition}[theorem]{Definition}

\newtheorem{Fact}[theorem]{Fact}
\newtheorem{Conjecture}[theorem]{Conjecture}
\newtheorem*{Induction hypothesis}{Induction hypothesis}

\theoremstyle{remark}

\newtheorem{remark}[theorem]{Remark}

\usepackage{biblatex} 
\numberwithin{equation}{section}
\title[Multiple singular values]{Small ball probability for multiple singular values of symmetric random matrices}

\author{Yi HAN}
\address{Department of Pure Mathematics and Mathematical Statistics, University of Cambridge.
}
\email{yh482@cam.ac.uk}
\thanks{Supported by EPSRC grant EP/W524141/1.}

\begin{document}

\begin{abstract} Let $A_n$ be an $n\times n$ random symmetric matrix with $(A_{ij})_{i< j}$ i.i.d. mean $0$, variance 1, following a subGaussian distribution and diagonal elements i.i.d. following a subGaussian distribution with a fixed variance. We investigate the joint small ball probability that $A_n$ has eigenvalues near two fixed locations $\lambda_1$ and $\lambda_2$, where $\lambda_1$ and $\lambda_2$ are sufficiently separated and in the bulk of the semicircle law. More precisely we prove that for a wide class of entry distributions of $A_{ij}$ that involve all Gaussian convolutions (where $\sigma_{min}(\cdot)$ denotes the least singular value of a square matrix), $$\mathbb{P}(\sigma_{min}(A_n-\lambda_1 I_n)\leq\delta_1n^{-1/2},\sigma_{min}(A_n-\lambda_2 I_n)\leq\delta_2n^{-1/2})\leq c\delta_1\delta_2+e^{-cn}.$$ The given estimate approximately factorizes as the product of the estimates for the two individual events, which is an indication of quantitative independence. The estimate readily generalizes to $d$ distinct locations. As an application, we upper bound the probability that there exist $d$ eigenvalues of $A_n$ asymptotically satisfying any fixed linear equation, which in particular gives a lower bound of the distance to this linear relation from any possible eigenvalue pair that holds with probability $1-o(1)$, and rules out the existence of two equal singular values in generic regions of the spectrum.
\end{abstract}

\maketitle

\section{Introduction}

Let $A_n$ be an $n\times n$ symmetric random matrix with entries above the diagonal $(A_{i,j})_{i< j}$ being identically distributed and having mean $0$ and variance $1$, and diagonal entries having a fixed variance.  Wigner proved in his seminal work \cite{wigner1958distribution} in the 1950s that the eigenvalue distribution of $A_n$ converges to the semicircle law on the real line.

In this work we investigate the following question: given two real numbers $\lambda_1\neq\lambda_2$ in the bulk of the semicircle law, what can be said about the dependence between the following two events: (a) $A_n$ has an eigenvalue near $\lambda_1$; and (b) $A_n$ has an eigenvalue near $\lambda_2$? As the spectra of $A_n$ at different locations are correlated, one cannot argue from elementary probability theory that the events (a) and (b) are asymptotically independent in some quantitative sense.

This question is related to the study of asymptotic independence in spatial point processes, especially determinantal point process (DPP). The GUE matrix ensemble and complex Ginibre ensemble are both determinantal point processes, and GOE matrices have instead a Pfaffian structure; see for example \cite{anderson2010introduction}. For DPPs, the fact that a positive definite matrix satisfies the Hadamard–Fischer inequality can be used to show that its $k$-fold correlation function approximately factorizes as the product of each individual factor when we take the limit, which is used in \cite{MR3112927} for the study of smallest gap of GUE. This fact is a strong indication of asymptotic independence, though restricted to DPPs and the GUE matrix. From the perspective of universality in random matrix theory, we should expect that the events (a) and (b) are approximately independent, in a certain quantitative sense, for a general class of Wigner matrices $A_n$ without a determinantal point process structure and do not match to fourth moments with the GUE. Another related notion is negative dependence in determinantal point processes, see for example the works \cite{pemantle2000towards}, \cite{lyons2003determinantal}, \cite{ghosh2015determinantal}, \cite{borcea2009negative}. 

Another closely related topic concerns the log correlated structure in Wigner matrices. This log correlated structure can be derived in the large $n$ limit for the log determinant \cite{tao2012central} \cite{MR4017114} and the log characteristic polynomial \cite{bourgade2010mesoscopic} \cite{bourgade2022optimal}. Owing to this structure, we can explicitly compute the correlation of the log characteristic functions and eigenvalue fluctuations at two different locations $\lambda_1,\lambda_2$. The limiting covariance of observable at these two locations is less than one when $\lambda_1$ and $\lambda_2$ are distant at mesoscopic scales, and is zero when they are distant at macroscopic scales, see \cite{mody2023log}. Therefore we should expect that the events (a) and (b) are approximately independent in a very strong sense, modulo numerical constants.

In this work we derive quantitative, non-asymptotic estimates to highlight the asymptotic independence of (a) and (b), for $A_n$ having a wide class of entry distributions (and therefore not assuming a DPP structure and not assuming $A_n$ is invariant under orthogonal or unitary transform). More precisely, we prove that when the entry distribution of $A_{ij}$ is subGaussian and has a small (but non-vanishing) Gaussian component, and when $|\lambda_1-\lambda_2|\geq\Delta\sqrt{n}$ and both lie in the bulk $[-(2-\kappa)\sqrt{n},(2-\kappa)\sqrt{n}]$ for any $\Delta,\kappa>0$, then 
\begin{equation}\label{singularvalue111}
    \mathbb{P}(\sigma_{min}(A_n-\lambda_1 I_n)\leq\delta_1n^{-1/2},\sigma_{min}(A_n-\lambda_2 I_n)\leq\delta_2n^{-1/2})\leq c\delta_1\delta_2+e^{-cn}
\end{equation}
for any $\delta_1,\delta_2>0$, where $c$ is a fixed constant. Here for any square matrix $M$ we use the notation $\sigma_{min}(M)$ to denote the smallest singular value of $M$. We will further prove this result for $d$ locations $\lambda_1,\cdots,\lambda_d$ for any $d$, and a similar result for $\lambda_1,\lambda_2$ with distance $|\lambda_1-\lambda_2|\geq\Delta n^{-\frac{1}{2}+\sigma}$ and any $\sigma\in(0,1]$. Moreover, we do not use the explicit density function of Gaussian distribution (but only use its Fourier decay) so that our assumption covers many other distributions.

The main estimate \eqref{singularvalue111} of this paper is a multilocation generalization of the main result in Campos, Jenssen, Michelen and Sahasrabudhe \cite{campos2024least} which proves that, for any symmetric random matrix $A_n$ with subGaussian entries, we have for any $\epsilon>0$,
\begin{equation}\label{subgaussimone}
    \mathbb{P}(\sigma_{min}(A_n)\leq\epsilon n^{-1/2})\leq c\epsilon+e^{-cn}.
\end{equation} It is not hard to check that the same estimate holds for $\sigma_{min}(A_n-\lambda_iI_n)$ for any $\lambda_i$ in the bulk of the semicircle law, so one can interpret \eqref{singularvalue111} as an approximate independence property thanks to \eqref{subgaussimone}. Estimate \eqref{subgaussimone} was previously proven fir $c=1$ and $\epsilon\geq n^{-d}$ for some sufficiently small $d>0$ via fixed energy universality results and comparison to the GOE matrix, obtaining the correct dependence in $\epsilon$, see \cite{MR3541852} and also \cite{MR2784665} where a moment matching condition was imposed.

To the best of our knowledge, the simultaneous small ball probability estimate for several distinct singular values \eqref{singularvalue111} is quite new in the literature. Similar results do not exist even for random matrices with i.i.d. entries where it is usually simpler, though still very challenging, to derive least singular value estimates; see \cite{rudelson2008littlewood}, \cite{MR4255145}, \cite{MR4076632}.

We briefly elaborate how our estimate \eqref{singularvalue111} compares to the log correlated structure interpretation. In \eqref{singularvalue111} the numerical constants are not explicitly computed; thus the covariance is not computed and we can say only that the two events are approximately independent. The advantage is that we can consider the event where $A_n$ has an eigenvalue in an $\epsilon$-neighborhood of $\lambda_1$ for any $\epsilon>e^{-cn}$, which is much smaller than the typical eigenvalue spacing $n^{-1/2}$. Additionally, our estimate holds with probability $1-e^{-\Omega(n)}$. These two advantages appear to be inaccessible from the log correlated field structure in \cite{bourgade2022optimal}, \cite{mody2023log}. The estimate we prove here mostly relies on the anti-concentration property of random matrix entries rather than comparison to the GOE matrix, and it appears that our results are new for the GOE matrix as well. Allowing $\epsilon$ to be much smaller than $n^{-1/2}$ is of primary importance to our application, as this permits us to take an $\epsilon$ net and deduce that the possibility that some $d$ eigenvalues of $A_n$ satisfy any fixed linear structure is exponentially small (Theorem \ref{theorem1.7}).

\subsection{Statement of main estimates} Now we state the standing assumptions and the main estimates of our paper. Although some conditions listed as follows seem technical, the reader can simply take $A_{ij}$ to be any subGaussian random variable having a nonvanishing Gaussian component; under this assumption all the results of this paper hold.

Recall that for a random variable $\zeta$ with mean 0 and variance one, we define its subGaussian moment via 
$$
\|\zeta\|_{\psi_2}:=\sup_{p\geq 1}p^{-1/2}(\mathbb{E}|\zeta|^p)^\frac{1}{p}. 
$$ 

For any $B>0$, let $\Gamma_B$ denote the set of random variables $\xi$ with mean 0, variance one that satisfies $\|\xi\|_{\psi_2}\leq B$.

\begin{Definition}\label{defininition1.1} Fix a parameter $\sigma_0\in(0,1)$ throughout the paper, and fix three real numbers $G>0,$ $K>0$, $B>0$.

Let $\Gamma_1(G,B,\sigma_0)$
    denote the collection of mean zero real random variables $\zeta_1$ with variance $\sigma$, which also satisfies a subGaussian estimate $\|\sigma_0^{-1/2}\zeta_1\|_{\psi_2}\leq G$ and also satisfies the following Fourier decay estimate: 
    \begin{equation}\label{assumptionone}
       | \mathbb{E}e^{2\pi i\theta \xi_1}|\leq \exp(-B|\theta|^2),\quad\text{ for all } \theta\in\mathbb{R}.    \end{equation}

       Let $\Gamma_2(G,K,\sigma_0)$ denote the collection of mean zero real random variables $\zeta_2$ with variance $1-\sigma_0$, which also satisfies sub-Gaussian estimate $\|(1-\sigma_0)^{-1/2}\zeta_2\|_{\psi_2}\leq G$, and have a bounded density function $\mu_{\zeta_2}(x)$ with respect to the Lebesgue measure that satisfies 
       \begin{equation} \label{assumptiontwo}|\mu_{\zeta_2}(x)|\leq K,\quad \text{ for all }x\in\mathbb{R}.       \end{equation}
    
\end{Definition}
The main result of this paper is the following estimate:
\begin{theorem}\label{Theorem1.1} Fix some positive constants $G,K,B$ and $\sigma_0\in(0,1)$. Consider two random variables $\zeta_1\in \Gamma_1(G,B,\sigma_0)$ and $\zeta_2\in (G,K,\sigma_0)$, such that $\zeta_1$ and $\zeta_2$ are independent.  Moreover, assume that

$(\mathbf{LSI})$ The random variable $\zeta_1+\zeta_2$ satisfies a logarithmic Sobolev inequality with the log Sobolev constant finite and independent of $n$.

Let $A_n$ be an $n\times n$ symmetric random matrix with upper-diagonal $\{A_{i,j}\}_{1\leq i< j\leq n}$ being independent and identically distributed, having distribution $\zeta_1+\zeta_2$. 

Additionally, assume that the diagonal entries $(A_{ii})_{1\leq i\leq n}$ are i.i.d. having distribution $\bar{\sigma}(\zeta_1+\zeta_2)$ for some $\bar{\sigma_0}>0$, and independent of $\{A_{i,j}\}_{1\leq i< j\leq n}$.

Consider any fixed $\kappa>0$, $\Delta>0$ and $d\in\mathbb{N}_+$, and consider any fixed locations $\lambda_1,\cdots,\lambda_d\in[(-(2-\kappa)\sqrt{n},(2-\kappa)\sqrt{n}]$ such that $|\lambda_i-\lambda_j|\geq \Delta n^{1/2}$ for any $i\neq j\in\{1,\cdots,d\}$. Then there exists $c>0$ depending on $G,K,B,\sigma_0,\bar{\sigma}_0,\Delta,\kappa$ such that for any $\delta_1,\cdots,\delta_d\geq 0$,
\begin{equation}\label{gasagagga}
    \mathbb{P}\left(\sigma_{min}(A_n-\lambda_i I_n)\leq\delta_i n^{-1/2},\quad\text{ for all }i=1,\cdots,d\right)\leq c\prod_{i=1}^d\delta_i+e^{-cn}. 
\end{equation}

\end{theorem}

Informally, this theorem implies that the eigenvalues of $A_n$ satisfy strong asymptotic independence properties and is, to some extent, similar to a Poisson point process.

\begin{Example}(Gaussian divisible ensemble) A class of symmetric random matrices that satisfy the assumptions in Theorem \ref{Theorem1.1} are matrices with a Gaussian component, including in particular the GOE matrix. That is, we take $\zeta_1$ to be a Gaussian random variable with mean $0$ and variance $\sigma_0$, and $\zeta_2$ any mean $0$ subGaussian random variable with variance $1-\sigma_0$. By properties of Gaussian distributions, $\zeta_1+\zeta_2$ has the same distribution as the independent sum of a $N(0,\sigma_0/2)$ random variable and another mean $0$ random variable with bounded density and variance $1-\sigma_0/2$. The $\mathbf{(LSI)}$ condition is also satisfied: indeed, by \cite{chen2021dimension}, Theorem 2, the convolution of a subGaussian random variable with a Gaussian variable always satisfies the log Sobolev inequality. Thus Theorem \ref{Theorem1.1} covers all subGaussian random matrices with a small Gaussian perturbation. To the best of our knowledge, Theorem \ref{Theorem1.1} is new for the GOE matrix as well. 
\end{Example}

\begin{remark}
    Theorem \ref{Theorem1.1} can be generalized without much difficulty to generalized Wigner matrices (see \cite{erdHos2011universality}) where each entry $A_{ij}$ has different variance $\frac{\sigma_{ij}}{n}$, and $C_1\leq \sigma_{ij}\leq C_2$ for some $C_1,C_2>0$. This is because we assume a continuous distribution, and the place where this assumption is used does not strictly require i.i.d. entries. For discrete entries, the use of Littlewood-Offord type arguments generally requires i.i.d. entries, but see \cite{MR4255145} for a non-i.i.d. extension.
\end{remark}

    We also formulate a variant of Theorem \ref{Theorem1.1} that considers locations $\lambda_i$'s that are distant only on mesoscopic scales.

    \begin{theorem}\label{Theorem1.2}(Mesoscopic distance) Fix some positive constants $G,K,B$ and $\sigma_0\in(0,1)$. Consider two random variables $\zeta_1\in \Gamma_1(G,B,\sigma_0)$ and $\zeta_2\in (G,K,\sigma_0)$, such that $\zeta_1$ and $\zeta_2$ are independent. 

Let $A_n$ be an $n\times n$ symmetric random matrix with $\{A_{i,j}\}_{1\leq i< j\leq n}$ being independent and identically distributed, with the distribution $\zeta_1+\zeta_2$. 

Additionally assume that the diagonal entries $(A_{ii})_{1\leq i\leq n}$ are i.i.d. with the distribution $\bar{\sigma_0}(\zeta_1+\zeta_2)$ for some $\bar{\sigma_0}>0$, and independent of $\{A_{i,j}\}_{1\leq i< j\leq n}$.

Fix some $\sigma\in(0,1)$. Consider any fixed $\kappa>0$, $\Delta>0$ and $d\in\mathbb{N}_+$, and consider any fixed locations $\lambda_1,\cdots,\lambda_d\in[(-(2-\kappa)\sqrt{n},(2-\kappa)\sqrt{n}]$ such that $|\lambda_i-\lambda_j|\geq \Delta n^{\sigma-1/2}$ for any $i\neq j\in\{1,\cdots,d\}$. Then there exists $c>0$ depending on $G,K,B,\sigma_0,\bar{\sigma}_0,\sigma,\Delta,\kappa$ such that for any $\delta_1,\cdots,\delta_d\geq 0$,
\begin{equation}
    \mathbb{P}\left(\sigma_{min}(A_n-\lambda_i I_n)\leq\delta_i n^{-1/2},\quad\text{ for all }i=1,\cdots,d\right)\leq c\prod_{i=1}^d\delta_i+e^{-cn^{\sigma/2}}.
\end{equation}

\end{theorem}
In light of the log correlated structure of Wigner matrices (see for example \cite{mody2023log}), the distance $|\lambda_i-\lambda_j|\geq n^{\sigma-1/2}$ for any $\sigma>0$ is almost the shortest distance at which we can prove approximate independence without too much difficulty. At scale $n^{-1/2}$ or on some even smaller scales, we believe the bounds in Theorem \ref{Theorem1.1} and \ref{Theorem1.2} still hold true, but this range is not accessible by the techniques presented in this paper.

\subsection{Applications of the main estimate: anticoncentration and equal singular values}We now discuss an application of the main estimate. In Conjecture 8.5 of the survey article \cite{MR4260513}, Vu posed the following question: 
  \begin{Conjecture}\label{conjecture1.6}
  let $M_n^{sym}$ be a symmetric random matrix, whose entry distributions are Rademacher variables. Prove that with probability $1-o(1)$, all the singular values of $M_n^{sym}$ are distinct.
\end{Conjecture}

This conjecture is equivalent to proving that there does not exist any two eigenvalues of $M_n^{sym}$ that sum up to zero, as singular values of $M_n^{sym}$ are the absolute values of eigenvalues of $M_n^{sym}$. To the best of our knowledge, this conjecture remains open up to now.

This conjecture shares some formal similarity to a related problem: proving that $M_n^{sym}$ has distinct eigenvalues with probability $1-o(1)$. Much more progress has been made in solving the latter problem: it was positively solved in \cite{tao2017random}, and then in \cite{campos2024least} the authors provided an exponential bound $1-e^{-cn}$ for subGaussian entries. (The analogous problem for i.i.d. random matrices has also been solved, see \cite{ge2017eigenvalue} and \cite{MR4235467}). The difference between these two problems is as follows: the problem of distinct eigenvalues only requires singular value estimates at one location, whereas Conjecture \eqref{conjecture1.6} requires singular value estimates at different locations.

We prove the following theorem, which considers a much more general problem than that posed in Conjecture \eqref{conjecture1.6}.

\begin{theorem}\label{theorem1.7} Let $A_n$ be an $n\times n$ symmetric random matrix satisfying the assumptions of Theorem \ref{Theorem1.1}. Fix any $\kappa>0$, $d>0$ and $\Delta>0$.

We say that the eigenvalues $x_1,\cdots,x_d$ are \textit{distant} if $|x_i-x_j|\geq\Delta n^{1/2}$ for any $i\neq j$. Let $I_\kappa:=[-(2-\kappa)\sqrt{n},(2-\kappa)\sqrt{n}]$. Then for any $a_1,\cdots,a_d\in\mathbb{R}$ and $c\in\mathbb{R}$, and any $\epsilon>0$,\footnote{In this paper we use the symbol $a\lesssim b$ to mean that $a\leq Cb$ for some constant $C$ that depends only on the density of matrix entries of $A_n$ and other given parameters $d$, $\sigma$ and $\Delta$.}
 \begin{equation} \begin{aligned}   &\mathbb{P}(A_n \text{ has }d \text{ distant eigenvalues } x_1,\cdots,x_d\in I_\kappa \text{ satisfying } |\sum_{i=1}^da_ix_i-c|\leq \epsilon)\\&\lesssim n^{(2d-1)/2}\epsilon+e^{-\Omega(n)}.\end{aligned}\end{equation}
 More generally, for any $\sigma\in(0,1)$, we say that the eigenvalues $x_1,\cdots,x_d$ are $\sigma$-\textit{distant} if $|x_i-x_j|\geq \Delta n^{\sigma-\frac{1}{2}}$ for any $j\neq i$. Then for any $a_1,\cdots,a_d\in\mathbb{R}$ and $c\in\mathbb{R}$, and any $\epsilon>0$,
  \begin{equation} \begin{aligned}   &\mathbb{P}(A_n \text{ has }d\quad 
  \sigma-\text{distant eigenvalues } x_1,\cdots,x_d\in I_\kappa \text{ satisfying } |\sum_{i=1}^da_ix_i-c|\leq \epsilon)\\&\lesssim n^{(2d-1)/2}\epsilon+e^{-\Omega(n^{\sigma/2})}.\end{aligned}\end{equation}
\end{theorem}

This theorem implies that the minimal value attained in the quantity $|\sum_{i=1}^da_ix_i-c|$ where $x_1,\cdots,x_d$ range from distant bulk eigenvalues of $A_n$, is typically at least $n^{-(2d-1)/2}$ with probability $1-o(1)$. In the case $d=2$, this distance is consistent with the fact that the minimal gap in a Poisson distribution on $[0,1]$ with $n$ particles is typically $n^{-2}$.

\begin{proof}[\proofname\ of Theorem \ref{theorem1.7}]

    Consider the first claim. Let $\epsilon>0$ and consider the question of bounding $|\sum_{i=1}^da_ix_i-c|\leq\epsilon$. We find an $\epsilon$ net of $I_\kappa$ and choose any $d-1$ points $p_1,\cdots,p_{d-1}$ from this net. Then there are only finitely many points $p_d$ in this net such that $|\sum_{i=1}^d a_ip_i-c|\leq C\epsilon$ for some very large $C>0$. Then the possibility that $A_n$ has $d$ distant eigenvalues in $I_\kappa$ satisfying the assumption is upper bounded by 
    $$
C(\frac{2\sqrt{n}}{\epsilon})^{d-1}(\epsilon^dn^{d/2}+e^{-\Omega(n)})
    $$ where the first factor is the number of all possible choice of $p_1,\cdots,p_{d-1}$ from the net and the second term comes from $$\mathbb{P}(\sigma_{min}(A_n-p_iI_n)\leq\epsilon,i=1,\cdots,d)\lesssim \epsilon^dn^{d/2}+e^{-cn}.$$ As we only consider distant eigenvalues, we can assume that $|p_i-p_j|\geq\Delta$ for any $i\neq j$. Then applying Theorem \ref{Theorem1.1} completes the proof. For the second claim, it suffices to use Theorem \ref{Theorem1.2} instead. 
\end{proof}

Compared to Conjecture \ref{conjecture1.6}, we have shown that Conjecture \ref{conjecture1.6} holds with possibility $1-e^{-\Omega(n)}$ for singular values in $[\kappa\sqrt{n},(2-\kappa)\sqrt{n}]$ (and any $\kappa>0$) and for any entry distribution specified in Theorem \ref{Theorem1.1}. To the best of our knowledge, Theorem \ref{theorem1.7} is the first time that we can estimate the probability that the eigenvalues of $A_n$ may approximately satisfy any fixed linear or nonlinear structure; moreover this theorem is new even for the GOE matrix, and all prior results consider only eigenvalue gap or eigenvalue overcrowding estimates, see \cite{nguyen2018random}, \cite{nguyen2017random}, \cite{MR4034920} and \cite{MR4416591}, and these works correspond to the special case $x_1-x_2=0$ of Theorem \ref{theorem1.7}. That is, we are now able to estimate the \textit{long-range} dependence between eigenvalues and prove they are asymptotically independent, at the precision of small ball probability bounds. 

Prior to this work, there appear to be very few papers investigating the independence, or anti-concentration property of the spectrum of a quantum system, beyond exactly integrable models such as the GUE matrix. In this work we rigorously justify the anti-concentration and approximate independence properties of the spectrum of a wide family of Wigner matrices. We hope this will initiate a new direction of research.

In Theorems \ref{Theorem1.1}, \ref{Theorem1.2} and \ref{theorem1.7}, we have assumed the locations $\lambda_1,\cdots,\lambda_d$ are separated by some distance $n^{\sigma-1/2}$. We believe that the same estimate should also work if such a distance separation assumption is removed. When some $\lambda_i's$ are very close to each other, the eigenvalue repulsion phenomena dominates so the minimal eigenvalue gap is larger than that in a Poisson point process (see \cite{MR4034920} and \cite{MR3112927}) while the bound in Theorem \ref{theorem1.7} does not take into account this repulsion effect. Under our assumption that $\lambda_i$'s are distant, however, we believe the estimate in Theorem \ref{theorem1.7} is sharp modulo constants.

Given the strength of Theorem \ref{theorem1.7}, there is still much to do before our theorem can be turned to a solution to Conjecture \ref{conjecture1.6}. First, we have made additional density assumptions on matrix entries, excluding the Rademacher variables; second, we only work with eigenvalues in the bulk $I_\kappa$; and third, the estimate breaks down when eigenvalues are very close: in the context of Conjecture \ref{conjecture1.6} this accounts for the region where two eigenvalues in $[-\epsilon,\epsilon]$ sum up to zero for $\epsilon\sim n^{-1/2}$. A different treatment is needed in this region, and we leave this topic for future work.

We should also mention that the phenomena revealed in Theorem \ref{theorem1.7} apply to far more general nonlinear relations than the one stated here. Although Theorem \ref{theorem1.7} is stated for linear relations only, it will be clear from the proof that the same result holds for a wide variety of nonlinear relations on $d$ variables. Namely, consider a function $F(x_1,\cdots,x_d)$ which is Lipschitz continuous in $x_1,\cdots,x_{d-1}$, that for any $x_1,\cdots,x_{d-1}$ there are only finitely many $x_d$ such that $F(x_1,\cdots,x_d)=0$ (we denote the collection of such points as $Z_{x_1,\cdots,x_{d-1}}$), and for a fixed $x_1,\cdots,x_{d-1}$, $|F(x_1,\cdots,x_d)|\leq\epsilon\implies \operatorname{dist}(x_d,Z_{x_1,\cdots,x_{d-1}})\leq C\epsilon$. Then the conclusion of Theorem \ref{theorem1.7} holds for any such function $F$ as well, with identical proofs.

\subsection{Proof techniques and major challenges}

The main result Theorem \ref{Theorem1.1} of this paper is a multilocation generalization of the singular value estimate in \cite{campos2024least}, which is the culmination of many prior works on the singularity of symmetric random matrices, see \cite{MR2267289}, \cite{MR2891529}, \cite{MR4439776}, \cite{vershynin2014invertibility}, \cite{MR4255046}. 

From a broad perspective, we follow Rudelson and Vershynin's geometric method of invertibility, see \cite{rudelson2008littlewood}. However, this seems to be the first time that the method has been generalized to multiple matrices. Recall that in the approach of \cite{rudelson2008littlewood}, we first show that the matrix is well invertible on the set of compressible vectors via a net argument. Since incompressible vectors have at least $\rho n$ coordinates with an absolute value of at least $cn^{-1/2}$, we can reduce the singular value lower bound to bounding the possibility that one column of this matrix is very close to the linear span of the other columns.

Now we have several matrices and would like to derive joint lower tail estimates for singular values. We may still heuristically expect the above reasoning to hold, so that we bound the possibility that a fixed column of each matrix is very close to the linear span of the other columns of the same matrix. However, there is a flaw: each incompressible vector has only $\rho n$ coordinates that are not too small, and if we take two such vectors, their location of large coordinates may not even intersect if $\rho$ is small. 

Fortunately, this flaw can be resolved by the following fact: the minimal singular values of $A_n$ are achieved by the eigenvectors of $A_n$, and the eigenvectors of $A_n$ satisfy a much stronger property than being incompressible: the no-gap delocalization result of \cite{rudelson2016no} (see Theorem \ref{veragaghagag}) guarantees that with a very high possibility, all eigenvectors have $(1-c_1)n$ coordinates that are at least $(c_1c_2)^6n^{-1/2}$ in absolute value, for any $c_1>0$. This allows us to take the union of the large locations of $d$ different eigenvectors, justifying the heuristic reduction. The multi-matrix version of the invertibility via the distance approach is quite nontrivial even from the first step.

A second, more serious issue emerges when we consider discrete entry distributions (i.e. entry distribution without a continuous density). It is not so difficult to prove that any eigenvector of $A_n$ should have a large least common denominator (LCD) with very high possibility (for example see \cite{rudelson2008littlewood}, \cite{vershynin2014invertibility} and \cite{campos2024least} for definitions and proofs). This is because each eigenvector $v$ satisfies $(A_n-\lambda I_n)v=0$ for some $\lambda$, and this matrix $A_n$ has a large portion of independent components, permitting the use of the tensorization lemma (see \cite{rudelson2008littlewood}, \cite{vershynin2014invertibility}) to obtain an exponential bound. Now we consider multiple locations, and thus we have to consider the least common denominator of vectors of the form $v=\cos\theta v_1+\sin\theta v_2$ where $v_1$ and $v_2$ are two eigenvectors of $A_n$. This time $v$ is only annihilated by $A_n^2+\alpha A_n+\beta$ for some $\alpha$ and $\beta$, and the matrix $A_n^2$ has a complete loss of independence: every other element is strongly dependent. We are not aware of any method that can effectively solve this issue, so we consider only entries with a continuous distribution, completely avoiding the arithmetic conditions based on LCD. It would be very fortunate to overcome this barrier in the future. Note that when $d=2$, we can still prove that $v$ is incompressible for any $\theta$ with very high probability, but this information is not strong enough to complete the proof. When $d\geq 3$ the argument completely breaks down for the linear span of three different eigenvectors, so we do not pursue this approach any further.

We should also note that in addition to assuming that the entry distributions of $A_n$ have a continuous density, we also assume that the density has a component whose Fourier transform decays like a Gaussian \eqref{assumptionone}. This assumption will be used in Section \ref{chap2chap2} and significantly simplifies the reasoning. It might be possible to remove this assumption with some effort, but we believe that removing all the density assumptions is pretty challenging at the current stage.

There are several other remaining technical difficulties. First, when analyzing the function $I(\theta)$ defined in \eqref{Itheta} for different values of $\theta$, it appears that we need some care to where $\theta$ lies in the integration region. We divide the region into an essentially lower dimensional region and a genuinely full dimensional region, and for the latter we have to rely on our induction hypothesis, that is, we prove Theorem \ref{Theorem1.1} inductively from $d=2$ upward. For the full dimensional region, it is very natural to hope for a probability bound on $I(\theta)$ that is good enough to tame the singularity of integration (see for example \eqref{integration97000}). This leads to the estimate in Lemma \ref{lemma6.61}. We have been slightly careless in the power of $s$ in Lemma \ref{lemma6.61}: one would expect the estimate to be $s^de^{-k}$ so as to tame the singularity in \eqref{integration97000}. Unfortunately this approach is neither feasible nor helpful: it seems that raising to power $s^d$ in Lemma \ref{lemma6.61}is extremely difficult and may not be possible at all, and more importantly, our decoupling estimate in Theorem \ref{theorem3.1} causes us to eventually lose half of the power of $s$. 

To address the last difficulty, we take special care and integrate to $(\prod_{i=1}^d\frac{\mu_1(\lambda_i)}{\sqrt{n}})^{1-\frac{1}{4d}}$ at the end (where $\mu_1(\lambda_i)$ is the largest singular value of $(A_n-\lambda_i I_n)^{-1}$). There are several considerations when we choose the product: it would also be possible to integrate to, say, $(\frac{\mu_1(\lambda_i)}{\sqrt{n}})^{d-1/4}$, but the expectation of this term diverges as the power becomes too high. One may also consider controlling $\prod_{i=1}^d\mu_1(\lambda_i)$ via Hölder inequality to decouple one another, but this would again lead to a very high power. The correct idea here is to treat the product $\prod _{i=1}^d\mu_1(\lambda_i)$ as a whole and perform a bootstrapping procedure, whose details are outlined in Section \ref{bootstrapping209}. Indeed, we are somewhat using the intuition that $\mu_1(\lambda_i)$ are mutually independent, so that once the power of each of them is less than one, the expectation is still finite. This is not a proof by circulation: we just regard $\prod \mu_1(\lambda_i)$ as a whole and bootstrap, ultimately showing that they are somewhat asymptotically independent. By these considerations, we choose to use $\prod_{i=1}^d |\theta_i|,\theta\in\mathbb{R}^d$ as a measure to partition the region of integration, rather than using the standard Euclidean norm.

\subsection{Notation conventions and some preliminary facts} 

Recall that $\Gamma$ denotes all the mean zero, variance one, subGaussian random variables.

For any $\zeta\in\Gamma$ we denote by $\operatorname{Sym}_n(\zeta)$ the probability space of symmetric $n\times n$ random matrices with entries $(A_{ij})_{i< j}$ independent and identically distributed with distribution $\zeta$, and entries $(A_{ii})_{1\leq i\leq n}$ i.i.d. with distribution $\bar{\sigma_0}\zeta$ and independent of $(A_{ij})_{i< j}$. Strictly speaking this definition also depends on $\bar{\sigma_0}$, but we fix this parameter in the sequel and omit it in the notation. We also write $X\in\operatorname{Col}_n(\zeta)$ if $X\in\mathbb{R}^n$ has independent coordinates following the distribution $\zeta$.

The spaces $(\operatorname{Sym}_n(\zeta))_{n\geq 1}$ are coupled as follows: we sample $A_{n+1}\sim\operatorname{Sym}_{n+1}(\zeta)$ via first sampling $A_n\sim\operatorname{Sym}_n(\zeta)$ and put it in the $[2,\cdots,n+1]$ principal minor in $\operatorname{Sym}_{n+1}(\zeta)$, then independently sample a random vector $X\in\operatorname{Col}_n(\zeta)$ as the first row and column (with the first entry removed), and finally sample the (1,1)-th entry of $A_{n+1}$ independently from $\zeta$.

Different sets of assumptions are used to state the main estimates in this paper.
In some of these lemmas, we only assume an entry distribution $\zeta\in\Gamma_B$ for some $B$. In other lemmas we assume moreover that $\zeta$ satisfies the assumptions in Theorem \ref{Theorem1.1}. In other words, in the former case the estimate depends only on $B$ but in the latter case the estimate also depends on the constants $G$,$K$ and $\sigma_0$ in the definitions of $\zeta_1$ and $\zeta_2$.

Here we collect some basic facts that will be useful throughout the paper.

Let $Y\sim\operatorname{Col}_n(\zeta)$ with $\zeta\in\Gamma_B$, then for any fixed $u\in\mathbb{R}^n$, 
\begin{equation}\label{subGaussianmoment1}
    \mathbb{E}_Ye^{\langle Y,u\rangle}\leq\exp(2B^2\|u\|_2^2).
\end{equation}

Before we start the proof we recall the following two useful properties from Vershynin \cite{vershynin2014invertibility}.

\begin{lemma}\label{operatorbound}
    Given $B>0$, $\zeta\in\Gamma_B$ and $A_n\sim\operatorname{Sym}_n(\xi)$. 
Then for any $t>0$  we have
\begin{equation}
\mathbb{P}\left(\|A_n\|_{op}\geq (3+t)\sqrt{n}\right)\leq e^{-ct^{3/2}n}
\end{equation}
for some constant $c>0$ depending only on $B$.

\end{lemma}

We also need the following linear algebra fact from Vershynin \cite{vershynin2014invertibility}:

\begin{Fact}\label{fact2.4}
    For $A=(a_{ij})$ an $n\times n$ matrix, let $A_1$ be the first column of $A$ and $H_1$ the span of the other $n-1$ columns. Additionally, let $B$ be the $(n-1)\times (n-1)$ minor from removing the first row and first column of $A$. Let $X\in\mathbb{R}^{n-1}$ be the first column of $A$ with the first element removed. Then 
    $$\operatorname{dist}(A_1,H_1)=\frac{\left|\langle B^{-1}X,X\rangle-a_{11}\right|}{\sqrt{1+\|B^{-1}X\|_2^2}}.$$
\end{Fact}

\subsection{Plan of the paper}
The main body of this paper, until Section \ref{multiparticlecase}, is devoted to proving Theorem \ref{Theorem1.1}. We will follow an inductive approach on $d\in\mathbb{N}_+$. We feel it would be better to perform the $d=2$ case first in full detail, and return to the general $d$ case later via an induction argument. We prove the $d=2$ case of Theorem \ref{Theorem1.1} in Sections \ref{bootstrapping209} and outline the main modifications for the general $d$ case in Section \ref{multiparticlecase}. The contents in sections \ref{chap2chap2}, \ref{chap3chap3chap3} and \ref{chap4chap4chap4} work for general $d$, where Section \ref{chap3chap3chap3} is only about eigenvalue spacing estimates. Finally, in Section \ref{section9section9} we prove Theorem \ref{Theorem1.2} as a modification of the proof of Theorem \ref{Theorem1.1}.

\section{Small ball probability of quadratic forms}\label{chap2chap2}

The main result of this section is the following lemma, which provides a means to effectively estimate the quantities appearing in the numerator of the distance function in Fact \ref{fact2.4}.

\begin{theorem}\label{theorem3.1}
    Given $B,\sigma,K>0$, $\zeta_1\in \Gamma_1(G,B,\sigma)$ and $\zeta_2\in\Gamma_2(G,K,\sigma)$ be independent. Let $A_n$ be any $n\times n$ matrix and $X\sim \operatorname{Col}_n(\zeta_1+\zeta_2)$ be a random vector independent of $A_n$. Let $\lambda_1,\cdots,\lambda_d\in\mathbb{R}$ be fixed real numbers and $\delta_1,\cdots,\delta_d>0$. For each $i=1,\cdots,d$ denote by $\mu_1(\lambda_i):=\sigma_{\operatorname{max}}\left((A_n-\lambda_i I_n)^{-1}\right)$. Then, for any $(t_1,\cdots,t_d)\in\mathbb{R}^d$, and any $u\in\mathbb{R}^n$ with $\|u\|_2\leq d$,
\begin{equation}\begin{aligned}
    &\mathbb{P}_X\left(\bigcap_{i=1}^d\left\{  \left\langle (A_n-\lambda_iI_n)^{-1}X,X\rangle-t_i|\leq \delta_i\mu_1(\lambda_i)\right\} ,\langle X,u\right\rangle\geq s\right)\\&\lesssim \prod_{i=1}^d\delta_i\cdot e^{-s}\cdot \int_{\{\sum_{i=1}^d |\theta_i\delta_i|^2\leq d\}\subset\mathbb{R}^d} I(\theta)^\frac{1}{2}d\theta,\end{aligned}
\end{equation}
where $I(\theta)$ is defined as follows: for two independent random vectors $X_2,X_2'\sim\operatorname{Col}_n(\zeta_2)$,  we define $I(\theta)$ for each $\theta\in\mathbb{R}^d$ via
\begin{equation}\label{Itheta}
I(\theta)=\mathbb{E}_{X_2,X_2'}\exp\left(\langle X_2+X_2',u\rangle-c\left\|\sum_{i=1}^d \frac{\theta_i}{\mu_1(\lambda_i)}(A_n-\lambda_iI_n)^{-1}(X_2-X_2')\right\|_2^2
    \right),
\end{equation}
where $c>0$ depends only on $G,B,\sigma$ and $K$.
\end{theorem}

We first prove a version of the quadratic form decoupling estimate, which is adapted from \cite{campos2024least}, Lemma 5.3.

\begin{lemma}\label{decouplinglemma} Given $\xi_1\in \Gamma_1(G,B,\sigma)$ and $\xi_2\in \Gamma_2(G,K,\sigma)$. Let $X_1,X_1'\sim \operatorname{Col}_n(\xi_1)$ be independent vectors and let let $X_2,X_2'\sim \operatorname{Col}_n(\xi_2)$ be independent vectors. We use the symbol $\mathbb{E}_X$ to denote the expectation with respect to both $X_1$ and $X_2$, and we use $\mathbb{E}_{X_i}$ to denote the expectation with respect to $X_i$ for each $i$. Then for any $n\times n$ symmetric matrix $M$ and $u\in\mathbb{R}^n$, we have
$$\begin{aligned}&
\left|\mathbb{E}_X e^{2\pi i\theta\langle M(X_1+X_2),(X_1+X_2)\rangle+\langle X_1+X_2,u\rangle}\right|^2
\\&\quad\leq \mathbb{E}_{X_2,X_2'}e^{\langle(X_2+X_2'),u\rangle}\cdot 
\left| \mathbb{E}_{X_1}e^{4\pi i\theta \langle M(X_2-X_2'),X_1\rangle+2\langle X_1,u\rangle}
\right|.
\end{aligned}$$
\end{lemma}

\begin{proof}
    Writing $\mathbb{E}_X=\mathbb{E}_{X_1}\mathbb{E}_{X_2}$ and applying Jensen's inequality, we deduce 
    $$\begin{aligned}
&E:=|\mathbb{E}_X e^{2\pi i\theta\langle M(X_1+X_2),X_1+X_2\rangle+\langle X_1+X_2,u\rangle}|^2\\&\leq \mathbb{E}_{X_1}|\mathbb{E}_{X_2}e^{2\pi i\theta \langle M(X_1+X_2),X_1+X_2\rangle+\langle X_1+X_2,u\rangle}|^2.\end{aligned}$$

Expanding the square $|\mathbb{E}_{X_2}e^{2\pi i\theta \langle M(X_1+X_2),X_1+X_2\rangle+\langle X_1+X_2,u\rangle}|^2$, we obtain

$$\begin{aligned}&
\mathbb{E}_{X_2,X_2'}e^{2\pi i\theta \langle M(X_1+X_2),(X_1+X_2)\rangle+\langle (X_1+X_2),u\rangle-2\pi i\theta 
\langle M(X_1+X_2'),(X_1+X_2')\rangle+\langle X_1+X_2',u\rangle}\\&
\quad=\mathbb{E}_{X_2,X_2'}e^{4\pi i\theta \langle M(X_2-X_2'),X_1\rangle+\langle X_2+X_2',u\rangle+2\langle X_1,u\rangle+2\pi i\theta \langle MX_2,X_2\rangle-2\pi i\theta \langle MX_2',X_2'\rangle}.
\end{aligned}$$
Swapping expectations, we obtain 
$$\begin{aligned}
E&\leq\mathbb{E}_{X_2,X_2'}|\mathbb{E}_{X_1}e^{4\pi i\theta \langle M(X_2-X_2'),X_1\rangle+\langle X_2+X_2',u\rangle+2\langle X_1,u\rangle+2\pi i\theta \langle MX_2,X_2\rangle-2\pi i\theta\langle MX_2',X_2'\rangle}|.
\end{aligned}$$ Swapping the $X_1$-independent terms out completes the proof.
\end{proof}

Now we prove a lemma that makes use of the Fourier decay of $X_1$ and decouples the inner product $\langle X_1,u\rangle$.

\begin{lemma}\label{lemma3.3great888} Let $X_1\sim \operatorname{Col}_n(\zeta_1)$ where $\zeta_1\in\Gamma_1(G,B,\sigma)$, and let $u,v\in\mathbb{R}^n$ be two fixed vectors. Then 
\begin{equation}
    |\mathbb{E}_X e^{2\pi i\langle X,v\rangle+\langle X,u\rangle}|\leq c^{-1}\exp\left(-c\|v\|_2^2+c^{-1}\|u\|_2^2\right)
\end{equation} where $c\in(0,1)$ is a fixed constant depending only on $G,B,\sigma$.
\end{lemma}

\begin{proof}
    Consider $\zeta_1'$ an independent copy of $\zeta_1$. We compute 
    $$
|\mathbb{E}_{\zeta_1}e^{2\pi i\zeta_1 v_j+\zeta_1 u_j}|^2=\mathbb{E}_{\zeta_1,\zeta_1'}e^{2\pi i(\zeta_1-\zeta_1')v_j+(\zeta_1+\zeta_1')u_j}=\mathbb{E}_{\zeta_1,\zeta_1'}[e^{(\zeta_1+\zeta_1')u_j}\cos(2\pi(\zeta_1-\zeta_1')v_j)].$$

Consider $\widetilde{X}=(\widetilde{X}_i)_{i=1}^n,$ $\widetilde{Y}=(\widetilde{Y}_i)_{i=1}^n$ two given vectors with i.i.d. coordinates having the distribution $\xi:=\zeta_1-\zeta_1'$ and $\zeta_1+\zeta_1'$. Then by Cauchy-Schwartz and the bound $|\cos(x)|^2\leq |\cos(x)|$,
\begin{equation}
|\mathbb{E}_X e^{2\pi i\langle X,v\rangle+\langle X,u\rangle}|^2\leq \mathbb{E}e^{\langle \widetilde{Y},u\rangle}\prod_j \cos(2\pi \widetilde{X}_jv_j)\leq (\mathbb{E}_{\widetilde{Y}}e^{2\langle \widetilde{Y},u})^\frac{1}{2}    \left(     \prod_j \mathbb{E}_\xi |\cos(2\pi \xi v_j)|\right)^\frac{1}{2}.
\end{equation} By \eqref{subGaussianmoment1}, we have $\mathbb{E}_{\widetilde{Y}}e^{2\langle \widetilde{Y},u\rangle}$ is bounded by $\exp(\mathcal{O}(\|u\|_2^2))$.
Applying the Fourier decay assumptions on $\zeta_1$ \eqref{assumptionone} finally completes the proof.
\end{proof}
Note: the Fourier decay assumption \eqref{assumptionone} is essential here. Without this assumption we can only obtain $\|v\|_\mathbb{T}$ in the estimate (see \cite{campos2024least}, Fact 5.4), which measures the closeness of $v$ to the integer lattice. Unfortunately we do not have good control of this quantity in this paper where we have multiple locations $\lambda_i\neq\lambda_j$.

We also need the following multidimensional Esseen inequality \cite{esseen1966kolmogorov}:
\begin{lemma}
Let $Z$ be a random vector in $\mathbb{R}^m$, then 
$$
\sup_{v\in\mathbb{R}^m} \mathbb{P}(\|Z-v\|_2\leq \sqrt{m})\leq C^m\int_{B(0,\sqrt{m})}|\phi_Z(\theta)|d\theta, 
$$
with $\phi_Z(\theta)=\mathbb{E}\exp( 2\pi i\langle \theta,Z \rangle)$.
\end{lemma}
Via a simple rescaling argument, we deduce that

\begin{corollary}\label{corollary2.4firsts}
    Let $Z_1,Z_2,\cdots Z_d$ be real-valued random variables (not necessarily independent) and let $\delta_1,\delta_2,\cdots,\delta_d>0$ be fixed constants. Then 
    \begin{equation}
 \sup_{\{v_i\}_{i=1}^d\in\mathbb{R}^d}\mathbb{P}\left(\cap_{i=1}^d \{\|Z_i-v_i\|\leq \delta_i\} \right)\lesssim \prod_{i=1}^d \delta_i \int_{\{\sum_{i=1}^d |\theta_i\delta_i|^2\leq d\}}|\phi_{\{Z_i\}_{i=1}^d}(\theta)|      d\theta, 
    \end{equation}
where $\theta\in\mathbb{R}^d$ and $\phi_{\{Z_i\}_{i=1}^d}(\theta)=\mathbb{E}\left(2\pi i(\sum_i \theta_i Z_i)\right)$.
    
\end{corollary}

We prove the following lemma, which is a multidimensional version of \cite{campos2024least}, Lemma 5.1:

\begin{lemma}
   Let $X\sim \operatorname{Col}_n(\zeta_1+\zeta_2)$, where $(\zeta_1,\zeta_2)\in \Gamma_1(G,B,\sigma)\times \Gamma_2(G,K,\sigma)$ and $\zeta_1,\zeta_2$ are independent. Let $M_i$, $i=1,\cdots,d$ be $d$ $n\times n$ real symmetric matrices, $u\in\mathbb{R}^n$, $t_1,\cdots,t_d\in\mathbb{R}$ and $s,\delta_1,\cdots,\delta_d\geq 0$. Then we have
\begin{equation}\begin{aligned}
 &\mathbb{P}\left(\bigcap_{i=1}^d \{|\langle M_iX,X\rangle-t_i|<\delta_i\},\langle X,u\rangle\geq s\right)       
   \\& 
   \lesssim \prod_{i=1}^d \delta_i \cdot e^{-s}\cdot \int_{\{\sum_{i=1}^d|\theta_i\delta_i|^2\leq d\}} |\mathbb{E}^{2\pi i\sum_{i=1}^d \theta_i\langle M_iX,X\rangle +\langle X,u\rangle}|d\theta.
\end{aligned}\end{equation}  
\end{lemma}

\begin{proof}
Using $1\{x\geq s\}\leq e^{x-s}$ we bound the left hand side of the equation from above by 
\begin{equation}\label{equation2.41}
    e^{-s}\mathbb{E}\left[\prod_{i=1}^d\mathbf{1}\{|\langle M_iX,X\rangle-t_i|<\delta_i\}\cdot e^{X,u}\right].
\end{equation}

Consider the auxiliary random variable $Y\in\mathbb{R}^n$ that satisfies, for all open $U\subseteq \mathbb{R}^n$,
\begin{equation}
    \mathbb{P}(Y\in U)=(\mathbb{E}e^{\langle X,u\rangle})^{-1}\mathbb{E}[\mathbf{1}_Ue^{\langle X,u\rangle}].
\end{equation}

Then we rewrite \eqref{equation2.41} as
\begin{equation}
 \eqref{equation2.41}=e^{-s}(\mathbb{E}e^{\langle X,u\rangle}) \mathbb{P}_Y\left(\cap_{i=1}^d \{|\langle M_iY,Y\rangle-t_i|\leq\delta_i\}\right).  
\end{equation}
Then we apply Corollary \ref{corollary2.4firsts} to $Y$ and obtain 
$$\begin{aligned}
&\mathbb{P}_Y\left(\cap_{i=1}^d \{|\langle M_iY,Y\rangle-t_i|\leq\delta_i\}\right)\\&\quad \lesssim\prod_{i=1}^d \delta_i \cdot \int_{\{\sum_{i=1}^d|\theta_i\delta_i|^2\leq d\}} |\mathbb{E}_Ye^{2\pi i\sum_{i=1}^d \theta_i\langle M_iY,Y\rangle}|d\theta.
\end{aligned}$$
Finally, by definition of $Y$,
$$
\mathbb{E}_Y e^{2\pi i\sum_{i=1}^d \theta_i\langle M_iY,Y\rangle}=(\mathbb{E}_Xe^{\langle X,u\rangle})^{-1}\mathbb{E}e^{2\pi i\sum_{i=1}^d \theta_i\langle M_iX,X\rangle+\langle X,u\rangle},
$$
and the proof is completed.
\end{proof}

Now we are ready to prove Theorem \ref{theorem3.1}.
\begin{proof}[\proofname\ of Theorem \ref{theorem3.1}]

Throughout the proof we write $X=X_1+X_2$ where $X_1\sim\operatorname{Col}_n(\zeta_1)$ and $X_2\sim\operatorname{Col}_n(\zeta_2)$ are independent random vectors.

    We first apply Corollary \ref{corollary2.4firsts} to write
\begin{equation}\begin{aligned}
    &\mathbb{P}_X\left(\cap_{i=1}^d \left\{\left|\langle (A_n-\lambda_iI_n)^{-1}X,X\rangle-t_i\right|\leq \delta_i\cdot  \mu_1(\lambda_i)\right\},\quad  \langle X,u\rangle\geq s\right)\\&\lesssim \prod_{i=1}^d\delta_i\cdot e^{-s}\int_{\{\sum_{i=1}^d|\theta_i\delta_i|^2\leq d\}} \left| \mathbb{E}_X e^{2\pi i\left\langle \left(\sum_{i=1}^d  \frac{ \theta_i (A_n-\lambda_i I_n)^{-1}X}{\mu_1(\lambda_i)}\right),X\right\rangle +\langle X,u\rangle }\right|d\theta.\end{aligned}
\end{equation}

Then we apply the decoupling lemma (Lemma \ref{decouplinglemma}) to deduce that 
\begin{equation}\begin{aligned}&
   \left| \mathbb{E}_X e^{2\pi i\left\langle \left(\sum_{i=1}^d  \frac{ \theta_i (A_n-\lambda_i I_n)^{-1}X}{\mu_1(\lambda_i)}\right),X\right\rangle +\langle X,u\rangle }\right|^2\\&\leq \mathbb{E}_{X_2,X_2'}e^{\langle (X_2+X_2'),u\rangle}\cdot \left|\mathbb{E}_{X_1}e^{4\pi i\left\langle \left(\sum_{i=1}^d \frac{\theta_i(A_n-\lambda_i I_n)^{-1}\widetilde{X}_2}{\mu_1(\lambda_i)}\right),X_1\right\rangle+2\langle X_1,u\rangle}
   \right|,\end{aligned}
\end{equation}
where $\widetilde{X}:=X_2-X_2'$.

Finally, we apply Lemma \ref{lemma3.3great888} to the $\mathbb{E}_{X_1}$ term and complete the proof.

\end{proof}

\section{Eigenvalue spacing estimates}\label{chap3chap3chap3}
The main objective of this section is to establish eigenvalue rigidity estimates for Wigner matrices. We make crucial use of the local semicircle law, as well as the super-exponential concentration of the empirical measure guaranteed by a finite log-Sobolev constant on the entries. Informally speaking, the main results are that with probability $1-e^{-\Omega(n)}$, we have very well control on $\mu_{k}(\lambda_i)$ from both sides, for $k\geq cn$ and $\lambda_i$ in the bulk. We also have very good control of $\sqrt{n}/(k\mu_k(\lambda_i))$ from above, for any $k$ and any $\lambda_i$ in the bulk. This section is mostly technical and readers can skip to the next section and turn back once the results here are needed.

Throughout this section we will use a simple yet crucial property of symmetric matrices $A_n$: for any $\lambda_i\in\mathbb{R}$, $(A_n-\lambda_i I_n)^{-1}$ has the same set of eigenvectors as $(A_n-\lambda_i I_n)$, thus it has the same set of eigenvectors as $A_n$. Moreover, the singular values of $(A_n-\lambda_i I_n)^{-1}$ and that of $A_n$ have very simple relationships.

Recall that for any $\lambda_i\in\mathbb{R}$, we denote by $\mu_1(\lambda_i)\geq\mu_2(\lambda_i)\cdots\geq \mu_n(\lambda_i)\geq 0$ the singular values of $(A_n-\lambda_i I)^{-1}$ in decreasing order. Additionally, we define a modified function $\|\cdot\|_*$ for $(A-\lambda_i I)^{-1}$ via
\begin{equation}
    \left\|(A_n-\lambda_i I_n)^{-1}\right\|_*^2=\sum_{k=1}^n \mu_k(\lambda_i)^2(\log(1+k))^2.
\end{equation}

\subsection{Single location estimate}
\label{section4fuck}
We first prove the following two lemmas that generalize Lemmas 8.1 and 8.2 of \cite{campos2024least} to nonzero eigenvalues in the bulk. 

\begin{lemma}\label{lemma4.11} For any $p>1$, $B>0$, $\kappa>0$ and $\zeta\in\Gamma_B$, let $A\sim\operatorname{Sym}_n(\zeta)$. Then we can find a constant $C_p$ depending on $B$, $p$ and $\kappa$ such that for any $\lambda_i\in[-(2-\kappa)\sqrt{n},(2-\kappa)\sqrt{n}]$ we have
$$
\mathbb{E}\left[\left(\frac{\sqrt{n}}{\mu_k(\lambda_i)\cdot k}\right)^p\right]\leq C_p.
$$
\end{lemma}

  We will also need the following corollary:  

\begin{corollary}\label{lemma4.22}
    For a given $p>1$, $B>0$, $\kappa>0$ and $\zeta\in\Gamma_B$, $A\in\operatorname{Sym}_n(\zeta)$. Then we can find a constant $C_P>0$ depending on $B$, $p$, $\kappa$ such that for any $\lambda_i\in[(-(2-\kappa)\sqrt{n},(2-\kappa)\sqrt{n}]$, we have
    \begin{equation}
        \mathbb{E}\left[\left(\frac{\|(A_n-\lambda_i I_n)^{-1}\|_*}{\mu_1(\lambda_i)}\right)^p\right]\leq C_p.
    \end{equation}
\end{corollary}

The proofs of both results are modifications of those in \cite{campos2024least} to nonzero bulk values. For this proof we need the following local semicircular law estimates of Erdös, Schlein and Yau on Wigner matrices from \cite{erdHos2011universality}, Theorem 1.11. The Wigner semicircle law $\rho_{sc}$ is defined as 
\begin{equation}\label{semicircle}
\rho_{sc}(x):=\frac{1}{2\pi}\sqrt{(4-x^2)_+},\quad x\in\mathbb{R}.
\end{equation}

\begin{theorem}\label{theorem4.3}(\cite{erdHos2011universality}, Theorem 1.11.) Let $H\in\operatorname{Sym}_n(\zeta)$ where $\zeta\in\Gamma_B$ for some $B>0$. Let $\kappa>0$ and consider $E\in[-2+\kappa,2-\kappa]$. Denote by $\mathcal{N}_{\eta^*}(E)=\mathcal{N}_{I^*}$ the number of eigenvalues of $n^{-\frac{1}{2}}H$ in $I^*:=[E-\eta^*/2,E+\eta^*/2]$. Then we can find a universal constant $c_1$ and two constants $C$, $c$ depending only on $\kappa$ such that for any $\delta\leq c_1\kappa$ we can find a constant $K_\delta$ depending only on $\delta$ such that 
\begin{equation}
    \mathbb{P}\left\{\left|\frac{\mathcal{N}_{\eta^*}(E)}{n\eta^*}-\rho_{sc}(E)\right|\geq\delta\right\}\leq Ce^{-c\delta^2\sqrt{n\eta^*}},
\end{equation}
for all $\eta^*$ satisfying $K_\delta/n\leq\eta^*\leq 1$ and all $n\geq 2$.
\end{theorem}
Since $\rho_{sc}(x)$ is maximal at $x=0$ and decreases as $|x|$ increases, we deduce the following user-friendly corollary:
\begin{corollary}\label{corollary425}
    In the setting of Theorem \ref{theorem4.3}, for any $\kappa>0$ we have the following two estimates uniform in all $E\in[-2+\kappa,2-\kappa]$: for any $\eta^*\in[Cn^{-1},1]$,
\end{corollary}
\begin{equation}\label{1stcor4.4}
    \mathbb{P}\left\{\frac{\mathcal{N}_{\eta^*}(E)}{n\eta^*}\geq\pi\right\}\lesssim \exp(-c_1\sqrt{n\eta^*}),
\end{equation}
and \begin{equation}\label{2ndcor4.4}
    \mathbb{P}\left\{\frac{\mathcal{N}_{\eta^*}(E)}{n\eta^*}\leq \frac{1}{4}\rho_{sc}(2-\kappa)\right\}\lesssim \exp(-c_1\sqrt{n\eta^*}),
\end{equation}
where $C$ and $c_1$ are constants that depend only on $\kappa>0$.

From this estimate we deduce two immediate corollaries.

\begin{corollary}\label{corollary8.5}
Fix $B>0$, $\zeta\in\Gamma_B$ and $A_n\sim\operatorname{Sym}_n(\zeta)$. Fixing $\kappa>0$, then we can find positive constants $C$ and $c$ depending only on $\kappa$ such that, for all $s\geq C$ and $k\in\mathbb{N}$ that satisfy $sk\leq n$, uniformly for all $\lambda_i\in[-(2-\kappa)\sqrt{n},(2-\kappa)\sqrt{n}]$,
$$
\mathbb{P}\left(\frac{\sqrt{n}}{\mu_k(\lambda_i)\cdot k}\geq s\right)\lesssim\exp(-c(sk)^\frac{1}{2}).
$$
We also have, uniformly for all $\lambda_i\in[-(2-\kappa)\sqrt{n},(2-\kappa)\sqrt{n}]$, for all $k\in\mathbb{N}_+,$
$$
\mathbb{P}\left(\mu_k(\lambda_i)\geq \frac{\sqrt{n}}{Ck}\right)\lesssim \exp(-ck^\frac{1}{2}).
$$
\end{corollary}

\begin{proof} Let $C$ be the larger one of $\pi$, $\frac{4}{\rho_{sc}(2-\kappa)}$ and the constant $C$ in the statement of Corollary \ref{corollary425}.
    If $\frac{\sqrt{n}}{\mu_k(\lambda_i)\cdot k}\geq s$, then $N_{n^{-\frac{1}{2}}A}(-skn^{-1}+\lambda_i n^-\frac{1}{2},skn^{-1}+\lambda_k n^-\frac{1}{2})\leq k$, where we use the symbol $N_{H}(I)$ to denote the number of eigenvalues of $H$ in the interval $I$. 
    
    Then the first claim follows by applying estimate \ref{2ndcor4.4} with $\eta_*=skn^{-1}\geq sn^{-1}\geq Cn^{-1}$ and noticing $s^{-1}\leq \frac{\rho_{sc}(2-\kappa)}{4}$. The second claim follows from applying the estimate \ref{1stcor4.4} with $\eta_*=Ckn^{-1}\geq Cn^{-1}$ and noticing $C\geq \pi$. 
\end{proof}

Now we are in the place to prove Lemma \ref{lemma4.11} and Corollary \ref{lemma4.22}.

\begin{proof}[\proofname\ of Lemma \ref{lemma4.11}] Let $C$ be the constant determined in Corollary \ref{corollary8.5}. From the tail estimate on $\|A_n\|_{op}$ in Lemma \ref{operatorbound}, we deduce that for any $k\geq n/c$ and any $\lambda_i\in[-(2-\kappa)\sqrt{n},(2-\kappa)\sqrt{n}]$,
$$
\mathbb{E}\left(\frac{\sqrt{n}}{\mu_k(\lambda_i)\cdot k}\right)^p\leq\mathbb{E}_{A_n}\left(\frac{\sigma_1(A_n+\lambda_i I_n)\sqrt{n}}{k}
\right) ^p=O_p((n/k))^p=O_p(1).
$$ Then for $k\leq n/c$, we consider separately the following three events
$$
E_1^i=\left\{\frac{\sqrt{n}}{\mu_k(\lambda_i)\cdot k}\leq C\right\},\quad 
E_2^i=\left\{\frac{\sqrt{n}}{\mu_k(\lambda_i)\cdot k}\in [C,\frac{n}{k}]\right\},
\quad 
E_3^i=\left\{\frac{\sqrt{n}}{\mu_k(\lambda_i)\cdot k}\geq \frac{n}{k}\right\}
$$ and bound 
\begin{equation}
  \mathbb{E}\left(\frac{\sqrt{n}}{\mu_k(\lambda_i)\cdot k}\right)^p\leq C^p+ \mathbb{E}\left(\frac{\sqrt{n}}{\mu_k(\lambda_i)\cdot k}\right)^p \mathbf{1}_{E_2^i}+\mathbb{E}\left(\frac{\sqrt{n}}{\mu_k(\lambda_i)\cdot k}\right)^p \mathbf{1}_{E_3^i}. 
\end{equation}

For the second term, we apply Corollary \ref{corollary8.5} to bound 
$$
\mathbb{E}\left(\frac{\sqrt{n}}{\mu_k(\lambda_i)\cdot k}\right)^p\lesssim\int_C^{n/k}ps^{p-1}e^{-c\sqrt{sk}}ds=O_p(1).
$$

  For the third term, since $n/k\geq C$, we apply Corollary \ref{corollary8.5} with $s=n/k$ and deduce that $\mathbb{P}(E_3^i)\lesssim e^{-c\sqrt{n}}$. Then applying the Cauchy-Schwartz inequality,
$$\begin{aligned}
\mathbb{E}&\left(\frac{\sqrt{n}}{\mu_k(\lambda_i)\cdot k}\right)^p\mathbf{1}_{E_3^i}\leq \left(\mathbb{E}(\frac{\sigma_1(A_n-\lambda_i I_n)\sqrt{n}}{k})^{2p}\right)^{1/2}\mathbb{P}(E_3^i)^{1/2}\\\quad&\leq O_p(1)n^p e^{-c\sqrt{n}}=O_p(1).\end{aligned}
$$ The bound on $\sigma_1$ follows from Lemma \ref{operatorbound}. This completes the proof.
\end{proof}

\begin{proof}[\proofname\ of Corollary \ref{lemma4.22}]
    Recall from definition that 
    $$
\|(A_n-\lambda_i I_n)_*^2=\sum_{k=1}^n \mu_k^2(\lambda_i)(\log(1+k))^2.
    $$ We may take without loss of generality $p\geq 2$ thanks to Hölder's inequality, and apply the triangle inequality to obtain
    $$
\left[\mathbb{E}\left(\sum_{k=1}^n \frac{\mu_k^2(\lambda_i)(\log(1+k))^2}{\mu_1^2(\lambda_i)}\right)^{p/2}\right]^{2/p}
\leq\sum_{k=1}^n(\log(1+k))^2\mathbb{E}\left[\frac{\mu_k^p(\lambda_i)}{\mu_1^p(\lambda_i)}\right]^{2/p}.
    $$
    Now, taking $C$ as the constant given in Corollary \ref{corollary8.5}, we apply Lemma \ref{lemma4.11} and Corollary \ref{corollary8.5} to the upper bound 
    $$
\mathbb{E}\left[\frac{\mu_k^p(\lambda_i)}{\mu_1^p(\lambda_i)}\right]\leq (Ck)^{-p}\mathbb{E}\left[\left(\frac{\sqrt{n}}{\mu_1(\lambda_i)}\right)^p\right]+\mathbb{P}(\mu_k\geq \frac{\sqrt{n}}{Ck})\lesssim C_p^pk^{-p} 
    $$ for some other constant $C_P$. Then the corollary follows by plugging this estimate into the previous one.
\end{proof}

\subsection{Multiple location decay estimates}\label{multipledecayestimate} The main objective of this paper is to study singular values of $A_n$ at $d$ distinct values $\lambda_1,\cdots,\lambda_d$ that are separated from one another. For this reason, more refined decaying estimates of eigenvalues are needed. 

The following corollary follows from reformulating Corollary \ref{corollary8.5} and taking a union bound over all $k\in[n]$. This tells us that we can have a rather strong control of $\mu_k(\lambda_i)$ as long as $k=\Omega(n^\sigma)$ and $\sigma>0$.

\begin{corollary}\label{corollary3.6chap9} Let $B>0$, $\zeta\in\Gamma_B$ and $A_n\in\operatorname{Sym}_n(\zeta)$. Fix any $\kappa>0$. Fix any $\lambda_i\in[-(2-\kappa)\sqrt{n},(2-\kappa)\sqrt{n}]$ and any $\sigma\in(0,1)$. Then for any $c_0>0$ we can find $c,C_1$ and $C_2$ depending only on $B$, $\kappa$ (and not on $c_0$) such that
\begin{equation}\label{estimate4s}
\mathbb{P}\left(  \frac{C_1\sqrt{n}}{k}  \leq \mu_{k}(\lambda_i)\leq  \frac{C_2\sqrt{n}}{k}\text{ for all } c_0n^\sigma\leq k\leq n/c\right)\geq 1-\exp(-\Omega(n^\frac{\sigma}{2})).
\end{equation}
\end{corollary}
\begin{proof}
    This follows from taking a union bound over $k$ in the conclusion of the Corollary \ref{corollary8.5}.
\end{proof}

When $\sigma=1$, this estimate \eqref{estimate4s} has the disadvantage that it only gives an error of order $e^{-c\sqrt{n}}$, falling short of the desired $e^{-cn}$ error. To remedy this, we invoke the following stronger concentration estimate, which relies on the fact that $\zeta$ has a finite log-Sobolev constant. 

\begin{Proposition}\label{proposition4.7}[\cite{WOS:000542157900013}, Lemma 6.1]
    Let $B>0$, and $\zeta\in\Gamma_B$ be a random variable having a finite, $n$-independent log-Sobolev constant. Consider $A_n\sim\operatorname{Sym}_n(\zeta)$, and denote by $\sigma_1,\cdots,\sigma_n$ the eigenvalues of $A_n$. Define the empirical measure $\rho_{A_n}^n$ via 
$$\rho_{A_n}^n:=\frac{1}{n}\sum_{i=1}^n \delta_{\frac{\sigma_i}{\sqrt{n}}}.$$ 
where $\delta_\cdot$ is the delta measure. Then there exists $\kappa\in(0,\frac{1}{10})$ such that 
\begin{equation}\label{refinedconcentration}
    \lim_{n\to\infty}\frac{1}{n}\ln\mathbb{P}[d(\rho_{A_n}^n,\rho_{sc})>n^{-\kappa}]=-\infty,
\end{equation}
    where $\rho_{sc}$ is the semicircle law \eqref{semicircle} and the distance $d$ is the Dudley (i.e., bounded Lipschitz) distance defined as follows: for two probability measures $\mu$ and $\nu$,
    $$
d(\mu,\nu)=\sup_{\|f\|_L\leq 1}\left|\int f(x)d\mu(x)-\int f(x)d\nu(x)\right|,
    $$
    where $\|f\|_L:=\sup_{x\neq y}\frac{|f(x)-f(y)|}{|x-y|}+\sup_x |f(x)|$.
    \end{Proposition}

Via this refined concentration estimate \eqref{refinedconcentration}, we can prove that on any macroscopic interval of length $x\sqrt{n}$ centered at $\lambda_i$, there cannot be too few eigenvalues or too many eigenvalues of $A_n$. We also derive the speed at which $\mu_k(\lambda_i)$ decreases as $k$ increases. We formulate this idea as follows:

\begin{corollary}\label{corollary4.88}
    Given any $B>0$, a random variable $\zeta\in \Gamma_B$ that has a finite, $n$-independent log-Sobolev constant, and a random matrix $A_n\sim\operatorname{Sym}_n(\zeta)$. Then for any $\kappa>0$ we can find some $C_1,C_2,C_3>0$ depending only on $\kappa$, $B$ and the log-Sobolev constant (and not on $c_0$) such that, for any $c_0>0$ and  $\lambda_i\in[-(2-\kappa)\sqrt{n},(2-\kappa)\sqrt{n}]$,
    \begin{equation}\label{coro4.8first}
   \mathbb{P}\left( \frac{C_1\sqrt{n}}{k}\leq 
   \mu_k(\lambda_i)\leq \frac{C_2\sqrt{n}}{k}\text{ for any } c_0n\leq k\leq C_3n
   \right)\geq 1-e^{-\Omega(n)}.     
    \end{equation}

In particular, we can find $C_4>1$ depending only on $\kappa$, $B$ and the log-Sobolev constant such that for any $c_0>0$ and for each $\lambda_i$ in the given interval,
$$
\left(\mu_{\frac{k}{C_4}}(\lambda_i)\geq {4}\mu_k(\lambda_i) \text{ for any } c_0n\leq k\leq C_3n\right)\geq 1-e^{-\Omega(n)}.
$$
    
\end{corollary}

\begin{proof} In the proof, the constants $C_1$ and $C_2$ will not be determined until the end of the proof. The fact that 
$\mu_k(\lambda_i)\geq \frac{C_2\sqrt{n}}{k}$ implies that on the interval $[\lambda_i-\frac{k}{C_2\sqrt{n}},\lambda_i+\frac{k}{C_2\sqrt{n}}]$, there are more than $k$ eigenvalues of $A_n$. Likewise, $\mu_k(\lambda_i)\leq \frac{C_1\sqrt{n}}{k}$ implies that on the interval $[\lambda_i-\frac{k}{C_1\sqrt{n}},\lambda_i+\frac{k}{C_1\sqrt{n}}]$, there are fewer than $k$ eigenvalues of $A_n$. 

We now normalize by dividing the interval by $\sqrt{n}$ and consider $I_2:=[\frac{\lambda_i}{\sqrt{n}}-\frac{k}{C_2n},\frac{\lambda_i}{\sqrt{n}}+\frac{k}{C_2n}]$. Since the semicircle law $\rho_{sc}(E)$ has bounded density, this would imply that $\rho_{sc}(I_2)\leq \rho_{sc}(0)|I_2|=\rho_{sc}(0)\frac{2k}{C_2n}$. Moreover, $\mu_k(\lambda_i)\geq \frac{C_2\sqrt{n}}{k}$ implies that $\rho_{A_n}^n(I_2)\geq \frac{k}{n}$. Therefore, if we choose $\rho_{sc}(0)\frac{2}{C_2}<1$, we must have $|\rho_{A_n}^n(I_2)-\rho_{sc}(I_2)|\geq c_0(1-\rho_{sc}(0)\frac{2}{C_2})>0$. Since $|I_2|\geq c_0>0$, this event is super-exponentially unlikely according to Proposition \ref{proposition4.7}.

Likewise, we consider $I_1:=[\frac{\lambda_i}{\sqrt{n}}-\frac{k}{C_1n},\frac{\lambda_i}{\sqrt{n}}+\frac{k}{C_1n}]$. Assume without loss of generality that $\lambda_i\geq 0$. If $\frac{k}{{C_1n}}\leq 2$, then $\mu_{sc}(I_1)\geq \frac{k}{C_1n}\rho_{sc}(2-\kappa)$. However the assumption that $\mu_k(\lambda_i)\leq \frac{C_1\sqrt{n}}{k}$ tells us that $\mu_{A_n}^n(I_1)\leq\frac{k}{n}$. Therefore, if we choose $C_1=\frac{1}{2}\rho_{sc}(2-\kappa)$, this event is superexponentially unlikely by Proposition \ref{proposition4.7}. The proof is completed by taking $C_3=2C_1$.

To check the last claim, we choose $C_4=4C_2/C_1$. We apply the first claim and obtain $\mu_{\frac{k}{C_4}}(\lambda_i)\geq \frac{C_1C_4\sqrt{n}}{k}
=4\frac{C_2\sqrt{n}}{k}\geq 4\mu_k(\lambda_i)$ on an event with probability $1-e^{-\Omega(n)}$.
\end{proof}

We derive one more estimate, showing that it is very unlikely for $\frac{\mu_k(\lambda_i)}{\mu_1(
\lambda_i)}$ to be $\Omega(1)$ when $k$ is large enough.

\begin{corollary}\label{corollary4.999}
    In the setting of Corollary \ref{corollary4.88}, for any $\lambda_i\in[(-(2-\kappa)\sqrt{n},(2-\kappa)\sqrt{n}]$, any $c_0>0$,
    \begin{equation}
\mathbb{P}\left(\frac{\mu_k(\lambda_i)}{\mu_1(\lambda_i)}\geq 10^{-2}\text{ for any } k\geq c_0n\right)\leq e^{-\Omega(n)}.
    \end{equation}
\end{corollary}

\begin{proof}
    By corollary \ref{corollary4.88}, we only need to prove the following claim: for any fixed constant $D>0$,
    \begin{equation}
        \mathbb{P}(\mu_1(\lambda_i)\leq Dn^{-1/2})=e^{-\Omega(n)}.
    \end{equation}
    To see this, observe that $\mu_1(\lambda_i)\leq Dn^{-1/2}$ implies that there is an interval of length $2D^{-1}\sqrt{n}$ around $\lambda_i$ on which $A_n$ has no eigenvalue. By the convergence in Proposition \ref{proposition4.7}, this event is superexponentially unlikely.
\end{proof}

\section{Small ball probability estimates}
\label{chap4chap4chap4}
In this section we estimate a crucial term in $I(\theta)$, the function defined in \eqref{Itheta} whose computations lead to the desired singular value bounds. The following lemma is the main result of this section. The reasoning in its formulation and proof are far more complicated than that in the one location case ($d=1$ and $\lambda_1$=0) as we need to balance $\mu_k(\lambda_j)$ for various $k$ and $j$. The eigenvalue rigidity estimates developed in the previous section are crucial here.

\begin{lemma}\label{lemma6.61}
 Assume that $A_n$ satisfies the assumptions in Theorem \ref{Theorem1.1} and assume the random vector $X_2\sim\operatorname{Col}_n(\zeta_2)$ and $\zeta_2\in\Gamma_2(G,K,\sigma_0)$. Fix any $\kappa>0$ and $\Delta>0$ that are sufficiently small. Then there exists an event $\mathcal{E}_1$ with $\mathbb{P}(\mathcal{E}_1)\geq 1-\exp(-\Omega(n))$ such that for any $A_n\in\mathcal{E}_1$, the following holds:

   Fix any given $(\theta_1,\cdots,\theta_d)\in\mathbb{R}^d$ with $\prod_{i=1}^d|\theta_i|=1$ , and fix any given
   $\lambda_1,\cdots,\lambda_d\in[-(2-\kappa)\sqrt{n},(2-\kappa)\sqrt{n}]$ with $|\lambda_i-\lambda_j|\geq \Delta\sqrt{n}$ whenever $i\neq j$.

   Then there exists some $J\in\{1,\cdots,d\}$ (depending on $A_n$ and $(\theta_1,\cdots,\theta_d)$ but not on $\widetilde{X}:=X_2-X_2'$, where $X_2'$ is an independent copy of $X_2$) such that, for any $$s\in\left(0, \frac{C\cdot \mu_k(\lambda_J)}{{\left(\prod_{i=1}^d\mu_1(\lambda_i)\right)^{1/d}}}\right)$$ and for any $k\leq c_0n$, we have
    \begin{equation}\label{212212212}
\mathbb{P}_{\widetilde{X}}\left(\left\|\sum_{i=1}^d \frac{\theta_i}{\mu_1(\lambda_i)}(A_n-\lambda_i I_n)^{-1}\widetilde{X}\right\|_2\leq s
\right) \lesssim s e^{-k},
    \end{equation} where $c_0$ and $C$ are universal constants that depend only on the various parameters $G$,$B$,$\sigma_0$ and $K$.
\end{lemma}

Before providing the proof we provide several essential explanations. First, the label $J$ in $\mu_k(\lambda_J)$ is not truly important to us and we will have very good control of it for any $J$, at any high power. However, the appearance of the $d$-th square root of the product of $\mu_1(
\lambda_i)$ is of fundamental importance and our proof breaks down if we replace it by $\mu_1(\lambda_i)$ for some fixed $i$. Second, we find it convenient to normalize the $(\theta_1,\cdots,\theta_d)$ via $\prod_i|\theta_i|=1$ rather than the more standard one of $\|\theta\|_2=1$, as our choice captures information from all coordinates whereas the latter one loses information when some $\theta_i$ is very small. Last, we are not very careful in the power of $s$. One might expect a $s^d$ power, but this is very hard to prove (and may not be true) and, more importantly, our decoupling procedure loses a power $1/2$ so that even getting a $d$-th power is not useful. We will use a different method to tame the singularity in $s$ in a forthcoming integration \eqref{integration97000}.

To illustrate the main ideas, we first recall how to compute this expectation with only one term $\lambda_1=0$. This is the content of \cite{campos2024least}, Chapter 9. Rephrased in the setting here, we first prove that for any $s\in(e^{-cn},\mu_k/\mu_1)$ (where we abbreviate $\mu_k:=\mu_k(0)$),
$$\mathbb{P}_{\widetilde{X}}\left(\|A^{-1}\widetilde{X}\|_2\leq s\mu_1\right)\lesssim se^{-ck}.$$

The idea is to expand 
$$\|A^{-1}\widetilde{X}|_2^2=\sum_{j=1}^n \mu_j^2 \langle v_j,\widetilde{X}\rangle 
^2$$
where $v_j$ is the eigenvector of $A^{-1}$ corresponding to $\mu_j$.

Then 
$$\mathbb{P}_{\widetilde{X}}(\|A^{-1}\widetilde{X}\|_2\mu_1^{-1}\leq s)\leq\mathbb{P}_{\widetilde{X}}\left(|\langle v_1,\widetilde{X}\rangle|\leq s,\quad \sum_{j=2}^k\frac{\mu_j^2}{\mu_1^2}\langle v_j,\widetilde{X}\rangle^2\leq s^2\right).$$
By the assumption on $\mu_k/\mu_1$, we further deduce that
$$
\mathbb{P}_{\widetilde{X}}(\|A^{-1}\widetilde{X}\|_2\mu_1^{-1}\leq s)\leq\mathbb{P}_{\widetilde{X}}\left(|\langle v_1,\widetilde{X}\rangle|\leq s,\quad \sum_{j=2}^k\langle v_j,\widetilde{X}\rangle^2\leq 1\right).
$$
The last term can be estimated via the novel inverse Littlewood-Offord inequality in \cite{campos2024least}.

We face new challenges as we have $d$ different $\lambda_i$'s. We introduce the following notation:

\begin{Definition}\label{definition4.2}
    Fix any distinct $\lambda_1,\cdots,\lambda_d\in\mathbb{R}$. For any $i\neq j\in \{1,\cdots,d\}$ and any $k\geq 1$ we denote by 
    $$
\mu_{c_j(i;k)}(\lambda_j)
    $$ the $c_j(i,k)$-th largest singular value of $(A_n-\lambda_j I_n)^{-1}$ whose corresponding singular vector coincides with the singular vector of $\mu_k(\lambda_i)$, the $k$-th largest singular value of $(A_n-\lambda_i I_n)^{-1}.$ 
\end{Definition} This definition makes sense because $(A_n-\lambda_i I_n)^{-1}$ and $(A_n-\lambda_j I_n)^{-1}$ share the same set of singular vectors.

We need this notation to record how singular values labeled with respect to $\lambda_j$ will affect the calculation of singular values labeled with respect to $\lambda_i$.  

The proof of Lemma \ref{lemma6.61} makes crucial use of the assumption that the random vector $X_2$ has a bounded density with respect to the Lebesgue measure. We will use this property via the following lemma taken from \cite{livshyts2016sharp}, Theorem 1.1.

\begin{lemma}\label{lemma5.1}
    Let $X_1,\cdots,X_n$ be a family of independent real valued random variables having densities bounded by $K$ with respect to the Lebesgue measure. Denote by $X=(X_1,\cdots,X_n)$. For any $\ell\leq n$ and any $V\in\mathbb{R}^{\ell\times n}$, $VX$ is an $\mathbb{R}^\ell$-valued random variable with density bounded by 
    $$
\frac{e^{\ell/2}K^\ell}{\det(VV^T)^{1/2}}.
    $$
\end{lemma}
In our application we will take the $\ell$-rows of $V$ to be the $\ell$ orthonormal eigenvectors of $A_n$. This lemma replaces the much more delicate Littlewood-Offord type arguments that work for discrete random variables.

Now we are ready to prove Lemma \ref{lemma6.61}.

\begin{proof}[\proofname\ of Lemma \ref{lemma6.61}]

We first expand the inner product in the following form, where $v_k(\lambda_i)$ denotes the singular vector associated with the singular value $\mu_k(\lambda_i):$
\begin{equation}\label{mainmainmian}\begin{aligned}&
\left\|\sum_{i=1}^d \frac{\theta_i}{\mu_1(\lambda_i)}(A_n-\lambda_i I_n)^{-1}\widetilde{X}\right\|_2=\sum_{i=1}^d \left|\theta_i+\sum_{j\neq i}\theta_j\frac{\mu_{c_j(i;1)(\lambda_j)}}{\mu_1(\lambda_j)}\right|^2 \langle \widetilde{X},v_1(\lambda_i)\rangle^2
\\&+\sum_{i=1}^d\sum_{k\in\mathcal{A}(i),k\geq 2} \left|\theta_i \frac{\mu_k(\lambda_i)}{\mu_1(\lambda_i)}+\sum_{j\neq i} \theta_j \frac{\mu_{c_j(i;k)(\lambda_j)}}{\mu_1(\lambda_j)}\right|^2\langle \widetilde{X},v_k(\lambda_i)\rangle ^2
\end{aligned}.\end{equation}
    Here $\mathcal{A}(i)\subset \{1,\cdots,n\}$ are index sets for which we give a descriptive construction as follows. We arrange without loss of generality that $\lambda_1<\lambda_2<\cdots<\lambda_d$ on the real line, then expand the inner product and write the corresponding weights on the site of the eigenvector. The right hand sides of the first line of \eqref{mainmainmian} are the hard constraints, corresponding to singular vectors that are associated with the least singular value of $A_n-\lambda_i I_n$ for some $i$. This will be our major focus. The second line of \eqref{mainmainmian} are the soft constraints, involving singular vectors that do not correspond to the least singular values of any $A_n-\lambda_i I_n$. For eigenvectors of $A_n$ whose eigenvalues are close to $\lambda_i$\footnote{Recall again that $A_n$ and $A_n-\lambda_i I_n$ share the same eigenvectors and singular vectors}, we collect these sums into the sum over $i$ part on the second line of \eqref{mainmainmian} and define $\mathcal{A}(i)\subset [2,cn]\cap\mathbb{Z}$ for some $c>0$ and each $i$. This procedure can exhaust all eigenvectors of $A_n$ into $d$ disjoint sums: those close to $\lambda_i$ are taken into the $i$-th sum, and we take an arbitrary separation between $\lambda_i$ and $\lambda_{i+1}$ deciding on which terms will be sent to the $i$-th sum and which will be sent to the $i+1$-th sum. The size of $\mathcal{A}(i)$ is now very flexible.

As can be seen from the expression \eqref{mainmainmian}, the major complication is that there will be many cancellations and we need to know which term is dominant and what is left after all the additions and taking the absolute value. 

For this purpose, we single out two special indices. Let $I\in\{1,\cdots,d\}$ be such that $$|\theta_I|=\max\{|\theta_1|,\cdots,|\theta_d|\}$$ and let $J\in\{1,\cdots,d\}$ be such that 
$$
\frac{|\theta_J|}{\mu_1(\lambda_J)}=\max\left\{\frac{|\theta_1|}{\mu_1(\lambda_1)},\cdots,\frac{|\theta_d|}{\mu_1(\lambda_d)}\right\}.
$$

First, by Corollary \ref{corollary4.999}, we may assume that on an event with probability $1-e^{-\Omega(n)}$, we have 
$$\frac{\mu_{c(I;1)(\lambda_j)}}{\mu_1(\lambda_j)}\leq 10^{-2}d^{-1}$$ for any $j\neq I$, so that by definition of $I$,
$$
\left|\theta_I+\sum_{j\neq I}\theta_j\frac{\mu_{c_j(i;1)(\lambda_j)}}{\mu_1(\lambda_j)}\right|\geq 0.9|\theta_I|\geq 0.9.
$$

Second, according to the second part of Corollary 
\ref{corollary4.88}, in an event with probability $1-e^{-\Omega(n)}$, we may find some $c_0>0$ depending only on $\kappa$, $B$ and the log-Sobolev constant; thus, for any $k\leq c_0n$, we have for any $j\neq J$,
$$
\mu_k(\lambda_J)\geq 10d\mu_{c_j(J;k)}(\lambda_j).
$$
By definition of $J$, this implies that for any $k\leq c_0n$ and any $j\neq J$,
$$
\theta_J\frac{\mu_k(\lambda_J)}{\mu_1(\lambda_J)}\geq 10d \theta_j\frac{\mu_{c_j(J,k)}(\lambda_j)}{\mu_1(\lambda_j)}.
$$
That is, for any $k\leq c_0n$,
$$
\left|\theta_J\frac{\mu_k(\lambda_J)}{\mu_1(\lambda_J)}+\sum_{j\neq J}\theta_j\frac{\mu_{c_j(J,k)}(\lambda_j)}{\mu_1(\lambda_j)}\right|\geq 0.9\left|\theta_J\frac{\mu_k(\lambda_J)}{\mu_1(\lambda_J)}\right|.
$$

Applying both bounds to \eqref{mainmainmian}, we obtain
\begin{equation}\label{}\begin{aligned}&
\left\|\sum_{i=1}^d \frac{\theta_i}{\mu_1(\lambda_i)}(A_n-\lambda_i I_n)^{-1}\widetilde{X}\right\|_2\geq 0.9  \langle \widetilde{X},v_1(\lambda_I)\rangle^2
 +0.9\sum_{j=2}^{c_0n} \left|\theta_J\frac{\mu_j(\lambda_J)}{\mu_1(\lambda_J)}\right|^2\langle \widetilde{X},v_j(\lambda_J)\rangle ^2.
\end{aligned}\end{equation}

 We are now in the position to apply Lemma \ref{lemma5.1}. By assumption $\widetilde{X}$ has bounded density. In the setting of Lemma \ref{lemma5.1}, let $V$ be a matrix with (orthogonal) rows given by $v_1(\lambda_I)$ and $v_k(\lambda_J)$, $k=2,\cdots,c_0n$.  

Then we have\begin{equation}\label{finsexpan}
\mathbb{P}_{\widetilde{X}}\left(\left\|\sum_{i=1}^d \frac{\theta_i}{\mu_1(\lambda_i)}(A_n-\lambda_i I_n)^{-1}\widetilde{X}\right\|_2\leq s
\right) \lesssim s e^{-k},\end{equation}
for any $0\leq s\leq C |\theta_J|\mu_k(\lambda_J)/\mu_1(\lambda_J)$ and any $k\leq c_0n$, where $C$ is a constant depending on the density of $X_2$. This is because the event in \eqref{finsexpan} can be decomposed as requiring 
$$\langle \widetilde{X},v_1(\lambda_i)\rangle \lesssim s,\quad\text{ and } \sum_{j=2}^{k} \left\|\theta_J\frac{\mu_j(\lambda_J)}{\mu_1(\lambda_J)}\right|^2\langle \widetilde{X},v_j(\lambda_J)\rangle ^2\lesssim s^2$$ for some appropriate $k\leq c_0n$.

The constant $C$ in the restriction on $s$ appears because we need a sufficiently small constant factor depending on the density of $X$ to overcome the $e^{^\ell/2}K^\ell$ factor in Lemma \ref{lemma5.1} and obtain the exponentially small factor $e^{-k}$. We may typically take $C\leq \frac{1}{4\sqrt{e}K}$ where $K$ is an upper bound on the density of $\widetilde{X}_2$.

Finally, by definition of $J$ and our normalization on $\theta$, we have 
$$
|\theta_J|\frac{\mu_k(\lambda_J)}{\mu_1(\lambda_J)}\geq \mu_k(\lambda_J)\left(\prod_{j=1}^d|\theta_j|\frac{1}{\mu_1(\lambda_j)}\right)^{1/d}= \frac{\mu_k(\lambda_J)}{\left(\prod_{j=1}^d \mu_1(\lambda_j)\right)^{1/d}}.
$$

This completes the proof, and remarkably, our final bound does not depend on $\theta$ saved for the label $J$ in $\mu_k(\lambda_J)$.

Checking the whole proof, there are two instances where we consider an event on $A_n$ that holds with probability $1-e^{-\Omega(n)}$. It suffices to take $\mathcal{E}_1$ the intersection of these two events.

\end{proof}

We will also need the following lemma which is very similar to Lemma \ref{lemma6.61}, but addresses a different geometric configuration.

\begin{lemma}\label{lemma6.612}
 Assume that $A_n$ satisfies assumptions in Theorem \ref{Theorem1.1} and $X_2\sim\operatorname{Col}_n(\zeta_2)$, $\zeta_2\in\Gamma_2(G,K,\sigma_0)$. Fix any $\kappa>0$ and $\Delta>0$ that are sufficiently small. Then there exists an event $\mathcal{E}_2$ with $\mathbb{P}(\mathcal{E}_2)\geq 1-\exp(-\Omega(n))$ such that for any $A\in\mathcal{E}_2$, the following holds:

   Fix any given $(\theta_1,\cdots,\theta_d)\in\mathbb{R}^d$, assume that we can find some $\mathcal{A}\subset\{1,\cdots,d\}$ satisfying: $|\mathcal{A}|=\ell$ for some $1\leq \ell\leq d-1$, that $\prod_{i\in\mathcal{A}}|\theta_i|=1$ and for any $j\notin\mathcal{A}$, $|\theta_j|<1$. Fix any given
   $\lambda_1,\cdots,\lambda_d\in[-(2-\kappa)\sqrt{n},(2-\kappa)\sqrt{n}]$ satisfying $|\lambda_i-\lambda_j|\geq \Delta\sqrt{n}$ whenever $i\neq j$. 
   
   Then there exist some $J\in\{1,\cdots,d\}$ (depending on $A_n$ and $(\theta_1,\cdots,\theta_d)$ but not on $\widetilde{X}$) such that, for any $$s\in\left(0, \frac{C\cdot \mu_k(\lambda_J)}{{\left(\prod_{i\in\mathcal{A}}\mu_1(\lambda_i)\right)^{1/\ell}}}\right)$$ and for any $k\leq c_0n$, we have
    \begin{equation}
\mathbb{P}_{\widetilde{X}}\left(\left\|\sum_{i=1}^d \frac{\theta_i}{\mu_1(\lambda_i)}(A_n-\lambda_i I_n)^{-1}\widetilde{X}\right\|_2\leq s
\right) \lesssim s e^{-k},
    \end{equation} where $c_0$ and $C$ are universal constants that depend only on the various parameters $G$,$B$,$\sigma_0$ and $K$.
\end{lemma}

\begin{proof}
    We follow exactly the same line of proof as Lemma \ref{lemma6.61} (defining $I$ and $J$ in the same way and performing the same estimates). The only places to change are at the end: we use $|\theta_I|\geq 1$  and that 
    $$
|\theta_J|\frac{\mu_k(\lambda_J)}{\mu_1(\lambda_J)}\geq \mu_k(\lambda_J)\left(\prod_{j\in\mathcal{A}}|\theta_j|\frac{1}{\mu_1(\lambda_j)}\right)^{1/\ell}\geq \frac{\mu_k(\lambda_J)}{\left(\prod_{j\in\mathcal{A}}\mu_1(\lambda_j)\right)^{1/\ell}}.
$$
\end{proof}

\section{Initial estimate: two locations}

In the next three sections we assume for the sake of simplicity that $d=2$, i.e. we have only two locations $\lambda_1\neq \lambda_2$. This approach will greatly simplify the computations and make the main ideas more transparent. The main result of this chapter is the following lemma:

\begin{lemma}\label{lemma6.111} Let $A_n$ satisfy the assumptions of Theorem \ref{Theorem1.1}. Fix any $\kappa>0$, $\Delta>0,$ and any two $\lambda_1,\lambda_2\in[(-(2-\kappa)\sqrt{n},(2-\kappa)\sqrt{n}]$ with $|\lambda_1-\lambda_2|\geq \Delta\sqrt{n}$. Then for any $\delta_1,\delta_2\geq e^{-cn}$, any $u\in\mathbb{R}^n$ with $\|u\|_2\leq 2$, and any $p>0$,
\begin{equation}\begin{aligned}
&\mathbb{E}_{A_n}\sup_{r_1,r_2}\mathbb{P}_X\left(
    \frac{|\langle (A_n-\lambda_i I_n)^{-1}X,X\rangle-r_i|}{\|(A_n-\lambda_i I_n)^{-1}\|_*}\leq\delta_i
    ,\langle X,u\rangle\geq s,\frac{\mu_1(\lambda_1)\mu_1(\lambda_2)}{n}\leq(\delta_1\delta_2)^{-p}
    \right)\\&\lesssim e^{-s}\delta_1\delta_2+e^{-s}\delta_1\delta_2\mathbb{E}\left[\left(\frac{\mu_1(\lambda_1)\mu_1(\lambda_2)}{n}\right)^{9/10}
    \bigm| \frac{\mu_1(\lambda_1)\mu_1(\lambda_2)}{n}\leq (\delta_1\delta_2)^{-p}\right]^{80/81}+e^{-\Omega(n)},\end{aligned}
\end{equation}
    where $c$ depends on various constants in the assumption of Theorem \ref{Theorem1.1}, on $\kappa$ and $\Delta$.
\end{lemma}
In future applications we will set $p>1$ to be a constant sufficiently close to 1.

\subsection{Initial decomposition}
We now explicitly compute the integral of $I(\theta)$ over the specified region. Thanks to Theorem \ref{theorem3.1}, this is the major step in estimating the least singular values of $A_n$.

We first apply Hölder's inequality and get \begin{equation}\label{holders}
I(\theta)\lesssim \left(\mathbb{E}_{X_2,X_2'}e^{9\langle X_2+X_2',u\rangle}\right)^{1/9}\left(\mathbb{E}_{\widetilde{X}}e^{-c''\|\sum_{i=1}^2 \frac{\theta_i}{\mu_1(\lambda_i)}(A_n-\lambda_i I_n)^{-1}\widetilde{X}\|_2^2}
\right)^\frac{8}{9}.
\end{equation}
The first term on the right is $O(1)$. We can deduce the remaining estimates as follows: 
$$\begin{aligned}
&I(\theta)^{9/8}\lesssim_B \mathbb{E}_{\widetilde{X}}e^{-c''\|\sum_{i=1}^2 \frac{\theta_i}{\mu_1(\lambda_i)}(A_n-\lambda_i I_n)^{-1}\widetilde{X}\|_2^2}\\&\lesssim 
\begin{cases} 1& |\theta_1|<1,|\theta_2|<1,
\\e^{-c''|\theta_1|^{1/5}}+\mathbb{P}_{\widetilde{X}}\left(\left\|\sum_{i=1}^2 \frac{\theta_i}{\mu_1(\lambda_i)}(A_n-\lambda_i I_n)^{-1}\widetilde{X}\right\|_2\leq |\theta_1|^{\frac{1}{10}}\right),& |\theta_1|\geq 1,|\theta_2|< 1\\
e^{-c''|\theta_2|^{1/5}}+\mathbb{P}_{\widetilde{X}}\left(\left\|\sum_{i=1}^2 \frac{\theta_i}{\mu_1(\lambda_i)}(A_n-\lambda_i I_n)^{-1}\widetilde{X}\right\|_2\leq |\theta_2|^{\frac{1}{10}}\right),& |\theta_1|< 1,|\theta_2|\geq 1\\
e^{-c''|\theta_1\theta_2|^{1/11}}+\mathbb{P}_{\widetilde{X}}\left(\left\|\sum_{i=1}^2 \frac{\theta_i}{\mu_1(\lambda_i)}(A_n-\lambda_i I_n)^{-1}\widetilde{X}\right\|_2\leq |\theta_1\theta_2|^{\frac{1}{22}}\right)& |\theta_1|\geq 1,|\theta_2|\geq 1
\end{cases}
\end{aligned}
$$

We split the integration of $I(\theta)^{1/2}$ over the specified region into three terms, noticing that the four functions $1$, $e^{-c''|\theta_1|^{1/5}}$, $e^{-c''|\theta_2|^{1/5}}$ and $e^{-c''|\theta_1\theta_2|^{1/11}}$ integrate into $O(1)$ on the respective regions:
\begin{equation}\label{decompositionformula}
\begin{aligned}
&\int_{\sum_i|\theta_i\delta_i|^2\leq 1}I(\theta)^{1/2}d\theta\\&
\lesssim 1+\int_{|\theta_1|\leq 1,1\leq|\theta_2|\leq\delta_2^{-1}}\mathbb{P}_{\widetilde{X}}(\|\sum_{i=1}^2 \frac{\theta_i}{\mu_1(\lambda_i)}(A_n-\lambda_i I_n)^{-1}\widetilde{X}\|\leq  |\theta_2|^\frac{1}{10})^{4/9}d\theta\\&
+ \int_{1\leq |\theta_1|\leq \delta_1^{-1},|\theta_2|\leq 1}\mathbb{P}_{\widetilde{X}}(\|\sum_{i=1}^2 \frac{\theta_i}{\mu_1(\lambda_i)}(A_n-\lambda_i I_n)^{-1}\widetilde{X}\|\leq |\theta_1|^\frac{1}{10})^{4/9}d\theta\\&+
\int_{1\leq |\theta_1|\leq \delta_1^{-1},1\leq|\theta_2|\leq\delta_2^{-1}}\mathbb{P}_{\widetilde{X}}(\|\sum_{i=1}^d \frac{\theta_i}{\mu_1(\lambda_i)}(A_n-\lambda_i I_n)^{-1}\widetilde{X}\|\leq (\theta_1\theta_2)^\frac{1}{22})^{4/9}d\theta\\&:=1+I_1+I_2+I_3.
\end{aligned}\end{equation}

Informally speaking, the integrals $I_1$ and $I_2$ represent terms that rely essentially on one location $\lambda_i$, and the integral $I_3$ consists of terms that rely genuinely on both locations. 

To estimate $I_1$ and $I_2$ we can simply reduce
ourselves to the $d=1$ setting, while estimating $I_3$ requires fundamentally different ideas and will be the major technical difficulty.

\subsection{Two preparatory results}

We recall the following simple linear algebra property: for any $n\times n$ matrix $M_1,\cdots,M_d$,
\begin{equation}\label{productsingular}
\sigma_{min}(M_1\cdots M_d)\geq\sigma_{min}(M_1)\cdots\sigma_{min}(M_d).
\end{equation}

Thanks to this lemma, we prove the following:
\begin{lemma}\label{singularvalueproduct} Given any $\sigma>0$ and any $\tau\in(0,1],$ let $M_1,\cdots,M_d$ be $n\times n$ random matrices (not necessarily independent) that satisfy, for any $\delta_1,\cdots,\delta_d>0$,
\begin{equation}\label{assumption}
    \mathbb{P}(\sigma_{min}(M_1)\leq\delta_1 n^{-1/2},\cdots, \sigma_{min}(M_d)\leq\delta_d n^{-1/2})\leq (C\delta_1\cdots\delta_d)^\tau+e^{-cn^\sigma},
\end{equation}
and more generally for any $k=1,2,\cdots,d-1$ and any $1\leq i_1<i_2<\cdots<i_k\leq d$, we have 
\begin{equation}\label{assumption2}
    \mathbb{P}(\sigma_{min}(M_{i_1})\leq\delta_{i_1} n^{-1/2},\cdots, \sigma_{min}(M_{i_{k}})\leq\delta_{i_k} n^{-1/2})\leq C(\delta_{i_1}\cdots\delta_{i_k})^\tau+e^{-cn^\sigma}.
\end{equation}
Then for any $\delta>0$, we have  
\begin{equation}
    \mathbb{P}\left((\sigma_{min}(M_1\cdots  M_d))^{1/d}\leq \delta n^{-1/2}\right)\leq C(\delta^\tau\log(\delta^{-1}))^d+e^{-cn^\sigma},
\end{equation}  and we have
$$
\mathbb{P}\left((\sigma_{min}(M_1)\cdots \sigma_{min}(M_d))^{1/d}\leq \delta n^{-1/2}\right)\leq C(\delta^\tau\log(\delta^{-1}))^d+e^{-cn^\sigma},
$$
where $C$ and $c$ are universal constants.
\end{lemma}

\begin{proof}
    Assume that $\frac{1}{2}\geq \delta\geq e^{-cn^\sigma}$; otherwise, the claim is trivial. 
    
    We decompose $[\delta^d,1]$ into dyadic intervals 
    $$[\delta^d,1]=\cup_{j=1}^{L_\delta} [I_{j,S},I_{j,L}]
    $$ where for each $j\geq 2$, we assume $I_{j,L}/I_{j,S}=2$ and we assume  $1\leq I_{1,L}/I_{1,S}\leq 2$. Then it is elementary to check that $L_\delta\leq C_d\log\delta^{-1}$.

We consider three separate cases: (a) either for some $i$, $\sigma_{min}(M_i)\geq n^{-1/2}$, then the event $\sigma_{min}(M_1\cdots M_d)\leq\delta^d n^{-d/2}$ implies $\sigma_{min}(M_1)\cdots\sigma_{min}(M_d)\leq \delta^d n^{-d/2}$, so that $\prod_{j\neq i}\sigma_{min}(M_j)\leq \delta^d n^{-(d-1)/2}$. Thus we can safely discard all subscripts $i$ such that $\sigma_{min}(M_i)\geq n^{-1/2}$, and we get to one of the following two cases.

Then consider case (b): for some $i$ we have $\sigma_{min}(\lambda_i)
\leq\delta^{d}$. The possibility of this event can be bounded via assumption \eqref{assumption2}, which directly completes the proof.

Then, we consider cases (c), where case (a) and (b) do not occur.

The event $\sigma_{min}(M_1\cdots M_d)\leq \delta^d n^{-d/2}$ implies $\sigma_{min}(M_1)\cdots\sigma_{min}(M_d)\leq \delta^d n^{-d/2}$ by \eqref{productsingular}, which is contained in the union of $L_\delta^d$ many events, each given by 
$$
\sigma_{min}(M_1)\leq I_{j_1,L}n^{-1/2},\cdots \sigma_{min}(M_1)\leq I_{j_d,L}n^{-1/2},\quad 
\prod_{i=1}^d I_{j_i,L}\leq 2^d \delta^d.
$$ This can be checked via considering all events $\sigma_{min}(M_i)\leq x_in^{-1/2},\cdots,$ $i=1,\cdots,d$, and finding the interval $[I_{j_i,s},I_{j_i,L}]$ where $x_i$ lies. By assumption of case (c), $x_i$ must lie in one of these events for each $i$, so the claim holds.
Then the lemma follows from taking a union bound and applying assumption \ref{assumption}.
\end{proof}

In our initial iteration, we will make direct use of the optimal singular value estimates of subGaussian symmetric matrices in \cite{campos2024least}. We will need a version of such estimates at any location in the bulk, while \cite{campos2024least} states the singular value estimate at $0$. Indeed, it is not difficult to generalize the main result of \cite{campos2024least} to these locations, and we state it as follows:

\begin{Proposition}\label{proposition6.666} Consider random matrix $A_n$ satisfying the assumption of Theorem \ref{Theorem1.1}, Fix any $\kappa>0$, then the following estimate holds uniformly for all $\lambda\in[-(2-\kappa)\sqrt{n},(2-\kappa)\sqrt{n}]$: for any $\epsilon>0$,
    \begin{equation}
    \mathbb{P}(    \sigma_{min}(A_n-\lambda I_n)\leq\epsilon n^{-1/2})\leq c\epsilon+e^{-cn},
    \end{equation} where $c$ only depends on $B$ and $\kappa$.
\end{Proposition}
We believe this generalization is very straightforward and that there is no need to write out a proof. Essentially, one only needs to repeat every step of the proof in \cite{campos2024least} and one needs only to use the local law estimate at generic locations (see Section \ref{section4fuck}) to replace the one at $0$. A subtlety is that we assume that the diagonal elements $A_{ii}$ have a different variance but this causes no harm as we assume entries have a continuous distribution and do not invoke the Littlewood-Offord theorem. Alternatively, the rest of this paper, taking only one location $d=1$, can be regarded as a proof of this proposition, so we do not bother to write it out again.

\subsection{Estimating essentially 1-d terms} 
The computations in this section are very similar to those in \cite{campos2024least}, Section 9 and Section 10, reflecting that $I_1$ and $I_2$ are essentially 1-d objects.

We begin with the $I_1$ term, recall that 
$$I_1=\int_{|\theta_1|\leq 1,1\leq|\theta_2|\leq\delta_2^{-1}}\mathbb{P}_{\widetilde{X}}(\|\sum_{i=1}^2 \frac{\theta_i}{\mu_1(\lambda_i)}(A_n-\lambda_i I_n)^{-1}\widetilde{X}\|\leq  |\theta_2|^\frac{1}{10})^{4/9}d\theta.$$ We bound 
$$\begin{aligned}
I_1&=\int_{|\theta_1|\leq 1,1\leq|\theta_2|\leq\delta_2^{-1}}\mathbb{P}_{\widetilde{X}}\left(\|
\sum_{i=1}^2\frac{\theta_i/\theta_2}{\mu_1(\lambda_i)}(A_n-\lambda_iI_n)^{-1}\widetilde{X}
\|\leq|\theta_2|^{-\frac{9}{10}}\right)^{4/9}d\theta\\&
\lesssim \int_1^{\delta_2^{-1}}\sup_{|\theta_1|\leq 1,|\theta_2|=1}\mathbb{P}_{\widetilde{X}} \left(\|
\sum_{i=1}^2\frac{\theta_i}{\mu_1(\lambda_i)}(A_n-\lambda_iI_n)^{-1}\widetilde{X}
\|\leq s^{-\frac{9}{10}}\right)^{4/9}ds\\&
\leq\int_{\delta_2}^1 s^{-19/9} \sup_{|\theta_1|\leq 1,|\theta_2|=1}\mathbb{P}_{\widetilde{X}} \left(\|
\sum_{i=1}^2\frac{\theta_i}{\mu_1(\lambda_i)}(A_n-\lambda_iI_n)^{-1}\widetilde{X}
\|\leq s\right)^{4/9} ds.
\end{aligned}$$

Now we are ready to apply Lemma \ref{lemma6.612}. First, we note that the conclusion of \ref{lemma6.612} is independent of $\theta$, except for the subscript $J\in\{1,2\}$ which depends only on $A_n$ and the relative ratios of $\theta_1$ and $\theta_2$. To remedy this, in the following we take both cases into account, i.e. we decompose the range of $s$ both by taking $J=1$ and $J=2$. To put things on a more concrete ground, let $\sum_1$ and $\sum_2$ be subsets of the parameter space $(\theta_1,\theta_2)$, $|\theta_1|\leq 1$ and $|\theta_2|=1$, such that on $\sum_\ell$, Lemma \ref{lemma6.612} yields $J=\ell$. Then, we write 
$$
I_1\lesssim \sum_{\ell=1}^2 \int_{\delta_2}^1 s^{-19/9} \sup_{(\theta_1,\theta_2)\in\sum_\ell}\mathbb{P}_{\widetilde{X}} \left(\|
\sum_{i=1}^2\frac{\theta_i}{\mu_1(\lambda_i)}(A_n-\lambda_iI_n)^{-1}\widetilde{X}
\|\leq s\right)^{4/9} ds.
$$

We decompose $[\delta_2,1]=[\delta_2,\frac{\mu_{c_0n}(\lambda_2)}{\mu_1(\lambda_2)}]\cup \cup_{k=2}^{cn}[\frac{\mu_k(\lambda_2)}{\mu_1(\lambda_2)},\frac{\mu_{k-1}(\lambda_2)}{\mu_1(\lambda_2)}]$ when $(\theta_1,\theta_2)\in\sum_2$ and we decompose $[\delta_2,1]=[\delta_2,\frac{\mu_{c_0n}(\lambda_1)}{\mu_1(\lambda_2)}]\cup \cup_{k=2}^{cn}[\frac{\mu_k(\lambda_1)}{\mu_1(\lambda_2)},\frac{\mu_{k-1}(\lambda_1)}{\mu_1(\lambda_2)}]$ when $(\theta_1,\theta_2)\in\sum_1$. (When $\delta_2$ is not small enough so that some of these intervals are not well-defined, we simply neglect these intervals: this approach will create no loss as shown in the forthcoming proofs. Additionally, when $\mu_k(\lambda_1)/\mu_1(\lambda_2)>1$ for some $k$, we also neglect such intervals without loss). Then by Lemma \ref{lemma6.612}, we have
\begin{equation}\label{whatisi1}
\begin{aligned}
I_1&\lesssim\sum_{\ell=1}^2\sum_{k=1}^{cn}\int_{\mu_{k}(\lambda_\ell)/\mu_1(\lambda_2)}^{\mu_{k-1}(\lambda_\ell)/\mu_1(\lambda_2)} e^{-cn} s^{-19/9+4/9}ds+\int_\delta^{\mu_{cn}(\lambda_\ell)/\mu_1(\lambda_2)} e^{-cn}s^{-19/9+4/9} ds\\&
\leq \sum_{\ell=1}^2 \sum_{k=1}^{cn}e^{-k}(\mu_1(\lambda_2)/\mu_k(\lambda_{\ell}))^{2/3}+2e^{-cn}\delta^{-3/2}
\end{aligned}\end{equation}
 where the last term is $e^{-\Omega(n)}$ when $\delta>e^{-cn}$.
Similarly, we have
\begin{equation}\label{whatisi2}
    I_2\lesssim \sum_{\ell=1}^2 \sum_{k=1}^{cn}e^{-k}(\mu_1(\lambda_1)/\mu_k(\lambda_{\ell}))^{2/3}+e^{-\Omega(n)}.
\end{equation}

Before concluding, we prove the following very important estimate, whose utility will be clear afterwards. The proof incorporates both the eigenvalue rigidity estimates in Section \ref{section4fuck} and the one-location singular value estimate in Proposition \ref{proposition6.666}, essentially proven in \cite{campos2024least}.

\begin{lemma}\label{whatislemma1} For any $B>0,$ let $\zeta\in\Gamma_B$ have a finite log-Sobolev constant, and $A_n\sim\operatorname{Sym}_n(\zeta)$.
For any $\delta_1,\delta_2\geq e^{-cn}$, define $$\mathcal{E}(\delta_1\delta_2):=\left\{\frac{\mu_1(\lambda_1)}{\sqrt{n}}\leq (\delta_1\delta_2)^{-1},\quad \frac{\mu_1(\lambda_2)}{\sqrt{n}}\leq (\delta_1\delta_2)^{-1}\right\}.$$     Then for any $p>0$,
\begin{equation}\label{6.3wowsog}
\mathbb{E}_{A_n}^{\mathcal{E}(\delta_1^p\delta_2^p)}\left[ \prod_{r=1}^2 \frac{\|(A_n-\lambda_rI_n)^{-1}\|_*}{\mu_1(\lambda_r)}\sum_{\ell_1=1}^2 \sum_{\ell_2=1}^2\sum_{k=1}^{cn}e^{-k}\left(\frac{\mu_1(\lambda_{\ell_1})}{\mu_k(\lambda_{\ell_2})}\right)^{2/3}+e^{-\Omega(n)}\right]\lesssim 1.
\end{equation}
\end{lemma}

\begin{proof}  Let $\mathcal{E}_0:=\mathcal{E}_1\cap\mathcal{E}_2\cap\mathcal{E}(\delta_1^p\delta_2^p)$, where $\mathcal{E}_1$ and $\mathcal{E}_2$ are defined in Lemma \ref{lemma6.61} and \ref{lemma6.612}. We may safely work with $\mathcal{E}_0$ since the two events added have a probability being $1-e^{-\Omega(n)}$.
    For any fixed $\ell_1,\ell_2$ we first estimate via Hölder's inequality
    $$\begin{aligned}
&\mathbb{E}_{A_n}^{\mathcal{E}_0}
\prod_{r=1}^2 \frac{\|(A_n-\lambda_rI_n)^{-1}\|_*}{\mu_1(\lambda_r)}\left(\frac{\mu_1(\lambda_{\ell_1})}{\mu_k(\lambda_{\ell_2})}\right)^{2/3}\\&\leq \prod_{r=1}^2\mathbb{E}_{A_n}^{\mathcal{E}_0}\left[\left(\frac{\|(A_n-\lambda_r I_n)^{-1}\|_*}{\mu_1(\lambda_r)}\right)^{28}\right]^{1/28}\\&\quad \cdot\mathbb{E}_{A_n}^{\mathcal{E}_0}\left[(\frac{\sqrt{n}}{k\mu_k(\lambda_{\ell_1})})^{28/3}\right]^{1/14}\mathbb{E}_{A_n}^{\mathcal{E}_0}\left[(\frac{\mu_1(\lambda_{\ell_2})}{\sqrt{n}})^{7/9}
\right]^{6/7} \cdot k^{2/3}
\\&\lesssim k\mathbb{E}_{A_n}^{\mathcal{E}_0}\left[(\frac{\mu_1(\lambda_{\ell_2})}{\sqrt{n}})^{7/9}
\right]^{6/7}
    \end{aligned}$$ where we use Lemma \ref{lemma4.11} and Lemma \ref{lemma4.22} to show that the first two terms are $O(1)$.

Therefore we have shown that 
\begin{equation}
    \text{The left hand side of } \eqref{6.3wowsog} \lesssim \mathbb{E}_{A_n}^{\mathcal{E}_0}\left[(\frac{\mu_1(\lambda_{1})}{\sqrt{n}})^{7/9}+(\frac{\mu_1(\lambda_{2})}{\sqrt{n}})^{7/9}\right]+e^{-\Omega(n)}.
\end{equation}

    The final estimate is performed using Proposition \ref{proposition6.666}. We bound
    $$\begin{aligned}
\mathbb{E}_{A_n}^{\mathcal{E}_0}[(\frac{\mu_1(\lambda_{1})}{\sqrt{n}})^{7/9}]&\lesssim 1+\int_{1}^{(\delta_1\delta_2)^{-7p/9}}\mathbb{P}(\mu_1(\lambda_1\leq s^{9/7}\sqrt{n})ds\\&\lesssim 1+\int_1^{(\delta_1\delta_2)^{-7p/9}} (s^{-9/7}+e^{-cn}) ds\lesssim 1,
    \end{aligned}$$ where the second line follows from Proposition \ref{proposition6.666}. The same computation holds for $\mu_1(\lambda_2)$, thus completing the proof of the lemma.
\end{proof}

\subsection{Estimating genuinely 2-d terms}

In this section we bound $I_3$. Recall that 
$$I_3=\int_{1\leq |\theta_1|\leq \delta_1^{-1},1\leq|\theta_2|\leq\delta_2^{-1}}\mathbb{P}_{\widetilde{X}}\left(\|\sum_{i=1}^2 \frac{\theta_i}{\mu_1(\lambda_i)}(A_n-\lambda_i I_n)^{-1}\widetilde{X}\|\leq (\theta_1\theta_2)^\frac{1}{22}\right)^{4/9}d\theta.$$

First we normalize the probability inside so that $\theta_1\theta_2=1$. That is, 
$$
I_3=\int_{1}^{1/\sqrt{\delta_1\delta_2}}\int_{|\theta_1\theta_2|=s^2,1\leq|\theta_i|}\mathbb{P}_{\widetilde{X}}
\left(\|\sum_{i=1}^2 \frac{\theta_i(A_n-\lambda_i I_n)^{-1}\widetilde{X}}{\mu_1(\lambda_i)\sqrt{\theta_1\theta_2}}\|\leq s^{-\frac{10}{11}}\right)^{4/9}ds.$$

We will use Lemma \ref{lemma6.61} to estimate the probability inside. Again, the bound therein is independent of the precise ratio of $\theta_1,\theta_2$ (after normalizing $\theta_1\theta_2=1$) saved for the subscript $J$. We again define $\sum_1$, $\sum_2$ the complementary subsets of the hyperplane $|\theta_1\theta_2|=1$ so that on $\sum_1$,  Lemma \ref{lemma6.61} yields $J=1$ and on $\sum_2$,  Lemma \ref{lemma6.61} yields $J=2$ (note that $\sum_1,\sum_2$ are also determined by $A_n$). Then we further bound, noting that we now normalize $|\theta_1\theta_2|=1:$
$$
I_3\lesssim \int_1^{1/\sqrt{\delta_1\delta_2}}s\ln s\sum_{\ell=1}^2\sup_{(\theta_1,\theta_2)\in\sum_\ell}\mathbb{P}_{\widetilde{X}}\left(\|\sum_{i=1}^2 \frac{\theta_i}{\mu_1(\lambda_i)}(A_n-\lambda_i I_n)^{-1}\widetilde{X}\|\leq s^{-\frac{10}{11}}\right)^{4/9}ds.
$$ We explain the factor $s\ln s$ here: for any $s>1$ the area of the region $$(\theta_1,\theta_2)\in\mathbb{R}^2:|\theta_1\theta_2|\leq s^2,|\theta_1|\geq 1,|\theta_2|\geq 1$$ can be computed via
$A_2(s):=4\int_1^{s^2}(\frac{s^2}{y}-1)dy=4(s^2\ln s^2-s^2+1),$ and $\frac{d}{ds}A_2(s)=16s\ln s$. Thus we obtain an additional $16s\ln s$ factor after we change the coordinates $(\theta_1,\theta_2)\to s=\sqrt{\theta_1\theta_2}$.

Changing the coordinates the second time, we obtain 
\begin{equation}\label{integration97000}
I_3\lesssim 
\sum_{\ell=1}^2\int_{\sqrt{\delta_1\delta_2}}^1 s^{-32/10}\ln s^{-1}  \sup_{(\theta_1,\theta_2)\in\sum_i} \mathbb{P}_{\widetilde{X}}\left(\|\sum_{i=1}^2 \frac{\theta_i}{\mu_1(\lambda_i)}(A_n-\lambda_i I_n)^{-1}\widetilde{X}\|\leq s\right)^{4/9}ds.
\end{equation}

Now we use lemma \ref{lemma6.61} and decompose, taking $J=1$ on $\sum_1$ and $J=2$ on $\sum_2$: \begin{equation}\label{decomposition2d}[\sqrt{\delta_1\delta_2},1]=\left[\sqrt{\delta_1\delta_2},\frac{\mu_{cn}(\lambda_J)}{\sqrt{\mu_1(\lambda_1)\mu_1(\lambda_2})}\right]\cup \bigcup_{k=2}^{cn}\left[\frac{\mu_{k}(\lambda_J)}{\sqrt{\mu_1(\lambda_1)\mu_1(\lambda_2)}},\frac{\mu_{k-1}(\lambda_J)}{\sqrt{\mu_1(\lambda_1)\mu_1(\lambda_2})}\right].\end{equation} Note: if we already have $\frac{\mu_{k-1}(\lambda_J)}{\sqrt{\mu_1(\lambda_1)\mu_1(\lambda_2})}\geq 1$, then we terminate the right end of the interval at 1; if the whole interval lies on the right hand side of 1 then we discard that interval.

Applying Lemma \ref{lemma6.61}, and using $s^{-3.2+4/9}\log s^{-1}\leq s^{-25/9}$ for $s\leq 1$, we obtain
$$
I_3\lesssim\sum_{\ell=1}^2\sum_{k=2}^{cn}e^{-k}\int_{\frac{\mu_{k}(\lambda_\ell)}{\sqrt{\mu_1(\lambda_1)\mu_1(\lambda_2})}}^{\frac{\mu_{k-1}(\lambda_\ell)}{\sqrt{\mu_1(\lambda_1)\mu_1(\lambda_2})}} s^{-25/9} ds+\sum_{\ell=1}^2 e^{-cn}\int_{\sqrt{\delta_1\delta_2}}^{\frac{\mu_{k-1}(\lambda_\ell)}{\sqrt{\mu_1(\lambda_1)\mu_1(\lambda_2})}}s^{-25/9}ds.
$$
The second term on the right hand side evaluates to $e^{-cn}(\sqrt{
\delta_1\delta_2})^{-16/9},$ which is $e^{-\Omega(n)}$ if we assume $\delta\delta_2\geq e^{-c'n}$ for some $c'>0$. For the first term, we integrate and obtain 
\begin{equation}\label{whatisi3}
I_3\lesssim\sum_{\ell=1}^2\sum_{k=2}^{cn}e^{-k}\left(\frac{\sqrt{\mu_1(\lambda_1)\mu_1(\lambda_2)}}{\mu_k(\lambda_\ell)}\right)^{16/9}+e^{-\Omega(n)}.
\end{equation}
At this point notice that each individual $\mu_i(\lambda_1),i=1,2$ has a power of $\frac{8}{9}<1$, but the overall power is greater than one. In the next section we will use a much finer estimate than simply applying Hölder's inequality, to obtain an effective bound of $I_3$.

In the next lemma we eliminate the $\mu_k(\lambda_\ell)$ term.

\begin{lemma}\label{whatislemma2}
    For any $B>0$, and any $\zeta\in\Gamma_B$ having a finite log-Sobolev constant, consider the random matrix $A_n\sim\operatorname{Sym}_n(\zeta)$. For any $\delta_1,\delta_2\geq e^{-c''n}$ and any $p>0$  define 
    $$\mathcal{E}(p,\delta_1\delta_2):=\left\{\frac{\mu_1(\lambda_1)\mu_1(\lambda_2)}{n}\leq (\delta_1\delta_2)^{-p}\right\}.$$ Then 
    \begin{equation}\begin{aligned}
    \mathbb{E}_{A_n}^{\mathcal{E}(p,\delta_1\delta_2)} &\left[\prod_{r=1}^2 \frac{\|(A_n-\lambda_rI_n)^{-1}\|_*}{\mu_1(\lambda_r)}\sum_{\ell=1}^2\sum_{k=2}^{cn}e^{-k}\left(\frac{\sqrt{\mu_1(\lambda_1)\mu_1(\lambda_2)}}{\mu_k(\lambda_\ell)}\right)^{16/9}
        \right]
        \\&\lesssim 1+\mathbb{E}_{A_n}^{\mathcal{E}(p,\delta_1\delta_2)} \left[\left(\frac{\mu_1(\lambda_1)\mu_1(\lambda_2)}{n}\right)^{9/10}
\right]^{80/81}.
        \end{aligned}
    \end{equation}
\end{lemma}

\begin{proof} Let $\mathcal{E}_{0,p}:=\mathcal{E}_1\cap\mathcal{E}_2\cap \mathcal{E}(p,\delta_1\delta_2)$, and it suffices to take expectation on $\mathcal{E}_{0,p}$ as both $\mathcal{E}_1,\mathcal{E}_2$ have probability $1-e^{-\Omega(n)}$.
    The desired claim again follows from computing
       $$\begin{aligned}
\mathbb{E}_{A_n}^{\mathcal{E}_0,p}&
\left[\prod_{r=1}^2 \frac{\|(A_n-\lambda_rI_n)^{-1}\|_*}{\mu_1(\lambda_r)}\cdot \frac{\mu_1(\lambda_{1})^{8/9}\mu_1(\lambda_2)^{8/9}}{\mu_k(\lambda_\ell)^{16/9}}\right]\\&\leq \prod_{r=1}^2\mathbb{E}_{A_n}^{\mathcal{E}_{0,p}}\left[\left(\frac{\|A_n-\lambda_r I_n\|_*}{\mu_1(\lambda_r)}\right)^{324}\right]^{1/324}\cdot \mathbb{E}_{A_n}^{\mathcal{E}_{0,p}}\left[\left(\frac{\sqrt{n}}{k\mu_k(\lambda_{\ell_1})}\right)^{288}\right]^{1/162}\\&\quad\quad \cdot \mathbb{E}_{A_n}^{\mathcal{E}_{0,p}}\left[\left(\frac{\mu_1(\lambda_1)\mu_1(\lambda_2)}{n}\right)^{9/10}
\right]^{80/81} \cdot k^{16/9}
\\&\lesssim k^2\mathbb{E}_{A_n}^{\mathcal{E}_{0,p}}\left[\left(\frac{\mu_1(\lambda_1)\mu_1(\lambda_2)}{n}\right)^{9/10}
\right]^{80/81}
    \end{aligned}$$ where we use Lemma \ref{lemma4.11} and Lemma \ref{lemma4.22} to show that the two terms on the second line are $O(1)$. Recalling that $\sum_{k\geq 1}k^2e^{-k}<+\infty,$ the proof is complete.
\end{proof}

\subsection{Conclusion}
Now we are ready to prove Lemma \ref{lemma6.111}.

\begin{proof}[\proofname\ of Lemma \ref{lemma6.111}]
    We make a further conditioning and write 

    $$\begin{aligned}
&\mathbb{E}_{A_n}\sup_{r_1,r_2}\mathbb{P}_X\left(
    \frac{|\langle (A_n-\lambda_i I_n)^{-1}X,X\rangle-r_i|}{\|(A_n-\lambda_i I_n)^{-1}\|_*}\leq\delta_i
    ,\langle X,u\rangle\geq s,\frac{\mu_1(\lambda_1)\mu_1(\lambda_2)}{n}\leq(\delta_1\delta_2)^{-p}
    \right)\\& \leq \mathbb{P}(\sigma_{min}(A_n-\lambda_1I_n)
\leq \delta_1\delta_2 n^{-1/2})+ \mathbb{P}(\sigma_{min}(A_n-\lambda_2I_n)
\leq \delta_1\delta_2 n^{-1/2})\\&+\mathbb{E}_{A_n}
\sup_{r_1,r_2}\mathbb{P}_{X}\left(\bigcap_{i=1}^2 \left\{\left\|\frac{\langle (A_n-\lambda_i I_n)^{-1}X,X\rangle-r_i}{\|(A_n-\lambda_i I_n)^{-1}X\|}\right\|\leq \delta_i\right\}\wedge \mathcal{E}\right),
    \end{aligned}$$
where we define $$\mathcal{E}:=\{\sigma_{min}(A_n-\lambda_iI_n)\geq\delta_1\delta_2n^{-1/2},i=1,2;\quad \frac{\mu_1(\lambda_1)\mu_1(\lambda_2)}{n}\leq (\delta_1\delta_2)^{-p}\}.$$

    The probability $\mathbb{P}(\sigma_{min}(A_n-\lambda_1I_n)
\leq \delta_1\delta_2 n^{-1/2})$ is easily bounded by $C\delta_1\delta_2+e^{-\Omega(n)}$ by an application of Proposition \ref{proposition6.666}.

To compute the last term, we use Theorem \ref{theorem3.1} to write the probability in question as the integral of $I(\theta)^{1/2}$ over the given region. Then we use the Hölder inequality estimate \eqref{holders} combined with decomposition \eqref{decompositionformula}, and the computations in \eqref{whatisi1}, \eqref{whatisi2} and \eqref{whatisi3}, as well as estimates in Lemma \ref{whatislemma1} and Lemma \ref{whatislemma2} to obtain the desired result.

\end{proof}

\section{Reduction to generic events: two locations}

Before we come to the proof of the main theorem, we need another preprocessing of the probability in question, to eliminate another pathological event. We will take $d=2$ in this section, and return to the general $d$ case later.

For any $p\geq1$, define
$$
\mathcal{P}^p:=\left\{ \prod_{i=1}^2\sigma_{min}(A^{(j)}_{n+1}-\lambda_i I_n)\leq(\delta_1\delta_2)^p n^{-1}\text{ for at least } \frac{1}{4}n \text{ values } j\in[n+1]\right\},
$$
where we recall that $A_{n+1}^{(j)}$ is the principal submatrix of $A_{n+1}$ with the $j$-th row and column removed. 

We shall prove in this section that 
\begin{Proposition}\label{proposition7.1} Fix $B>0$ and $\zeta\in\Gamma_B$, consider $A_{n+1}\sim \operatorname{Sym}_{n+1}(\zeta)$. Then for any $p>1$,
   $$ \mathbb{P}(\sigma_{min}(A_{n+1}-\lambda_i I_{n+1})\leq\delta_i n^{-1/2},i=1,2,\quad \mathcal{P}^p)\lesssim \delta_1\delta_2+e^{-\Omega(n)}.$$
\end{Proposition}

 The proof of Theorem \ref{Theorem1.1}
   can therefore be reduced to the following:

    \begin{Proposition}\label{finalfuckpropositionga} Let $A_{n+1}$ satisfy the assumptions of Theorem \ref{Theorem1.1} and $\lambda_1,\lambda_2\in\mathbb{R}$ . For any $p>1$,
    $$\begin{aligned}
           &\mathbb{P}(\sigma_{min}(A_{n+1}-\lambda_i I_{n+1})\leq\delta_i n^{-1/2},i=1,2)\lesssim \delta_1\delta_2+e^{-\Omega(n)}\\&+\mathbb{E}_{A_n}\sup_{r_1,r_2}\mathbb{P}_X\left(
    \frac{|\langle (A_n-\lambda_i I_n)^{-1}X,X\rangle-r_i|}{\|(A_n-\lambda_i I_n)^{-1}X\|_2}\leq\delta_i,i=1,2,\frac{\mu_1(\lambda_1)\mu_1(\lambda_2)}{n}\leq(\delta_1\delta_2)^{-p}
    \right).\end{aligned}$$
    \end{Proposition}

\subsection{Eliminating pathological events.}
We first prove Proposition \ref{proposition7.1}. To complete the proof we need two auxiliary results. The first concerns the eigenvector delocalization of Wigner matrices, taken from Rudelson and Vershynin \cite{rudelson2016no}, Theorem 1.3. This result is fundamental to the validity of the proof of this paper. Note that thanks to the continuous density assumption in Theorem \ref{Theorem1.1}, we do not need to assume that the entries of $A_{ij}$ have the same variance, so the result holds for GOE type matrices as well where the diagonal entries have a different variance.

\begin{theorem}\label{veragaghagag} Fix $B>0$ and let $\zeta\in\Gamma_B$ have a continuous density. Consider $A_n\sim\operatorname{Sym}_n(\zeta)$ and let $v$ be any eigenvector of $A_n$. Then there is some $c_2>0$ such that for any sufficiently small $c_1>0$, for $n$ sufficiently large,
$$
\mathbb{P}(|v_j|\geq (c_2c_1)^6 n^{-1/2}\text{ for at least }(1-c_1)n \text{ indices }j)\geq 1-e^{-c_1n}.
$$    
\end{theorem}

 The estimate is independent of the chosen eigenvector. 

We also need the following geometric fact from \cite{campos2024least}, Fact 6.7, whose first version appeared in Tao-Vu \cite{tao2017random}.

\begin{Fact}\label{fact7.3}
    Given any $n\times n$ real symmetric matrix $M$ and $\lambda$ as an eigenvector of $M$ with associated eigenvector $u$. Let $j\in[n]$ and assume $\lambda'$ is an eigenvector of the principal minor $M^{(j)}$ corresponding to a unit eigenvector $v$. Then we have
$$
|\langle v,X^{(j)}\rangle|\leq |\lambda-\lambda'|/|u_j|,
$$ where $X^{(j)}$ denotes the $j$-th column of $M$ having removed the $j$-th entry.
\end{Fact}

\begin{proof}[\proofname\ of Proposition \ref{proposition7.1}] We first find some value $j$ such that the event in $\mathcal{P}^p$ holds for the subscript $j$ and the unit eigenvectors (denoted $u_1$ and $u_2$) of $A_{n+1}$ corresponding to the least singular values of $A_{n+1}-\lambda_1 I_{n+1}$ and $A_{n+1}-\lambda_2 I_{n+1}$ respectively, satisfy $|(u_1)_j|,|(u_2)_j|\geq (c_2c_1)^6n^{-1/2}$ for some fixed $c_1$. The fact that such $j$ exists among the  given $\frac{1}{4}n$ values of $j$ specified in $\mathcal{P}^p$ are guaranteed by Theorem \ref{veragaghagag} with probability $1-e^{-\Omega(n)}$.

Now we would like to find a dyadic decomposition of $[(\delta_1\delta_2)^p,1]=\cup_{j=1}^{I_n}I_j$ into $I_n$ intervals, where we denote by $I_j=[I_{j,S},I_{j,L}]$ such that $I_{j,L}=2I_{j,S}$ and $I_{j,S}=I_{j+1,L}$ (except for $I_1$, where we do slightly differently). Then we have $I_n=O(\log(\delta_1\delta_2)^{-1})$ and we can classify the condition $\prod_{i=1}^2\sigma_{min}(A_{n}-\lambda_i I_n)\leq (\delta_1\delta_2)^p n^{-1}$ into $(I_n)^2$ sub-events, each given by the form $\sigma_{min}(A_{n}-\lambda_i I_n)\leq x_in^{-1/2}$, $\sigma_{min}(A_{n}-\lambda_i I_n)\leq x_jn^{-1/2}$ and $x_ix_j\leq 4(\delta_1\delta_2)^p$. (Note that if for some $i$ we already have $\sigma_{min}(A_{n}-\lambda_i I_n)\leq (\delta_1\delta_2)^p n^{-1/2}$ we may directly apply Proposition \ref{proposition6.666}, or if for some $i$ we have $\sigma_{min}(A_{n}-\lambda_i I_n)\geq n^{-1/2}$ then for the other $\ell\neq i$ we must have $\sigma_{min}(A_{n}-\lambda_\ell I_n)\leq (\delta_1\delta_2)^p n^{-1/2}$, and we are also done).

Now we analyze each of these $(I_n)^2$ events. We first assume that for at least one $i$ we must have $x_i\geq\delta_i$. Let us take $i=1$ without loss of generality, that is, $\sigma_{min}(A_{n+1}^{(j)}-\lambda_1 I_n)\geq \delta_1$. We first consider the other event $\sigma_{min}(A_{n+1}^{(j)}-\lambda_2 I_n)\leq x_2n^{-1/2}$. By Proposition \ref{proposition6.666}, we have the bound
\begin{equation}\label{secondcond}\mathbb{P}\left(\sigma_{min}(A_{n+1}^{(j)}-\lambda_2 I_n)\leq x_2n^{-1/2}\right)\lesssim x_2+e^{-\Omega(n)}.\end{equation}
Now we return to $\lambda_1$. By assumption that $x_1\geq\delta_1$, we need only to bound
    \begin{equation}\label{conditionone}
\mathbb{P}\left(\sigma_{min}(A_{n+1}-\lambda_1 I_{n+1})\leq x_1n^{-1/2},\sigma_{min}(A_{n+1}^{(j)}-\lambda_1 I_{n})\leq x_1n^{-1/2}\mid \mathcal{E}_{2,p} 
\right) 
    \end{equation} where we are conditioning on the event 
    $$ \mathcal{E}_{2,p}:=\{\sigma_{min}(A_{n+1}^{(j)}-\lambda_2 I_n)\leq x_2n^{-1/2}\}.$$

We use Fact \ref{fact7.3}, where we make the crucial observation that $X^{(j)}$ is independent of $A_{n+1}^{(j)}$, so we may assume that the vector $v$ is any fixed vector of unit length in Fact \ref{fact7.3} and apply the randomness of $X^{(j)}$. Since $X^{(j)}$ has a bounded density and $v$ has unit length, we apply Lemma \ref{lemma5.1} and conclude that
\begin{equation}
    \eqref{conditionone}\lesssim x_1+e^{-\Omega(n)}.
\end{equation} Then we can combine this with \eqref{secondcond} to obtain 
$$\begin{aligned}&\mathbb{P}\left(\sigma_{min}(A_{n+1}-\lambda_i I_{n+1})\leq \delta_i n^{-1/2},\sigma_{min}(A_{n+1}^{(j)}-\lambda_i I_{n})\leq x_in^{1/2},i=1,.2 
\right)\\&\quad\quad\lesssim x_1x_2+e^{-\Omega(n)}\lesssim (\delta_1\delta_2)^p+e^{-\Omega(n)}.\end{aligned}$$

Then we are left with the remaining case where $x_1<\delta_1$ and $x_2<\delta_2$. Our assumption that $x_1x_2<(\delta_1\delta_2)^p$ implies that for some $i=1,2$ we have $x_i\leq\delta_i(\delta_1\delta_2)^{(p-1)/2}$. Let us say we have $i=2$ ($i=1$ is analogous), so $x_2\leq \delta_2(\delta_1\delta_2)^{(p-1)/2}$.

We first apply Proposition \ref{proposition6.666} and obtain the same estimate as in \eqref{secondcond}. We then apply fact \ref{fact7.3} and Lemma \ref{lemma5.1} to obtain
\begin{equation}\begin{aligned}
&\mathbb{P}\left(\sigma_{min}(A_{n+1}-\lambda_1 I_{n+1})\leq \delta_1n^{-1/2},\sigma_{min}(A_{n+1}^{(j)}-\lambda_1 I_{n})\leq \delta_1n^{-1/2}\mid \mathcal{E}_{2,p} 
\right)\\& \lesssim \delta_1+e^{-\Omega(n)}
\end{aligned}    \end{equation} Combining these two estimates, we have
$$\begin{aligned}&\mathbb{P}\left(\sigma_{min}(A_{n+1}-\lambda_i I_{n+1})\leq \delta_i n^{-1/2},\sigma_{min}(A_{n+1}^{(j)}-\lambda_i I_{n})\leq x_in^{1/2},i=1,.2 
\right)\\&\quad\quad\lesssim \delta_1x_2+e^{-\Omega(n)}\lesssim (\delta_1\delta_2)^{(p+1)/2}+e^{-\Omega(n)}.\end{aligned}$$
Then we complete the proof by considering all the $(I_n)^2$ events, which leads to an entropy cost of $\log((\delta_1\delta_2)^{-1})^{C_d}$.
\end{proof}
We note that in the proof we have repeatedly applied Lemma \ref{lemma5.1} to deduce small ball probabilities assuming that $X^{(j)}$ has a continuous density. This can well be replaced by applying the inverse Littlewood-Offord lemmas (see \cite{rudelson2008littlewood}) and we do not need to assume a continuous density. However, in other parts of the paper, such Littlewood-Offord lemmas are not effective enough to address the joint LCD of many different eigenvectors (see \cite{rudelson2009smallest} for the use of joint LCD for bounding singular values); therefore, we do not use this approach and maintain the continuous density assumption throughout.

\subsection{Refinement of invertibility via distance}

Now we have shown that in order to prove
 $$ \mathbb{P}(\sigma_{min}(A_{n+1}-\lambda_i I_{n+1})\leq\delta_i n^{-1/2},i=1,2)\lesssim \delta_1\delta_2+e^{-\Omega(n)},$$ it suffices to prove, for any $p>1$,
  \begin{equation}\label{what'fnesa} \mathbb{P}(\sigma_{min}(A_{n+1}-\lambda_i I_{n+1})\leq\delta_i n^{-1/2},i=1,2,\quad\wedge\quad (\mathcal{P}^p)^c)\lesssim \delta_1\delta_2+e^{-\Omega(n)}.\end{equation} We now investigate how to determine this quantity. First recall the fundamental fact (see \cite{rudelson2008littlewood})

  \begin{Fact}
      For any $n\times n$ symmetric matrix, let $v$ be the unit vector corresponding to the least singular value of $M$, that is, $\|Mv\|_2=\sigma_{min}(M)$ and $\|v\|_2=1$. Then we have
      $$ \sigma_{min}(M)\geq|v_j|\cdot d_j(M)\quad\text{ for each }j\in[n],
      $$ where $d_j(M)$ is the distance of the $j$-th column of $M$ to the subspace spanned by the other $n-1$ columns of $M$.
  \end{Fact}

\begin{proof}[\proofname\ of Proposition \ref{finalfuckpropositionga}]
 Let the eigenvalues associated with the least singular value of $A_{n+1}-\lambda_i I_{n+1}$ be denoted by $u_i$ for each $i=1,2$. Let $\mathcal{C}$ be the rare event that, for some $i=1,2$, $u_i$ does not satisfy the conclusion of Theorem \ref{veragaghagag} where we take $c_1=\frac{1}{4}$. Then $\mathbb{P}(\mathcal{C})=e^{-\Omega(n)}$. On $\mathcal{C}^c$, we have that 
 \begin{equation}\label{fuckyousga}\sigma_{min}(A_{n+1}-\lambda_i I_{n+1})\geq c_3 n^{-1/2}\cdot d_j(A_{n+1}-\lambda_iI_{n+1})\text{ for both } i=1,2\end{equation}
for at least $\frac{1}{2}n$ choices of $j\in[n+1]$, where $c_3$ is a universal constant. 

Furthermore, by assumption of $(\mathcal{P}^{p})^c$, there are at least $\frac{1}{4}n$ values of $j\in[n+1]$ such that \eqref{fuckyousga} holds for such $j$ and moreover, 
$$
\prod_{i=1}^2 \sigma_{min}(A_{n+1}^{(j)}-\lambda_i I_{n})\geq (\delta_1\delta_2)^p n^{-1}.
$$

Now let $S\subset \{1,2,\cdots,n+1\}$ denote the collection of subscripts $j\in[n+1]$ such that the following event $\mathcal{H}_j$ holds (for some universal constant $C$) 
$$\mathcal{H}_j:=
\left\{d_j(A_{n+1}-\lambda_i I_{n+1})\leq C\delta_i\text{ for } i=1,2, \prod_{i=1}^2 \sigma_{min}(A_{n+1}^{(j)}-\lambda_i I_{n})\geq (\delta_1\delta_2)^p n^{-1}\right\}.
$$
Our previous discussions imply that
$$\left\{\sigma_{min}(A_{n+1}-\lambda_i I_{n+1})\leq\delta_i n^{-1/2},i=1,2\right\}\bigcap (\mathcal{P}^p)^c\bigcap\mathcal{C}^c\Rightarrow |S|>\frac{1}{4}n.$$

By a classical first moment method, $$\mathbb{P}(|S|>\frac{1}{4}n)\leq\frac{4}{n}\sum_{j=1}^{n+1}\mathbb{P}(\mathcal{H}_j),$$
therefore, for some specific $j_*\in[n+1]$ we have $\mathbb{P}(|S|\geq\frac{1}{4}n)\leq 5 \mathbb{P}(\mathcal{H}_{j_*})$ when $n$ is large.

In the final step, we use Fact \ref{fact2.4} to rewrite $d_j(A_{n+1}-\lambda_i I_{n+1})$ in the form of distance and complete the proof.
\end{proof}

Note that in contrast to the usual proofs via invertibility by distance (see \cite{rudelson2008littlewood} and \cite{campos2024least}, Lemma 6.3) where one only uses the fact that the test vector is incompressible, here we use the much stronger result from \cite{rudelson2016no} that both eigenvectors have $(1-c_1n)$ coordinates that are not too small, so that we can find many coordinates that are not too small for \textbf{both} vectors. This stronger property is fundamental in this multimatrix setting.

\section{Bootstrap and proof completion: two locations}\label{bootstrapping209}

In this section we complete the proof of our main Theorem, theorem \ref{Theorem1.1} for the $d=2$ case. We will first use an initial estimate $(\delta_1\delta_2)^{1/10}$ and bootstrap to obtain an estimate $\delta_1\delta_2(\log(\delta_1\delta_2)^{-1})^{1/2}$, and subsequently eliminate the log factor and obtain the optimal bound.

\subsection{Initial estimate and bootstrapping}
First recall the well-known concentration inequality from Hanson and Wright \cite{wright1973bound} \cite{hanson1971bound}, see also \cite{vershynin2018high}, Theorem 6.2.1.

\begin{theorem}[Hanson-Wright]\label{hansonwright8.1}
    Given $B>0$, $\zeta\in\Gamma_B$ and consider $X\sim\operatorname{Sym}_n(\zeta)$. Let $M$ be an $n\times n$ matrix. Then given any $t>0$,
    $$
\mathbb{P}_X(|\|MX\|_2-\|M\|_{HS}|>t)\leq 4\exp(-\frac{ct^2}{B^4\|M\|_2})
    $$ for a universal constant $c>0$.
\end{theorem}

We can then perform the initial estimate as follows:

\begin{lemma}\label{lemma7.27} Let the random matrix $A_n$ satisfy the assumptions of Theorem \ref{Theorem1.1}. For any fixed $\kappa>0$ and $\Delta>0$, and any $\lambda_1,\lambda_2\in[-(2-\kappa)\sqrt{n},(2-\kappa)\sqrt{n}]$ with $|\lambda_1-\lambda_2|\geq\Delta\sqrt{n}$, we have the following initial estimate:
for any $\delta_1,\delta_2>0$,
$$
\mathbb{P}\left(\sigma_{min}(A_{n+1}-\lambda_i I_{n+1})\leq \delta_i n^{-1/2}, \quad i=1,2 \right)\lesssim (\delta_1\delta_2)^{1/10}+e^{-\Omega(n)}.
$$
\end{lemma}
The rate $\epsilon^{1/10}$ is rather weak, but we will bootstrap it to the optimal rate later.

\begin{proof}
By Proposition \ref{finalfuckpropositionga}, it suffices to prove that for any $(r_1,r_2)\in\mathbb{R}^2$ and some $p>1$,
\begin{equation}\begin{aligned}
    &\mathbb{P}_{A_n,X}\left(
    \frac{|\langle (A_n-\lambda_i I_n)^{-1}X,X\rangle-r_i|}{\|(A_n-\lambda_iI_n)^{-1}X\|_2}\leq C\delta_i,  i=1,2,  \prod_{i=1}^2\sigma_{min}(A_n-\lambda_iI_n)\geq (\delta_1\delta_2)^pn^{-1}
    \right)\\&\lesssim (\delta_1\delta_2)^{1/10}+e^{-\Omega(n)}.
\end{aligned}\end{equation}
Applying Hanson-Wright (Theorem \ref{hansonwright8.1}), there is a $C'>0$ satisfying, for $i=1,2$,

\begin{equation}
    \mathbb{P}_X(\|(A_n-\lambda_i I_n)^{-1}X\|_2\geq C'(\log(\delta_1\delta_2)^{-1})^{1/2}\cdot\|(A_n-\lambda_i I_n)^{-1}\|_{HS})\leq \delta_1\delta_2,
\end{equation}
so it suffices to bound, for any $(\theta_1,\theta_2)\in\mathbb{R}^2$,
\begin{equation}\label{sufficestobound}
\mathbb{P}_{A_n,X}\left(\frac{|\langle (A_n-\lambda_i I_n)^{-1}X,X \rangle-r_i|}{\|(A_n-\lambda_i I_n)\|_{HS}}\leq\bar{\delta}_i,i=1,2,\prod_{i=1}^2\sigma_n(A_n-\lambda_i I_n)\geq \delta_in^{-1/2}\right),\end{equation} where we set $\bar{\delta}_i:=C^{''}\delta_i(\log(\delta_1\delta_2)^{-1})^{1/2}$ for $i=1,2$, and we used the fact that $\|M\|_*\geq \|M\|_{HS}$ for any square matrix $M$.

Then we apply Proposition \ref{finalfuckpropositionga} (with $p=1.01$) and conclude with

\begin{equation}
    \eqref{sufficestobound}\lesssim \bar{\delta}_1\bar{\delta_2}(\delta_1\delta_2)^{-8p/9}+e^{-\Omega(n)}\leq (\delta_1\delta_2)^{1/10}+e^{-\Omega(n)}.
\end{equation}
\end{proof}

Now we set up our bootstrapping estimate. 
\begin{lemma}\label{bootstraplemma} Let $A_{n+1}$ and $\lambda_1,\lambda_2$ satisfy the assumptions of Theorem \ref{Theorem1.1}.
Suppose that for any $\delta_1,\delta_2>0$ and all $n$ we already have
\begin{equation}
    \mathbb{P}\left(\sigma_{min}(A_{n+1}-\lambda_iI_{n+1})\leq\delta_i n^{-1/2},i=1,2\right)\lesssim(\delta_1\delta_2)^\tau+e^{-\Omega(n)}. 
\end{equation} Then for all all $\delta_1,\delta_2>0$, all $n$ and any $p>1$, and $\omega>0$ that is sufficiently small, we have 
\begin{equation}\begin{aligned}&
    \mathbb{P}(\sigma_{min}(A_{n+1}-\lambda_iI_{n+1})\leq\delta_i n^{-1/2},i=1,2)\\&\lesssim
    (\delta_1\delta_2)^{\min(1,\frac{80}{81}(1-\omega)\tau p-\frac{8}{9}p+1)}(\log(\delta_1\delta_2)^{-1})^{1/2}+e^{-\Omega(n)}.
\end{aligned}\end{equation}
\end{lemma}

\begin{proof}
    Assume $\delta_1,\delta_2\geq e^{-cn}$, otherwise the claim is trivial.
    We first apply Lemma \ref{singularvalueproduct} to obtain the following: for any $\epsilon>0$, and any sufficiently small $\omega>0$,
    $$\begin{aligned}&
\mathbb{P}(\sigma_{min}(A_{n+1}-\lambda_1I_{n+1})\cdot \sigma_{min}(A_{n+1}-\lambda_2I_{n+1})\\&\quad\leq \epsilon n^{-1})\lesssim \epsilon(\log\epsilon^{-1})^{d}+e^{-\Omega(n)}\lesssim \epsilon^{1-\omega}+e^{-\Omega(n)}.\end{aligned}
    $$
    
    Following the same lines as in the proof of the base step, we evaluate (using Lemma \ref{lemma6.111} with parameter $p$, with $u=0$, $s=0$ and take $\epsilon=\delta_1\delta_2$):
$$\begin{aligned}
&\mathbb{E}_{A_n}\left[\left(\frac{\mu_1(\lambda_1)\mu_1(\lambda_2)}{n}\right)^{9/10}\mathbf{1}\left\{\frac{\mu_1(\lambda_1)\mu_1(\lambda_2)}{n}\leq\epsilon^{-p}\right\}\right]
\\&
\lesssim \int_0^{\epsilon^{-9p/10}}\mathbb{P}\left(\sigma_{min}(A_n-\lambda_1I_n)\cdot\sigma_{min}(A_n-\lambda_2I_n)\leq x^{-10/9}n^{-1}\right)dx,\\&
\lesssim 1+\int_1^{\epsilon^{-9p/10}}(x^{-(1-\omega)10\tau/9}+e^{-cn})dx\lesssim \max\{1,\epsilon^{(1-\omega)\tau p-9p/10}\}.
\end{aligned}$$

Then we apply Lemma \ref{lemma6.111} with $s=0$, $u=0$ and get: for any $(r_1,r_2)\in\mathbb{R}^2$,
\begin{equation}\begin{aligned}\mathbb{P}_{A_n,X}&\left(\prod_{i=1}^2 \frac{|\langle (A_n-\lambda_i I_n)^{-1}X,X\rangle-r_i|}{\|A_n-\lambda_i I_n\|_{HS}}\leq\delta_1\delta_2,\frac{\mu_1(\lambda_1)\mu_1(\lambda_2)}{n}\leq\epsilon^{-p}
\right)\\&\lesssim \max\{\delta_1\delta_2,(\delta_1\delta_2)^{\frac{80}{81}(1-\omega)\tau p-\frac{8}{9}p+1}\}\cdot(\log \epsilon^{-1})^{1/2}+e^{-\Omega(n)},
\end{aligned}\end{equation}
where we have used the fact that $\|A_n-\lambda_i I_n\|_{HS}\leq \|A_n-\lambda_i I_n\|_*$.

Finally, applying Hanson-Wright (Theorem \ref{hansonwright8.1}), for $i=1,2$, we obtain
$$
\mathbb{P}_X\left(\|(A_n-\lambda_i I_n)^{-1}X\|_2\geq C'\|(A_n-\lambda_i I_n)^{-1}\|_{HS}\cdot(\log(\delta_1\delta_2)^{-1})^{1/2}\right)\leq\delta_1\delta_2.
$$
\end{proof}

Now we can finish the bootstrap procedure. We need to first fix $p$, bootstrap to increase $\tau$, then increase $p$ and bootstrap $\tau$ again. We will prove that only finitely many steps are needed.

\begin{lemma}\label{lemma8.4wefinish}
    In the setting of Lemma \ref{lemma7.27}, for any $\delta_1,\delta_2>0$,
$$
\mathbb{P}\left(\sigma_{min}(A_{n+1}-\lambda_i I_{n+1})\leq \delta_i n^{-1/2}, \quad i=1,2 \right)\lesssim \delta_1\delta_2(\log(\delta_1\delta_2)^{-1})^{1/2}+e^{-\Omega(n)}.
$$

\end{lemma}

\begin{proof} By lemma \ref{lemma7.27}, we start with the initial value $\tau=0.1$.
We take $\omega$ arbitrarily small (say $10^{-10}$) and take $p=\frac{81}{80(1-\omega)}$ such that Lemma \ref{lemma7.27} yields the bootstrapping step $\tau\to \tau+0.09$ (we upper bound $\epsilon^{\tau+1-0.9/(1-\omega)}(\log\epsilon^{-1})^{1/2}$ by $\epsilon^{\tau+0.09}$) until $\tau=1$. Then after ten applications of Lemma \ref{lemma7.27} to bootstrap $\tau$ from 0.1 to $0.19,0.28,\cdots,0.91,1$, we finish the proof.
\end{proof}

\subsection{Completing the proof: eliminating the log factor}\label{removalladfaga}

By Lemma \ref{finalfuckpropositionga}, to conclude the proof of Theorem \ref{Theorem1.1} it suffices to prove that we can find some $C>0$ so that for any $\delta_1,\delta_2>0$ and any $p>1$,
\begin{equation}\label{section11base}\begin{aligned}
    \mathbb{P}&\left(\frac{\langle (A_n-\lambda_i I_n)^{-1}X,X\rangle}{\|(A_n-\lambda_i I_n)^{-1}X\|_2}\leq C\delta_i,i=1,2,\prod_{i=1}^2\sigma_{min}(A_{n}-\lambda_i I_n)\geq (\delta_1\delta_2)^pn^{-1}\right)\\&\lesssim\delta_1\delta_2+e^{-\Omega(n)}.\end{aligned}
\end{equation}
For any fixed $(r_1,r_2)\in\mathbb{R}^2$ We introduce the following notations (we do not make explicit the dependence on $r_i$ in the symbols introduced, as they are not essential)
$$
Q(A,X,\lambda_i):=\frac{|\langle (A_n-\lambda_i I_n)^{-1}X,X\rangle-r_i|}{\|(A_n-\lambda_i I_n)^{-1}X\|_2},$$
$$ Q_*(A,\lambda_i,X):=\frac{|\langle (A_n-\lambda_i I_n)^{-1}X,X\rangle-r_i|}{\|(A_n-\lambda_i I_n)^{-1}\|_*},
$$
recalling that 
$$
\|(A_n-\lambda_iI_n)^{-1}\|_*^2:=\sum_{k=1}^n \mu_k^2(\lambda_i)(\log(1+k))^2.
$$

Let $\mathcal{E}:=\left\{\prod_{i=1}^2 \sigma_{min}(A_n-\lambda_i I_n)\geq(\delta_1\delta_2)^p n^{-1}\right\}.$
We now decompose the left hand side of \eqref{section11base}:
\begin{equation}\label{whatistheleftside?}\begin{aligned}
    &\mathbb{P}^{\mathcal{E}}\left(Q(A,X,\lambda_i)\leq C\delta_i,i=1,2\right)\leq\mathbb{P}^{\mathcal{E}}(Q_*(A,X,\lambda_i)\leq C\delta_i,i=1,2)\\&+\mathbb{P}^{\mathcal{E}}\left(Q(A,X,\lambda_i)\leq C\delta_i,\quad i=1,2;\quad \frac{\|(A_n-\lambda_i I_n)^{-1} X\|_2}{\|(A_n-\lambda_i I_n)^{-1}\|_*}\geq 2,i=1,2.\right)\\&+\mathbb{P}^{\mathcal{E}}\left(Q(A,X,\lambda_i)\leq C\delta_i, i=1,2;\quad \frac{\|(A_n-\lambda_i I_n)^{-1} X\|_2}{\|(A_n-\lambda_i I_n)^{-1}\|_*}\geq 2,
    \text{for only one }i\in\{1,2\}\right).\end{aligned}
\end{equation}

Applying Lemma \ref{lemma8.4wefinish} together with Lemma \ref{lemma6.111} (taking $u=0$, $s=0$) we can easily prove the following Lemma, which bounds the right hand side of the first line of \eqref{whatistheleftside?}.
\begin{lemma}\label{lemmaa7.5fff} Let $A_n$ and $\lambda_1,\lambda_2$ satisfy the assumptions in Theorem \ref{Theorem1.1}. Then there exists some $p_0>1$ such that for any $p\in(1,p_0)$ we have
    $$
    \mathbb{P}^{\mathcal{E}}(Q_*(A,X,\lambda_i)\leq C\delta_i,i=1,2)\lesssim\delta_1\delta_2+e^{-\Omega(n)}.$$
\end{lemma}

In the following we evaluate the probability in the second line of \eqref{whatistheleftside?}, i.e.
\begin{equation}\label{secondline}
\mathbb{P}^{\mathcal{E}}\left(Q(A,X,\lambda_i)\leq C\delta_i,\quad i=1,2;\quad \frac{\|(A_n-\lambda_i I_n)^{-1} X\|_2}{\|(A_n-\lambda_i I_n)^{-1}\|_*}\geq 2,\quad i=1,2.\right).\end{equation}

Here we consider the direct product of two dyadic partitions $$2^{j_i}\leq\|(A_n-\lambda_i I_n)^{-1}X\|_2/\|(A_n-\lambda_i I_n)^{-1}\|_* \leq 2^{j_i+1},\quad i=1,2,$$ and upper bound \eqref{secondline} by 
\begin{equation}\label{bybounds}
     \sum_{j_1=1}^{\log n}\sum_{j_2=1}^{\log n}\mathbb{P}^{\mathcal{E}}\left(Q_*(A,X,\lambda_i)\leq 2^{j_i+1}C\delta_i,\quad \frac{\|(A_n-\lambda_i I_n)^{-1}X\|_2}{\|(A_n-\lambda_i I_n)^{-1}\|_*}\geq 2^{j_i},\quad i=1,2\right)+e^{-\Omega(n)}.
\end{equation}
Note that we can terminate the summation at $\log n$ because, by Hanson-Wright (Theorem \ref{hansonwright8.1}) and $\|M\|_*\geq \|M\|_{HS}$ for any square matrix $M$, we have
for both $i=1,2$,
$$
\mathbb{P}_X(\|(A_n-\lambda_i I_n)^{-1}X\|_2\geq\sqrt{n}\|(A_n-\lambda_i I_n)^{-1}\|_*)\lesssim e^{-\Omega(n)}.
$$
Now we estimate each individual probability via
\begin{lemma}\label{lemma7.6alreadytire} For any $t_i,\delta_i>0,i=1,2$ we have the following bound
$$\begin{aligned}
    &\mathbb{P}_X\left(Q_*(A,X,\lambda_i)\leq 2Ct_i\delta_i,\quad \frac{\|(A_n-\lambda_i I_n)^{-1}X\|_2}{\|(A_n-\lambda_i I_n)^{-1}\|_*}\geq t_i,\quad \text{ for each }i=1,2 
    \right)\\& \leq 4\sum_{k_1,k_2=1}^n\mathbb{P}_X(Q_*(A,X,\lambda_i)\leq 2Ct_i\delta_i,\langle X,v_{k_i}(\lambda_i)\rangle\geq t_i\log(1+k_i),\text{ for each }i=1,2),
\end{aligned}$$ where $v_k(\lambda_i)$ is the singular vector associated with the $k$-th largest singular value of $A_n-\lambda_i I_n$.
Since each $v_k(\lambda_i)$ is an eigenvector of $A_n$, we assume moreover that all the $v_k(\lambda_i)$ are chosen from an orthonormal basis of eigenvectors of $A_n$.
\end{lemma}
\begin{proof}
We apply singular value decomposition to $\|(A_n-\lambda_i I_n)^{-1}X\|_2\geq t_i\|(A_n-\lambda_i I_n)^{-1}\|_*$ and obtain
$$
t_i^2\sum_k \mu_k^2(\lambda_i)(\log(k+1))^2\leq \sum_k \mu_k^2(\lambda_i)\langle v_k(\lambda_i),X\rangle^2.
$$
That is, 
$$
\{\|(A_n-\lambda_i I_n)^{-1}X\|_2\geq t_i\|(A_n-\lambda_i I_n)^{-1}\|_*\}\subset\bigcup_k \{|\langle X,v_k(\lambda_i)\rangle|\geq t_i\log(k+1)\}.
$$ We apply this decomposition to both $i=1$ and $i=2$. That is, there exists $k_1,k_2\in[1,n]$ such that $|\langle X,v_{k_i}(\lambda_i)\rangle|\geq t_i\log(k_i+1)$. In the generic case where $v_{k_1}(\lambda_1)$ and $v_{k_2}(\lambda_2)$ are orthogonal, we may consider flipping the sign of $v_{k_i}(\lambda_i)$ so that for both $i=1,2$ we have $\langle X,v_{k_i}(\lambda_i)\rangle\geq t_i\log(k_i+1).$ This leads to a leading power 4.
It is also possible that $v_{k_1}(\lambda_1)$ and $v_{k_2}(\lambda_2)$ are colinear, and that we take $v_{k_1}(\lambda_1)=v_{k_2}(\lambda_2)$ without loss of generality. In this case we are also finished by simultaneously flipping the sign of  $v_{k_i}(\lambda_i)$.
\end{proof}

Finally, to estimate the third line of \eqref{whatistheleftside?}, it suffices to consider
$$\mathbb{P}^{\mathcal{E}}\left(Q_*(A,X,\lambda_1)\leq C\delta_1, Q(A,X,\lambda_2)\leq C\delta_2;\quad \frac{\|(A_n-\lambda_2 I_n)^{-1} X\|_2}{\|(A_n-\lambda_2 I_n)^{-1}\|_*}\geq 2\right),$$
(the case when we swap subscripts 1 and 2 is exactly the same). This possibility in question is bounded by $e^{-cn}$ plus the following: 
$$
\sum_{j=1}^{\log n}\mathbb{P}^{\mathcal{E}}\left( Q_*(A,X,\lambda_1)\leq 2C\delta_1,  Q_*(A,X,\lambda_2)\leq 2^{j+1}C\delta_2,
\frac{\|(A_n-\lambda_2 I_n)^{-1}X\|_2}{\|(A_n-\lambda_2 I_n)^{-1}\|_*}
\geq 2^j\right). 
$$ As in the argument of Lemma \ref{lemma7.6alreadytire}, this possibility is bounded by $e^{-cn}$ plus the following:
\begin{equation}\label{deductionsigagg}
\sum_{j=1}^{\log n}\sum_{k=1}^n\mathbb{P}^{\mathcal{E}}\left( Q_*(A,X,\lambda_1)\leq 2C\delta_1,  Q_*(A,X,\lambda_2)\leq 2^{j+1}C\delta_2,
\langle X,v_k\rangle\geq 2^j\log(1+k)\right)\end{equation}
where $\{v_i\}_{i=1}^n$ is an orthonormal base of the eigenvectors of $A_n$.

Now we can finish the proof of Theorem \ref{Theorem1.1}.

\begin{proof}[\proofname\ of Theorem \ref{Theorem1.1} for $d=2$ locations]
We combine \eqref{whatistheleftside?}, Lemma \ref{lemmaa7.5fff} with equation \eqref{bybounds}, Lemma \ref{lemma7.6alreadytire} and the deduction of \eqref{deductionsigagg} to obtain
$$\begin{aligned}
&\mathbb{P}^{\mathcal{E}}(Q(A,X,\lambda_i)\leq C\delta_i,i=1,2)\lesssim e^{-\Omega(n)}+\delta_1\delta_2\\&+\sum_{j_1,j_2=1}^{\log n}\sum_{k_1,k_2=1}^{n}\mathbb{P}^{\mathcal{E}_0}\left(Q_*(A,X,\lambda_i)\leq 2^{j_i+1}C\delta_i;\langle X,v_{k_i}(\lambda_i)\rangle\geq  2^{j_i}\log((1+k_i))
\right)\\&+\sum_{j_1=1,j_2\equiv 0}^{j_1=\log n}\sum_{k_1=1,k_2\equiv 0}^{k_1=n}\mathbb{P}^{\mathcal{E}_0}\left(Q_*(A,X,\lambda_i)\leq 2^{j_i+1}C\delta_i;\langle X,v_{k_i}(\lambda_i)\rangle\geq  2^{j_i}\log((1+k_i))
\right)\\&+\sum_{j_1\equiv 0,j_2=1}^{j_2=\log n}\sum_{k_1\equiv0,k_2=1}^{k_2=n}\mathbb{P}^{\mathcal{E}_0}\left(Q_*(A,X,\lambda_i)\leq 2^{j_i+1}C\delta_i;\langle X,v_{k_i}(\lambda_i)\rangle\geq  2^{j_i}\log((1+k_i))
\right).
\end{aligned}$$
We can combine the two conditions $\langle X,v_{k_i}(\lambda_i)\rangle\geq 2^{j_i}\log(1+k_i)$, $i=1,2$ into $\langle X,v_{k_1k_2}\rangle\geq 2^{j_1}\log(1+k_1)+2^{j_2}\log(1+k_2)$,
where we set $u_{k_1k_2}=v_{k_1}(\lambda_1)+v_{k_2}(\lambda_2)$. Note that $v_{k_1}(\lambda_1)$ and $v_{k_2}(\lambda_2)$ are both eigenvectors of $A_n$, so either we have $\|u_{k_1k_2}\|=\sqrt{2}$ if they correspond to two orthogonal eigenvectors, or if they correspond to co-linear vectors we must have that $v_{k_1}(\lambda_1)=v_{k_2}(\lambda_2)$ (by assumption of Lemma \ref{lemma7.6alreadytire} they are either co-linear or orthogonal) so that $\|u_{k_1k_2}\|=2$. We can now apply Lemma \ref{lemma6.111} for all $t_i>0$ where we choose $\delta_i$ to be $2Ct_i\delta_i$, $s=t_1\log(k_1+1)+t_2\log(k_2+1)$ and $u=u_{k_1k_2}$. Then Lemma \ref{lemma6.111} implies
\begin{equation}\begin{aligned}
    &\mathbb{P}^{\mathcal{E}_0}(Q_*(A,X,\lambda_i)\leq 2Ct_i\delta_i,\langle X,u_{k_1k_2}\rangle\geq t_1\log(1+k_1)+t_2\log(1+k_2))\\&\lesssim \delta_1\delta_2 t_1t_2(k_1+1)^{-t_1/2}(k_2+1)^{-t_2/2}\cdot I^{81/80}+e^{-\Omega(n)},
\end{aligned}\end{equation}
where $$I:=\mathbb{E}_{A_n}\left[\left(\frac{\mu_1(\lambda_1)\mu_1(\lambda_2)}{n}\right)^{9/10}\mathbf{1}\left\{\frac{\mu_1(\lambda_1)\mu_1(\lambda_2)}{n}\leq(\delta_1\delta_2)^{-p}\right\} \right]$$
and $p>1$ can be chosen to be arbitrarily close to 1. 
Applying Lemma \ref{lemma8.4wefinish} combined with Lemma \ref{singularvalueproduct} and a suitable choice of parameters (i.e. setting $p$ sufficiently close to 1 and when applying Lemma \ref{singularvalueproduct}, setting $\tau$ sufficiently close to 1), we can show that $$I\lesssim 1.$$
Combining all the above, we obtain 
$$\begin{aligned}
\mathbb{P}^{\mathcal{E}_0}(Q(A,X,\lambda_i)\leq C\delta_i,i=1,2)&\lesssim \delta_1\delta_2 \sum_{j_1,j_2=1}^{\log n}\sum_{k_1,k_2=1}^n 2^{j_1+j_2}(k_1+1)^{-2^{j_1}}(k_2+1)^{-2^{j_2}}
\\&+2\delta_1\delta_2\sum_{j=1}^{\log n}\sum_{k=1}^n 2^j(k+1)^{-2^j}
+e^{-\Omega(n)}.
\end{aligned}$$

Now the proof is completed by the elementary computation
$$\sum_{j_1=1}^\infty\sum_{j_2=1}^\infty\sum_{k_1=1}^\infty \sum_{k_2=1}^\infty 2^{j_1+j_2}(k_1+1)^{-2^{j_1}}(k_2+1)^{-2^{j_2}}=O(1),$$ 
and $\sum_{j=1}^\infty\sum_{k=1}^\infty 2^j(k+1)^{-2^j}=O(1)$.
\end{proof}

\section{The multilocation case}\label{multiparticlecase}
In the previous four sections we have invested considerable effort in proving Theorem \ref{Theorem1.1} in the $d=2$ case, and now we do so for general $d\in\mathbb{N}_+$. We will use an inductive argument, assuming that Theorem \ref{Theorem1.1} has been proven for all $1,2,\cdots,d-1$ different locations $\lambda_i$'s, and then prove it for $d$ different $\lambda_i$'s. All the essential difficulties are contained in the $d=2$ case, and we outline in this section how to adapt the proof to general $d$. Some proofs are given in detail, while the others are sketched and referred to those in the previous sections.

We first state the induction hypothesis that will be in effect throughout the whole section:

\begin{Induction hypothesis} Fix some $\kappa>0$ and $\Delta>0$, and let $A_n$ satisfy the assumptions of Theorem \ref{Theorem1.1}. Fix any $d$ real numbers $\lambda_1,\cdots,\lambda_d\in[-(2-\kappa)\sqrt{n},(2-\kappa)\sqrt{n}]$ such that for any $i\neq j$, $|\lambda_i-\lambda_j|\geq \Delta\sqrt{n}$.

Then we have proven that, for any $k\in\{1,2,\cdots,d-1\}$, $k$ distinct subscripts $i_1,\cdots,i_k\in\{1,2,\cdots,d\}$ and for any $\delta_{i_1},\cdots,\delta_{i_k}>0$, we have
\begin{equation}
    \mathbb{P}\left(\sigma_{min}(A_n-\lambda_{i_\ell} I_n)\leq\delta_{i_\ell}n^{-1/2}, \ell=1,\cdots,k
    \right)\lesssim \delta_{i_1}\cdots\delta_{i_k}+e^{-\Omega(n)}
\end{equation} for any sufficiently large $n\in\mathbb{N}_+$.
\end{Induction hypothesis}

\subsection{Outline of the main arguments}\label{section8.111}
We outline how to prove Theorem \ref{Theorem1.1} for $d$ different $\lambda_i$'s under the inductive hypothesis. All the results presented in this section will be proven later.

The following proposition generalizes Proposition \ref{finalfuckpropositionga}. We did not state the general $d$ version there as the proof will rely on the (not yet verified) induction hypothesis.
\begin{Proposition}
    \label{lemma8.11111} Let $A_{n+1}$ and $\lambda_1,\cdots,\lambda_d$ satisfy the assumptions of Theorem \ref{Theorem1.1}. For any $p>1$ we have
     $$\begin{aligned}
           &\mathbb{P}(\sigma_{min}(A_{n+1}-\lambda_i I_{n+1})\leq\delta_i n^{-1/2},i=1,2,\cdots,d)\lesssim \delta_1\delta_2\cdots\delta_d+e^{-\Omega(n)}\\&+\mathbb{E}_{A_n}\sup_{r_1,\cdots,r_d}\mathbb{P}_X\left(
    \frac{|\langle (A_n-\lambda_i I_n)^{-1}X,X\rangle-r_i|}{\|(A_n-\lambda_i I_n)^{-1}X\|_2}\leq\delta_i,\frac{\prod_{i=1}^d\mu_1(\lambda_i)}{n^{d/2}}\leq(\prod_{i=1}^n\delta_i)^{-p}
    \right),\end{aligned}$$ where we omit $i=1,\cdots,d$ in the second line.
\end{Proposition}
At this point we can apply Theorem \ref{theorem3.1} to bound the probability in question, so that we need to integrate the function $I(\theta)$ \eqref{Itheta} over the region $\{\sum_{i=1}^d|\theta_i\delta_i|^2\leq d\}$. Again we temporarily replace the denominator $\|(A_n-\lambda_i I_n)^{-1}X\|_2$ by $\|(A_n-\lambda_i I_n)^{-1}X\|_*$ and return to this issue at the end. We apply Hölder's inequality to $I(\theta)$ as in \eqref{holders} and obtain 
$$
I(\theta)\lesssim \left(\mathbb{E}_{X_2,X_2'}e^{9\langle X_2+X_2',u\rangle}\right)^{1/9}\left(\mathbb{E}_{\widetilde{X}}e^{-c''\|\sum_{i=1}^d \frac{\theta_i}{\mu_1(\lambda_i)}(A_n-\lambda_i I_n)^{-1}\widetilde{X}\|_2^2}
\right)^\frac{8}{9},\theta\in\mathbb{R}^d,$$
where the first term on the right hand side is $O(1)$.

Now we have to bound from above $I(\theta)$ carefully, depending on both the dimension $d$ and the region of $\theta$. Let $2^{[d]}$ denote the collection of all the subsets of $\{1,2,\cdots,d\}$, and for any $\mathcal{A}\in 2^{[d]}$ let $\mathcal{D}_\mathcal{A}$ be defined as follows:
$$\mathcal{D}_\mathcal{A}:=\left\{\theta\in\mathbb{R}^d:\sum_{i=1}^d|\theta_i\delta_i|^2\leq d,\quad |\theta_i|\geq 1 \text{ if } i\in\mathcal{A},\quad |\theta_i|< 1 \text{ if } i\notin \mathcal{A}\right\}.$$

We introduce another useful notation: for any $\mathcal{A}\in 2^{[d]}$ with $\mathcal{A}\neq\emptyset$, define
$$ \|\theta\|_{\mathcal{A}}:=\left(\prod_{i\in\mathcal{A}}|\theta_i|\right)^\frac{1}{|\mathcal{A}|}
$$ the geometric average of $|\theta_i|$ on $\mathcal{A}$, where $|\mathcal{A}|$ denotes the cardinality of the set $\mathcal{A}$. Then, we observe that, to bound $\int_{\mathcal{D}_\mathcal{A}}I(\theta)^{1/2}d\theta$, it suffices to bound 
\begin{equation}\label{integralforsimplified}
\int_{\mathcal{D}_{\mathcal{A}}}\mathbb{P}_{\widetilde{X}}\left( \left\|\sum_{i=1}^d \frac{\theta_i}{\mu_1(\lambda_i)}(A_n-\lambda_i I_n)^{-1}\widetilde{X}\right\|_2\leq \|\theta\|_{\mathcal{A}}^{\frac{1}{1o^d+1}}
\right)^\frac{4}{9}d\theta.
\end{equation} This is because for any $\mathcal{A}\neq\emptyset$ and any $c''>0$,
$$
\int_{\mathcal{D}_{\mathcal{A}}}e^{-c''\|\theta\|_\mathcal{A}^{2/(10^d+1)}}d\theta=O(1).
$$

Now we estimate \eqref{integralforsimplified} for $\mathcal{A}=\{1,2,\cdots,d\}$ (we simply write $\mathcal{A}=[d]$ in this case). We first renormalize $(\theta_1,\cdots,\theta_d)$, then apply Lemma \ref{lemma6.61}, and then simplify the integral to a one-variate integration. That is,
\begin{equation}
    \eqref{integralforsimplified}_{\mathcal{A}=[d]}=
\int_{\mathcal{D}_{[d]}}\mathbb{P}_{\widetilde{X}}\left(\left\|\sum_{i=1}^d \frac{\theta_i/\|\theta\|_{[d]}}{\mu_1(\lambda_i)}(A_n-\lambda_i I_n)^{-1}\widetilde{X}
    \right\|_2\leq\|\theta\|_{[d]}^{-\frac{10^d}{10^d+1}}
    \right)^{\frac{4}{9}}d\theta.
\end{equation}

Then we wish to apply Lemma \ref{lemma6.61} to the probability inside the integration, which provides an estimate independent of $(\theta_1,\cdots,\theta_d)$ with $\|\theta\|_{[d]}=1$ except for the subscript $J$, where $J$ depends on the relative value of $\theta_i$.

We let $\sum_1^{[d]},\cdots,
\sum_d^{[d]}$ denote a partition of the hyperplane $\|\theta\|_{[d]}=1$ (this partition also depends on $A_n$) such that Lemma \ref{lemma6.61} gives $J=i$ when $\theta/\|\theta\|_{[d]}\in\sum_i^{[d]}$. Then, for any $\eta>0$, we have
\begin{equation}\label{line13234}
\eqref{integralforsimplified}_{\mathcal{A}=[d]}\lesssim \sum_{\ell=1}^d \int_{s=1}^{\|\delta\|_{[d]}}  s^{d-1+\eta}
    \sup_{\theta\in \sum_\ell^{[d]}} \mathbb{P}_{\widetilde{X}}\left(\left\|
    \sum_{i=1}^d\frac{\theta_i}{\mu_1(\lambda_i)}(A_n-\lambda_i I_n)^{-1}\widetilde{X}\right\|_2\leq s^{-\frac{10^d}{10^d+1}}
   \right)^{\frac{4}{9}}ds
\end{equation}
where $\|\delta\|_{[d]}:=(\delta_1\cdots \delta_d)^{1/d}$. The factor $s^{d-1+\eta}$ comes from the change in coordinates, as implied by the following lemma:

\begin{lemma}\label{lemma8.2s}
    Let $A_d(s)$ denote the volume of the following region in $\mathbb{R}^d$:
$$\left\{(\theta_1,\cdots,\theta_d)\in\mathbb{R}^d:|\prod_{i=1}^d\theta_i|\leq s^d, |\theta_i|\geq 1\text{ for each }i\right\}.$$ Then there exists a constant $C_d$ and a polynomial $p_d(x)$ such that for any $s>1$, $$0<\frac{d}{ds}A_d(s)\leq C_ds^{d-1}p_d(\log s).$$ 
\end{lemma}

Now we apply a further change in the coordinates to obtain 
\begin{equation}\label{line13234agfaa}
\eqref{integralforsimplified}_{\mathcal{A}=[d]}\lesssim \sum_{\ell=1}^d \int_{\frac{1}{\|\delta\|_{[d]}}}^{1}  s^{-\frac{(10^d+1)(d+\eta)}{10^d}-1}
    \sup_{\theta\in \sum_\ell^{[d]}} \mathbb{P}_{\widetilde{X}}\left(\left\|
    \sum_{i=1}^d\frac{\theta_i}{\mu_1(\lambda_i)}(A_n-\lambda_i I_n)^{-1}\widetilde{X}\right\|_2\leq s
   \right)^{\frac{4}{9}}ds.
\end{equation}
We can choose $\eta>0$ to be sufficiently small (depending on $d$) so that $(10^d+1)(d+\eta)/10^d+1-\frac{4}{9}\leq d+\frac{3}{4}$. Applying Lemma \ref{lemma6.61} and taking the decomposition in the same form as \eqref{decomposition2d}, replacing $\sqrt{\mu_1(\lambda_1)\mu_1(\lambda_2)}$ by $(\prod_{i=1}^d\mu_1(\lambda_i))^{1/d}$, we get
\begin{equation}\label{whatdoweget1?}
\eqref{integralforsimplified}_{\mathcal{A}=[d]}\lesssim\sum_{\ell=1}^d\sum_{k=2}^{cn}e^{-k}\left(\frac{(\prod_{i=1}^d\mu_1(\lambda_i))^{1/d}}{\mu_k(\lambda_\ell)}\right)^{d-1/4}+e^{-\Omega(n)}.
\end{equation} We have made very careful choices of parameters such that the power of each $\mu_1(\lambda_i)$ is less than 1. Similarly, applying Lemma \ref{lemma6.612} for any $\emptyset\neq \mathcal{A}\in 2^{[d]}$
we obtain 
\begin{equation}\label{whatdoweget2?}
\eqref{integralforsimplified}_{\mathcal{A}}\lesssim\sum_{\ell=1}^d\sum_{k=2}^{cn}e^{-k}\left(\frac{(\prod_{i=1}^{|\mathcal{A}|}\mu_1(\lambda_i))^{1/|\mathcal{A}|}}{\mu_k(\lambda_\ell)}\right)^{|\mathcal{A}|-1/4}+e^{-\Omega(n)}.
\end{equation} 
Now we can prove the following lemma for general $d$, generalizing Lemma \ref{lemma6.111}.

\begin{lemma}\label{910lemma6.111} Under theorem \ref{Theorem1.1}, for any $\delta_1,\cdots,\delta_d\geq e^{-cn}$, any $u\in\mathbb{R}^n$ with $\|u\|_2\leq d$, and any $p>1$, we have the following estimate:
\begin{equation}\label{lemma8.3firstsecond}
\begin{aligned}
&\mathbb{E}_{A_n}\sup_{r_1,\cdots,r_d}\mathbb{P}_X\left(
    \frac{|\langle (A_n-\lambda_i I_n)^{-1}X,X\rangle-r_i|}{\|(A_n-\lambda_i I_n)^{-1}\|_*}\leq\delta_i
    ,\langle X,u\rangle\geq s,\frac{\prod_{i=1}^d\mu_1(\lambda_i)}{n^{d/2}}\leq(\prod_{i=1}^d\delta_i)^{-p}
    \right)\\&\lesssim e^{-s}\prod_{i=1}^d\delta_i+e^{-\Omega(n)}
    \\&+e^{-s}(\prod_{i=1}^d\delta_i)\mathbb{E}_{A_n}\left[\left(\frac{\prod_{i=1}^d\mu_1(\lambda_i)}{n^{d/2}}\right)^{1-\frac{1}{8d}}
    \wedge\left\{ \frac{\prod_{i=1}^d\mu_1(\lambda_i)}{n^{d/2}}\leq (\prod_{i=1}^d\delta_i)^{-p}\right\}\right]^{\frac{8d-2}{8d-1}},\end{aligned}
\end{equation}
    where $c$ depends on the assumptions in Theorem, on $\kappa$ and $\Delta$.
\end{lemma}
 This lemma (with $u=0$, $s=0$ and $p>1$ sufficiently close to 1) combined with Proposition \ref{lemma8.11111} and Hanson-Wright imply the following initial estimate:
 \begin{lemma}\label{initial8}(Initial estimate) Under theorem \ref{Theorem1.1}, for any $\delta_1,\cdots,\delta_d>0$ 
     $$\mathbb{P}(\sigma_{min}(A_{n+1}-\lambda_i I_{n+1})\leq\delta_i n^{-1/2},i=1,\cdots,d)\lesssim (\prod_{i=1}^d\delta_i)^{\frac{1}{8d}}+e^{-cn}.$$
 \end{lemma}
By bootstrapping this estimate finitely many times, we can prove that 
\begin{lemma}\label{bootstrap8}(Bootstrap) In the setting of this section, for any $\delta_1,\cdots,\delta_d>0$ 
     $$\mathbb{P}(\sigma_{min}(A_{n+1}-\lambda_i I_{n+1})\leq\delta_i n^{-1/2},i=1,\cdots,d)\lesssim (\prod_{i=1}^d\delta_i)(\log(\prod_{i=1}^d\delta_i)^{-1})^{d/2}+e^{-cn}.$$
 \end{lemma}
 Finally, the additional log factor can be removed following the procedure in Section \ref{removalladfaga}. The details are given in the next subsection, and this completes the proof of Theorem \ref{Theorem1.1} for any $d\in\mathbb{N}_+$. 
 
\subsection{Proofs of intermediate steps}
Now we prove all the results claimed in Section \ref{section8.111}.

\begin{proof}[\proofname\ of Proposition \ref{lemma8.11111}]
We first prove that, for any $p>1$,
   \begin{equation}\label{8.1replaced} \mathbb{P}(\sigma_{min}(A_{n+1}-\lambda_i I_{n+1})\leq\delta_i n^{-1/2},i=1,\cdots,d\quad \mathcal{P}^{p,d})\lesssim \prod_{i=1}^d\delta_i+e^{-\Omega(n)},\end{equation}
    where we define 
    $$
\mathcal{P}^{p,d}:=\left\{ \prod_{i=1}^d\sigma_{min}(A^{(j)}_{n+1}-\lambda_i I_n)\leq(\prod_{i=1}^d\delta_i)^p n^{-d/2}\text{ for at least } \frac{1}{4}n \text{ values } j\in[n+1]\right\}.
$$ Once this is proven, we can deduce Proposition \ref{lemma8.11111} in exactly the same way as the proof of Proposition \ref{finalfuckpropositionga} is worked out: we use the no-gaps delocalization result (Theorem \ref{veragaghagag}) to deduce that \eqref{fuckyousga} holds for all $i=1,\cdots,d$ and for at least $\frac{1}{2}n$ choices of $j\in[n+1]$.

Now we prove \eqref{8.1replaced}. We fix some $j\in[n+1]$ such that the event in $\mathcal{P}^{p,d}$ holds for $j$ and moreover, for all the $d$ eigenvectors associated with the least singular values of $A_{n+1}-\lambda_j I_{n+1}$ ($j=1,\cdots,d$), their $j$-th coordinates have an absolute value of at least $cn^{-1/2}$ for some fixed $c$. The latter property can be satisfied with probability $1-e^{-\Omega(n)}$ thanks to Theorem \ref{veragaghagag}. We take a dyadic partition of $[(\prod_{i=1}^d\delta_i)^p,1]$ and consider each subevent where $\sigma_{min}(A_{n+1}^{(j)}-\lambda_i I_n)\leq x_in^{-1/2}$ and $\prod_{i=1}^d x_i\leq 2^p(\prod_{i=1}^d\delta_i)^p$. If there exists some $k\in[d]$ with $x_k>\delta_k$, then we first use our induction hypothesis and Lemma \eqref{singularvalueproduct} to deduce that, for any $\eta>0$, 
\begin{equation}\label{productequations}\mathbb{P}(\prod_{i\in[d]\setminus\{k\}}\sigma_{min}(A_{n+1}^{(j)}-\lambda_i I_n)\leq \prod_{i\in[d]\setminus\{k\}}x_i)\lesssim (\prod_{i\in[d]\setminus\{k\}}x_i)^{1-\eta}+e^{-\Omega(n)}\end{equation} 
Then we use Fact \eqref{fact7.3} to show that we have $$\mathbb{P}(\sigma_{min}(A_{n+1}-\lambda_k I_{n+1})\leq x_kn^{-1/2},\sigma_{min}(A_{n+1}^{(j)}-\lambda_k I_n)\leq x_kn^{-1/2}\mid\mathcal{HA}_{[d]\setminus\{k\}})\lesssim x_k,$$ where we denote by $\mathcal{HA}_{[d]\setminus\{k\}}$ the event 
$$
\mathcal{HA}_{[d]\setminus\{k\}}:=\left\{\prod_{i\in[d]\setminus\{k\}}\sigma_{min}(A_{n+1}^{(j)}-\lambda_i I_n)\leq \prod_{i\in[d]\setminus\{k\}}x_i\right\}.
$$
Combining the last two estimates completes the proof, and our dyadic partition leads to an entropy cost of $\log((\prod_{i=1}^d\delta_i)^{-1})^{C_d}$. If for all $k$ we have $x_k<\delta_k$ then for some $j\in[d]$ we have $x_j\leq\delta_j(\prod_{i=1}^d\delta_i)^{(p-1)/d}$. We choose some $k\neq j$ and apply the induction hypothesis to $[d]\setminus\{k\}$ and apply Fact \eqref{fact7.3} to $\{k\}$. It is straightforward to check that the bound is good enough.
\end{proof}

\begin{proof}[\proofname\ of Lemma \ref{lemma8.2s}] Evidently we have $A_1(s)=2(s-1)$ and we have already computed $A_2(s)=4(s^2\ln s^2-s^2+1)$. We let $F_d(s^d):=A_d(s)$ and note the following iterative formula:
$$
A_d(s)\equiv F_d(s^d)=2\int_1^{s^{d}}F_{d-1}\left(\frac{s^{d}}{y}\right)dy,
$$ where we first integrate over $y=|\theta_1|$, and on this cross-section we have $\prod_{i=2}^d|\theta_i|\leq \frac{s^d}{y}$. Then the claim follows from direct, but tedious computations. We can check via induction that for each $d$, we can find a constant $C_d$ satisfying $$F_d(s^d)=C_ds^d\ln^{d-1}(s^d)+s^d(\text{ lower powers of }\ln(s^d))+\text{ const}.$$
\end{proof}

\begin{proof}[\proofname\ of Lemma \ref{910lemma6.111}]
    We first define the following event, for any $p>1$:
    $$\mathcal{H}:=\left\{\langle X,u\rangle\geq s,\frac{\prod_{i=1}^d\mu_1(\lambda_i)}{n^{d/2}}\leq(\prod_{i=1}^d\delta_i)^{-p},\frac{\prod_{i\in\mathcal{A}}\mu_1(\lambda_i)}{n^{|\mathcal{A}|/2}}\leq(\prod_{i=1}^d\delta_i)^{-p},\mathcal{A}\in 2^{[d]}\right\}.$$
    Then to prove Lemma \ref{910lemma6.111}, we suffice to prove the lemma if we replace the first line of \eqref{lemma8.3firstsecond} with 
$$
\mathbb{E}_{A_n}\sup_{r_1,\cdots,r_d}\mathbb{P}_X\left( \frac{|\langle (A_n-\lambda_i I_n)^{-1}X,X\rangle-r_i|}{\|A_n-\lambda_i I_n)^{-1}\|_*} \leq\delta_i,i=1,\cdots,d, \quad \wedge\mathcal{H}.
\right)
$$ This is because for each $\emptyset\neq \mathcal{A}\in 2^{[d]}$ such that $\mathcal{A}\neq \{1,\cdots,d\}$, we have by the induction hypothesis and Lemma \ref{singularvalueproduct} that $\mathbb{P}\left(\frac{\prod_{i\in\mathcal{A}}\mu_1(\lambda_i)}{n^{|\mathcal{A}|/2}}\geq(\prod_{i=1}^d\delta_i)^{-p}\right)\lesssim \prod_{i=1}^d\delta_i+e^{-\Omega(n)}$.

Now we can apply Theorem \ref{theorem3.1} which yields that it suffices to bound 
$$
\sum_{A\in 2^{[d]}}\mathbb{E}_{A_n}\left[\frac{\prod_{i\in\mathcal{A}}\mu_1(\lambda_i)}{\prod_{i\in\mathcal{A}}\|(A_n-\lambda_i I_n)\|_*}\cdot\int_{\mathcal{D}_{\mathcal{A}}} I(\theta)^{1/2}d\theta\quad\wedge\mathcal{H}\right].
$$ Now we plug in estimates \eqref{whatdoweget1?} and \eqref{whatdoweget2?}, and it suffices to bound 
$$
\sum_{A\in 2^{[d]}}\sum_{\ell=1}^d\mathbb{E}_{A_n}
\left[ \frac{\prod_{i\in\mathcal{A}}\mu_1(\lambda_i)}{\prod_{i\in\mathcal{A}}\|(A_n-\lambda_i I_n)\|_*}\cdot\sum_{k=2}^{cn}e^{-k}\left(\frac{(\prod_{i=1}^{|\mathcal{A}|}\mu_1(\lambda_i))^{1/|\mathcal{A}|}}{\mu_k(\lambda_\ell)}\right)^{|\mathcal{A}|-1/4}
\wedge\mathcal{H}\right].
$$
Now we apply Hölder's inequality, combined with Lemma \ref{lemma4.11} and Corollary \ref{lemma4.22} which shows $\mu_1(\lambda_i)/\|A_n-\lambda_i I_n\|_*$ and $\sqrt{n}/k\mu_k(\lambda_\ell)$ have all finite $p$-th moments, to deduce that it suffices to bound the following expression (omitting computation details)
$$
\sum_{\mathcal{A}\in 2^{[d]}}\mathbb{E}_{A_n}
\left[\left(\frac{\prod_{i=1}^{|\mathcal{A}|}\mu_1(\lambda_i)}{n^{|\mathcal{A}|/2}}\right)^{1-\frac{1}{8|\mathcal{A}|}}\wedge\mathcal{H}\right]^{\frac{8|\mathcal{A}|-2}{8|\mathcal{A}|-1}}
.
$$
The case where $\mathcal{A}=\{1,\cdots,d\}$ is as given in the statement of Lemma \ref{910lemma6.111}. It remains to show that for any $\mathcal{A}\neq\{1,\cdots,d\}$, the given expectation is $O(1)$. For this purpose we compute
\begin{equation}\label{whatline1432}
\int_0^{(\prod_{i\in\mathcal{A}}\delta_i)^{-\frac{8|\mathcal{A}|-1}{8|\mathcal{A}|}p}} \mathbb{P}\left(\prod_{i\in\mathcal{A}}\sigma_{min}(A_n-\lambda_iI_n)\leq x^{-\frac{8|\mathcal{A}|}{8|\mathcal{A}|-1}}n^{-|\mathcal{A}|/2}\right)dx.
\end{equation}
By the induction hypothesis and an application of Lemma \ref{singularvalueproduct}, we obtain for any $\eta>0$ and $x>1$,
$$\mathbb{P}\left(\prod_{i\in\mathcal{A}}\sigma_{min}(A_n-\lambda_iI_n)\leq x^{-\frac{8|\mathcal{A}|}{8|\mathcal{A}|-1}}n^{-|\mathcal{A}|/2}\right)\leq x^{-\frac{8|\mathcal{A}|}{8|\mathcal{A}|-1}+\eta}+e^{-\Omega(n)},$$ and hence that $\eqref{whatline1432}=O(1)$ by choosing $\eta>0$ small. This completes the proof.
\end{proof}

\begin{proof}[\proofname\ of Lemma \ref{initial8}] We first apply Hanson-Wright (Theorem \ref{hansonwright8.1}) to deduce that for each $i\in\{1,2,\cdots,d\}$, where we also use $\|M\|_*\geq \|M\|_{HS}$ for any square matrix $M$,
$$\mathbb{P}_X\left(\|(A_n-\lambda_i I_n)^{-1}X\|_2\geq C'(\log(\prod_{i=1}^d\delta_i)^{-1})^{1/2}\cdot \|(A_n-\lambda_i I_n)^{-1}\|_{*}
\right)\leq\prod_{i=1}^d\delta_i$$
Then the result follows from applying Proposition \ref{lemma8.11111} and Lemma \ref{910lemma6.111} where we take $u=0$, $s=0$, take $p>1$ sufficiently close to 1, and take $\delta_i$ to be $\delta_i\cdot(\log(\prod_{i=1}^d\delta_i)^{-1})^{1/2}$.
    
\end{proof}

\begin{proof}[\proofname\ of Lemma \ref{bootstrap8}] We set up the following bootstrap argument: suppose that for some $\tau\in(0,1)$ we have proven the following for all $\delta_1,\cdots,\delta_d>0$,
     $$\mathbb{P}(\sigma_{min}(A_{n+1}-\lambda_i I_{n+1})\leq\delta_i n^{-1/2},i=1,\cdots,d)\lesssim \left(\prod_{i=1}^d\delta_i\right)^\tau+e^{-\Omega(n)},$$ then this implies that 
      $$\begin{aligned}&\mathbb{P}(\sigma_{min}(A_{n+1}-\lambda_i I_{n+1})\leq\delta_i n^{-1/2},i=1,\cdots,d)\\&\quad\lesssim \left(\prod_{i=1}^d\delta_i\right)^{\tau+\frac{1}{9d}}(\log(\prod_{i=1}^d\delta_i)^{-1})^{d/2}+e^{-\Omega(n)}.\end{aligned}$$ 
     We prove this claim via the following steps. By our assumption and an application of Lemma \ref{singularvalueproduct}, we get for any $\omega>0$ and $x>1$,
$$\mathbb{P}\left(\prod_{i=1}^d\sigma_{min}(A_n-\lambda_iI_n)\leq x^{-\frac{8d}{8d-1}}n^{-d/2}\right)\leq x^{-\frac{8d}{8d-1}\tau(1-\omega)}+e^{-\Omega(n)},$$ so that

$$\begin{aligned}
&\int_0^{(\prod_{i=1}^d\delta_i)^{-\frac{8d-1}{8d}p}} \mathbb{P}\left(\prod_{i=1}^d\sigma_{min}(A_n-\lambda_iI_n)\leq x^{-\frac{8d}{8d-1}}n^{-d/2}\right)dx\\&\quad\quad\lesssim 1+(\prod_{i=1}^d\delta_i)^{p\tau(1-\omega)-\frac{8d-1}{8d}p}
,\end{aligned}$$
and 
$\mathbb{E}_{A_n}\left[ (\frac{\prod_{i=1}^d\mu_1(\lambda_i)}{n^{d/2}})^{1-\frac{1}{8d}}
\right]^\frac{8d-2}{8d-1}\lesssim \max \left\{1, (\prod_{i=1}^d\delta_i)^{\frac{8d-2}{8d-1}p\tau(1-\omega)-\frac{4d-1}{4d}p}\right\}.$

 Now we apply Hanson-Wright (Theorem \ref{hansonwright8.1}) as in the proof of the previous lemma, then apply Proposition \ref{lemma8.11111} and Lemma \ref{910lemma6.111} where we take $u=0$, $s=0$ and take $\delta_i$ to be $\delta_i\log((\prod_{j=1}^d\delta_j)^{-1})^{1/2}.$ This gives us
 $$\begin{aligned}&\mathbb{P}(\sigma_{min}(A_{n+1}-\lambda_i I_{n+1})\leq\delta_i n^{-1/2},i=1,\cdots,d)\\&\lesssim \left(\prod_{i=1}^d\delta_i\right)^{\min\{1,\frac{8d-2}{8d-1}p\tau(1-\omega)-\frac{4d-1}{4d}p+1\}}(\log(\prod_{i=1}^d\delta_i)^{-1})^{d/2}+e^{-\Omega(n)}.\end{aligned}$$ 
Now we take $\omega>0$ sufficiently small, take $p>1$ that satisfies $p(1-\omega)=\frac{8d-1}{8d-2}$, and satisfies $-\frac{4d-1}{4d}p+1\geq \frac{1}{9d}$. This justifies the claim.

Finally, we apply the bootstrap procedure finitely many times, initiated at $\frac{1}{8d}$ and proceed with step size $\frac{1}{10d}$. The procedure stops after less than $10d$ iterations.
 \end{proof}

\begin{proof}[\proofname\ of Theorem \ref{Theorem1.1} for general $d$] The arguments follow those given in Section \ref{removalladfaga} and we will only give a sketch. Let $\mathcal{E}:=\{\prod_{i=1}^d \sigma_{min}(A_n-\lambda_i I_n)\geq(\prod_i\delta_i)^pn^{-p/2}\}$. Consider the following decomposition
\begin{equation}\begin{aligned}
    &\mathbb{P}^{\mathcal{E}}(Q(A,X,\lambda_i)\leq C\delta_i,i=1,\cdots,d)\\&\leq \sum_{\mathcal{A}\in 2^{[d]}}\mathbb{P}^{\mathcal{E}}(Q(A,X,\lambda_i)\leq C\delta_i,i\in [d];\frac{\|(A_n-\lambda_i I_n)^{-1}X\|_2}{\|(A_n-\lambda_i I_n)^{-1}\|_*}\geq 2,i\in\mathcal{A}).
\end{aligned}\end{equation}
    The case $\mathcal{A}=\emptyset$ can be verified directly. We just apply Lemma \ref{910lemma6.111} (taking $s=0$, $u=0$) and Lemma \ref{bootstrap8} to show that this term is bounded by $\delta_1\cdots\delta_d+e^{-\Omega(n)}$. 

    For a general $\mathcal{A}\neq\emptyset$, we consider a direct product of $|\mathcal{A}|$ dyadic partitions
    $$2^{j_i}\leq \|(A_n-\lambda_i I_n)^{-1}X\|_2/\|(A_n-\lambda_i I_n)^{-1}\|_*\leq 2^{j_{i+1}},i\in\mathcal{A}$$ where we terminate each dyadic partition at $\log n$ (i.e. set $j_i\leq\log n$ for each $i\in\mathcal{A}$) thanks to an application of Hanson-Wright (Theorem \ref{hansonwright8.1}).

As in Lemma \ref{lemma7.6alreadytire}, we can prove the following: for any $(t_1,\cdots,t_d)$ satisfying $t_i=1$ if $i\notin\mathcal{A}$ and $t_i\geq 2$ if $i\in\mathcal{A}$,
\begin{equation}\begin{aligned}\label{decomposition1481}
    &\mathbb{P}_X^{\mathcal{E}}\left(Q_*(A,X,\lambda_i)\leq 2Ct_i\delta_i,i=1,\cdots,d;\quad \frac{\|(A_n-\lambda_i I_n)^{-1}X\|_2}{\|(A_n-\lambda_i I_n)^{-1}\|_*}\geq t_i,\quad i\in\mathcal{A}
    \right)\\& \leq 2^d\sum_{k_j=1,j\in\mathcal{A}}^n\mathbb{P}_X^{\mathcal{E}}\left(Q_*(A,X,\lambda_i)\leq 2Ct_i\delta_i,i\in[d];\langle X,v_{k_j}(\lambda_j)\rangle\geq t_j\log(1+k_j),j\in\mathcal{A}\right),
\end{aligned}\end{equation}
    where each $v_{k_j}(\lambda_j)$, $j\in\mathcal{A}$ is chosen from an orthonormal basis of eigenvectors of $A_n$ and these vectors are not necessarily distinct. 
    
    To compute the probability in the second line of \eqref{decomposition1481}, we use Lemmas \ref{910lemma6.111} and \ref{bootstrap8} to obtain a bound (where we take the vector $u:=\sum_{j\in\mathcal{A}} v_{k_j}(\lambda_j)$ in Lemma \ref{910lemma6.111})
$$\begin{aligned}&\mathbb{P}_X^{\mathcal{E}}\left(Q_*(A,X,\lambda_i)\leq 2Ct_i\delta_i,i\in[d];\langle X,v_{k_j}(\lambda_j)\rangle\geq t_j\log(1+k_j),j\in\mathcal{A}\right)\\&\lesssim \delta_1\cdots\delta_d\prod_{j\in\mathcal{A}}t_j\cdot e^{-\sum_{j\in\mathcal{A}} t_j\log(1+k_j)}\end{aligned}.$$ This step completes the proof by noticing that, for each $\mathcal{A}\in 2^{[d]}$, $$\sum_{j_i=1,i\in\mathcal{A}}^\infty \sum_{k_i=1,i\in\mathcal{A}}^\infty  2^{\sum_{i\in\mathcal{A}}j_i}\prod_{i\in\mathcal{A}}(k_i+1)^{-2^{j_i}}=\left(\sum_{j=1}^\infty\sum_{k=1}^
\infty 2^j(k+1)^{-2^j}\right)^{|\mathcal{A}|}=O(1).$$ 

\end{proof}

\section{Mesoscopically separated locations}
\label{section9section9}
In this section, we assume that the locations $\lambda_1,\cdots,\lambda_d$ are separated from one another only by distance $\Delta n^{-1/2+\sigma}$, for some fixed values of $\Delta>0$ and $\sigma>0$.

We prove Theorem \ref{Theorem1.2}, which states that all the results in Theorem \ref{Theorem1.1} hold true but with a larger error $e^{-\Omega(n^{\sigma/2})}$ rather than $e^{-\Omega(n)}$.

We first state two useful corollaries of the local arcsine law. The first is an adaptation of Corollary \ref{corollary4.999} to scale $k\sim n^\sigma$.

\begin{corollary}\label{corollary11.999}
    In the setting of Corollary \ref{corollary4.88}, for any $\lambda_i\in[(-(2-\kappa)\sqrt{n},(2-\kappa)\sqrt{n}]$, $i=1,\cdots,d$, any $c_0>0$ and any $k\geq c_0n^{\sigma}$,
    \begin{equation}
\mathbb{P}\left(\frac{\mu_k(\lambda_i)}{\mu_1(\lambda_i)}\geq 10^{-5}d^{-1}\right)\leq e^{-\Omega(n^{\sigma/2})}.
    \end{equation}
\end{corollary}

\begin{proof}
    By Corollary \ref{corollary3.6chap9}, with probability $1-e^{-\Omega(n^{\frac{\sigma}{2}})}$, $\mu_k(\lambda_i)\in [\frac{C_1\sqrt{n}}{k},\frac{C_2\sqrt{n}}{k}]$ for all $k\geq c_0n^\sigma$ any any $\sigma>0$. Then it suffices to consider the event  $\mu_1(\lambda_i)\leq C\sqrt{n}/{n^\sigma}$ for any $C>0$, but this event implies that on the interval of length $2C^{-1}n^{\sigma-\frac{1}{2}}$ centered at $\lambda_i$, $A_n$ has no eigenvalues. By Theorem \ref{theorem4.3}, the possibility of this event is less than $e^{-\Omega(n^{\sigma/2})}$.
\end{proof}

Recall the notation $\mu_{c_j(i;k)}(\lambda_j)$ introduced in Definition \ref{definition4.2}, where $c_j(i;k)$ is the subscript of the singular vector of $(A_n-\lambda_j I_n)^{-1}$ sharing the same singular vector with the $k$-th largest singular value of $(A_n-\lambda_i I_n)^{-1}$.

We prove the following further estimate: 
\begin{lemma} Let $\Delta>0,\kappa>0$ and $\sigma\in(0,1)$ be fixed. Fix any $\lambda_1,\cdots,\lambda_d\in[-(2-\kappa)\sqrt{n},(2-\kappa)\sqrt{n}]$ satisfying $|\lambda_i-\lambda_j|\geq\Delta n^{\sigma-\frac{1}{2}}$ for any $i\neq j$. Then \begin{enumerate}
\item We can find some constant $\Delta_0>0$ and some $C_0>0$ depending on $\Delta$ such that, for any $k\leq \Delta_0n^\sigma$, we have
    \begin{equation}
        \mathbb{P}(c_j(i;k)\leq C_0n^\sigma)\leq \exp(-\Omega(n^{\sigma/2})).
    \end{equation} \item We can further find some $\Delta_0>0$ which is sufficiently small such that for any $i\neq j\in\{1,2,\cdots,d\}$ and, we have on an event with probability $1-e^{-\Omega(n^{\sigma/2})},$
\begin{equation}\label{proofnameagaga}
    \frac{\mu_{c_i(j;1)}(\lambda_i)}{\mu_1(\lambda_i)}\leq 10^{-5}d^{-1}\frac{\mu_1(\lambda_j)}{\mu_1(\lambda_j)}
\end{equation}
\item Suppose we are given $(\theta_1,\cdots,\theta_d)\in\mathbb{R}^d$ and some $J\in\{1,\cdots,d\}$ that satisfy $|\frac{\theta_J}{\mu_1(\lambda_J)}|\geq |\frac{\theta_i}{\mu_1(\lambda_i)}|$ for any $i\neq J$. Then we can find some (small) $\Delta_0>0$ such that on an event with probability $1-e^{-\Omega(n^{\sigma/2})},$ for any $1\leq k\leq \Delta_0n^{\sigma/2}$, we have 
\begin{equation}
   |\theta_i| \frac{\mu_{c_i(J;k)}(\lambda_i)}{\mu_1(\lambda_i)}\leq 10^{-2}d^{-1}|\theta_J|\frac{\mu_k(\lambda_J)}{\mu_1(\lambda_J)}
\end{equation}
\end{enumerate}
\end{lemma}
\begin{proof}
    Suppose that for $c_j(i;k)\leq C_0n^\sigma$, then there are at most $k+C_0n^\sigma\leq(\Delta_0+C_0)n^\sigma$ number of eigenvalues of $A_n$ in the interval $[\lambda_i,\lambda_j]$ (or $[\lambda_j,\lambda_i]$). However, if $\Delta_0+C_0$ is chosen to be sufficiently small, this event has probability $e^{-\Omega(n^{\sigma/2})}$ according to Theorem \ref{theorem4.3}. 

    For the second claim, we combine the first claim (we keep $C_0$ fixed and set $\Delta_0$ even smaller) with Corollary \ref{corollary11.999} to deduce that \eqref{proofnameagaga} holds.

For the third claim, we combine our initial assumption on the subscript $J$, the estimate in Corollary \ref{corollary3.6chap9} giving a range $\mu_k\in[C_1\sqrt{n}/k,C_2\sqrt{n}/k]$, and our first claim stating that $c_i(j,k)$ are large enough. Fine tuning the value of $\Delta_0$ to be small completes the proof.   
\end{proof}

Having proven these estimates, we can derive a straightforward generalization of Lemma \ref{lemma6.61} and \ref{lemma6.612} to $\lambda_i$'s at a mesoscopic distance. We outline them here but we omit the proof, as the proofs are essentially identical.

\begin{lemma}\label{9--lemma6.61}
 Assume that $A_n$ satisfies the assumptions in Theorem \ref{Theorem1.2} and assume that $X_2\sim\operatorname{Col}_n(\zeta_2)$ and $\zeta_2\in\Gamma_2(G,K,\sigma_0)$. Fix any $\kappa>0$ and $\Delta>0$ to be sufficiently small, and some $\sigma\in(0,1)$. Then there exists an event $\mathcal{E}_1$ with $\mathbb{P}(\mathcal{E}_1)\geq 1-\exp(-\Omega(n^{\sigma/2}))$ such that for any $A_n\in\mathcal{E}_1$, the following holds:

   Fix any given $(\theta_1,\cdots,\theta_d)\in\mathbb{R}^d$ with $\prod_{i=1}^d|\theta_i|=1$ , and fix any given
   $\lambda_1,\cdots,\lambda_d\in[-(2-\kappa)\sqrt{n},(2-\kappa)\sqrt{n}]$ with $|\lambda_i-\lambda_j|\geq \Delta n^{-\frac{1}{2}+\sigma}$ whenever $i\neq j$.

   Then there exists some $J\in\{1,\cdots,d\}$ (depending on $A_n$ and $\theta$) such that, for any $$s\in\left(0, \frac{C\cdot \mu_k(\lambda_J)}{{\left(\prod_{i=1}^d\mu_1(\lambda_i)\right)^{1/d}}}\right)$$ and any $k\leq c_0n^\sigma$, we have
    \begin{equation}
\mathbb{P}_{\widetilde{X}}\left(\left\|\sum_{i=1}^d \frac{\theta_i}{\mu_1(\lambda_i)}(A_n-\lambda_i I_n)^{-1}\widetilde{X}\right\|_2\leq s
\right) \lesssim s e^{-k},
    \end{equation} where $c_0$, $C$ are universal constant depending only on the various parameters $G,B$,$\sigma_0$,$K$.
\end{lemma}

\begin{lemma}\label{9--lemma6.612}
 Assume that $A_n$ satisfies the assumptions in Theorem \ref{Theorem1.2} and that $X_2\sim\operatorname{Col}_n(\zeta_2)$ and $\zeta_2\in\Gamma_2(G,K,\sigma_0)$. Fix any $\kappa>0$ and $\Delta>0$ to be sufficiently small, and some $\sigma\in(0,1)$. Then there exists an event $\mathcal{E}_2$ with $\mathbb{P}(\mathcal{E}_2)\geq 1-\exp(-\Omega(n^{\sigma/2}))$ such that for any $A_n\in\mathcal{E}_2$, the following holds:

   Fix any given $(\theta_1,\cdots,\theta_d)\in\mathbb{R}^d$, assume that we can find some $\mathcal{A}\subset\{1,\cdots,d\}$ satisfying: $|\mathcal{A}|=\ell$ for some $1\leq \ell\leq d-1$, that $\prod_{i\in\mathcal{A}}|\theta_i|=1$ and for any $j\notin\mathcal{A}$, $|\theta_j|<1$. Fix any given
   $\lambda_1,\cdots,\lambda_d\in[-(2-\kappa)\sqrt{n},(2-\kappa)\sqrt{n}]$ satisfying $|\lambda_i-\lambda_j|\geq \Delta n^{-\frac{1}{2}+\sigma}$ whenever $i\neq j$. 
   
   Then there exists some $J\in\{1,\cdots,d\}$ (depending on $A_n$ and $\theta$) such that, for any $$s\in\left(0, \frac{C\cdot \mu_k(\lambda_J)}{{\left(\prod_{i\in\mathcal{A}}\mu_1(\lambda_i)\right)^{1/\ell}}}\right)$$ and for any $k\leq c_0n^\sigma$, we have
    \begin{equation}
\mathbb{P}_{\widetilde{X}}\left(\left\|\sum_{i=1}^d \frac{\theta_i}{\mu_1(\lambda_i)}(A_n-\lambda_i I_n)^{-1}\widetilde{X}\right\|_2\leq s
\right) \lesssim s e^{-k},
    \end{equation} where $c_0$ and $C$ are universal constants that depend only on the various parameters $G,B$,$\sigma_0$,$K$.
\end{lemma}

Under the assumption of Theorem \ref{Theorem1.2}, the statement of Proposition \ref{lemma8.11111} holds almost without change. We restate it here:

\begin{Proposition}
   Under the assumption of Theorem \ref{Theorem1.2}, for any $p>1$ we have
     $$\begin{aligned}
           &\mathbb{P}(\sigma_{min}(A_{n+1}-\lambda_i I_{n+1})\leq\delta_i n^{-1/2},i=1,2,\cdots,d)\lesssim \delta_1\delta_2\cdots\delta_d+e^{-\Omega(n^{\sigma/2})}\\&+\mathbb{E}_{A_n}\sup_{r_1,\cdots,r_d}\mathbb{P}_X\left(
    \frac{|\langle (A_n-\lambda_i I_n)^{-1}X,X\rangle-r_i|}{\|(A_n-\lambda_i I_n)^{-1}X\|_2}\leq\delta_i,\frac{\prod_{i=1}^d\mu_1(\lambda_i)}{n^{d/2}}\leq(\prod_{i=1}^n\delta_i)^{-p}
    \right).\end{aligned}$$
\end{Proposition}
The only difference here, compared to Proposition \ref{lemma8.11111} is that the $e^{-\Omega(n)}$ error is replaced by $e^{-\Omega(n^{\sigma/2})}$. This change comes from the induction hypothesis: in the induction hypothesis here, the error is 
$e^{-\Omega(n^{\sigma/2})}$.
We can then easily prove the following analog of Lemma \ref{910lemma6.111}:

\begin{lemma}  Under the assumption of Theorem \ref{Theorem1.2}, for any $\delta_1,\cdots,\delta_d\geq e^{-cn^{\sigma}}$, any $u\in\mathbb{R}^n$ with $\|u\|_2\leq d$, and any $p>1$, we have the following estimate:
\begin{equation}
\begin{aligned}
&\mathbb{E}_{A_n}\sup_{r_1,\cdots,r_d}\mathbb{P}_X\left(
    \frac{|\langle (A_n-\lambda_i I_n)^{-1}X,X\rangle-r_i|}{\|(A_n-\lambda_i I_n)^{-1}\|_*}\leq\delta_i
    ,\langle X,u\rangle\geq s,\frac{\prod_{i=1}^d\mu_1(\lambda_i)}{n^{d/2}}\leq(\prod_{i=1}^d\delta_i)^{-p}
    \right)\\&\lesssim e^{-s}\prod_{i=1}^d\delta_i+e^{-\Omega(n^{\sigma/2})}
    \\&+e^{-s}(\prod_{i=1}^d\delta_i)\mathbb{E}_{A_n}\left[\left(\frac{\prod_{i=1}^d\mu_1(\lambda_i)}{n^{d/2}}\right)^{1-\frac{1}{8d}}
    \wedge\left\{ \frac{\prod_{i=1}^d\mu_1(\lambda_i)}{n^{d/2}}\leq (\prod_{i=1}^d\delta_i)^{-p}\right\}\right]^{\frac{8d-2}{8d-1}},\end{aligned}
\end{equation}
    where $c$ depends on the assumptions in Theorem, on $\kappa$ and $\Delta$.
\end{lemma}
We will not give the proof because it essentially follows the proof of Lemma \ref{910lemma6.111}. The only place where a change is made is where we take the following decomposition
\begin{equation}\label{decompositiongaag}[\prod_{i=1}^d\delta_i,1]=\left[\prod_{i=1}^d\delta_i,\frac{\mu_{c_0 n^\sigma}(\lambda_J)}{(\prod_{i=1}^d\mu_1(\lambda_i))^{1/d}}\right]\cup \bigcup_{k=2}^{c_0 n^\sigma}\left[\frac{\mu_{k}(\lambda_J)}{(\prod_{i=1}^d\mu_1(\lambda_i))^{1/d}},\frac{\mu_{k-1}(\lambda_J)}{(\prod_{i=1}^d\mu_1(\lambda_i))^{1/d}}\right]\end{equation} when applying Lemma \ref{9--lemma6.61} and \ref{9--lemma6.612}. The assumption that $\prod_{i=1}^d\delta_i\geq e^{-cn^{\sigma}}$ for an appropriate $c>0$ ensures that the integral over the first interval in the decomposition \eqref{decompositiongaag} is finite.

Now we can prove Theorem 1.2. 

\begin{proof}[\proofname\ of Theorem \ref{Theorem1.2}] The proof here is identical to the proof of Theorem \ref{Theorem1.1}. We can derive the corresponding versions of Lemma \ref{initial8} and \ref{bootstrap8} where the only difference is that we replace all $e^{-\Omega(n)}$ terms by $e^{-\Omega(n^{\sigma/2})}$. Then we remove the log factor in exactly the same way as the proof of Theorem \ref{Theorem1.1}. The details are omitted.
\end{proof}

\section*{Acknowledgments}
The author thanks Professor Julian Sahasrabudhe for some related discussions.

\printbibliography

\end{document}